%% file: main.tex
\DeclareSymbolFont{AMSb}{U}{msb}{m}{n}
\documentclass[11pt,a4paper,twoside,leqno,noamsfonts]{amsart}
           \makeatletter
           \newcommand{\mylabel}[2]{#2\def\@currentlabel{#2}\label{#1}}
           \renewcommand\@biblabel[1]{#1.}
           \makeatother
\usepackage[english]{babel}
\usepackage[dvipsnames]{xcolor}
\usepackage{graphicx,pifont}
\usepackage[utopia]{mathdesign}
\usepackage[a4paper,top=2.8cm,bottom=2.8cm,left=2.7cm,
           right=2.7cm,bindingoffset=5mm,marginparwidth=60pt]{geometry}
\usepackage[utf8]{inputenc}
\usepackage{braket,caption,comment,mathtools,stmaryrd}
\usepackage[usestackEOL]{stackengine}
\usepackage{multirow,booktabs,microtype,relsize}
\usepackage[colorlinks,bookmarks]{hyperref} % ,backref=page
      \hypersetup{colorlinks,%
            citecolor=britishracinggreen,%
            filecolor=black,%
            linkcolor=cobalt,%
            urlcolor=bole}
      \numberwithin{equation}{section}
\input{planepartition.tex}
\input{macros.tex}

\title[Higher rank K-theoretic DT theory of points]{Higher rank K-theoretic Donaldson--Thomas Theory of points}

\author{Nadir Fasola}
\address{SISSA Trieste, Via Bonomea 265, 34136 Trieste, Italy\newline\indent Institute for Geometry and Physics, Via Beirut 4, 34100 Trieste, Italy\newline\indent Istituto Nazionale di Fisica Nucleare, Sezione di Trieste, Via Valerio 2, 34127 Trieste, Italy}
\email{nfasola@sissa.it}
\author{Sergej Monavari}
\address{Mathematical Institute, Utrecht University, P.O.~Box 80010 3508 TA Utrecht, The Netherlands}
\email{s.monavari@uu.nl}
\author{Andrea T. Ricolfi}
\address{SISSA Trieste, Via Bonomea 265, 34136 Trieste, Italy\newline\indent Institute for Geometry and Physics, via Beirut 4, 34100 Trieste, Italy}
\email{aricolfi@sissa.it}

\begin{document}
\maketitle

\begin{abstract}
We exploit the critical structure on the Quot scheme $\Quot_{\BA^3}(\OO^{\oplus r},n)$, in particular the associated symmetric obstruction theory, in order to study rank $r$ \emph{K-theoretic} Donaldson--Thomas invariants of the local Calabi--Yau $3$-fold $\BA^3$. We compute the associated partition function as a plethystic exponential, proving a conjecture proposed in string theory by Awata--Kanno and Benini--Bonelli--Poggi--Tanzini. A crucial step in the proof is the fact, nontrival if $r>1$, that the invariants do not depend on the equivariant parameters of the framing torus $(\BC^\ast)^r$.
Reducing from K-theoretic to \emph{cohomological} invariants, we compute the corresponding DT invariants, proving a conjecture of Szabo.
Reducing further to \emph{enumerative} DT invariants, we solve the higher rank DT theory of a pair $(X,F)$, where $F$ is an equivariant exceptional locally free sheaf on a projective toric $3$-fold $X$.

As a further refinement of the K-theoretic DT invariants, we formulate a mathematical definition of the chiral elliptic genus studied in physics. This allows us to define \emph{elliptic DT invariants} of $\BA^3$ in arbitrary rank, which we use to tackle a conjecture of Benini--Bonelli--Poggi--Tanzini.
\end{abstract}

{\hypersetup{linkcolor=black}
\tableofcontents}

\section{Introduction}

\subsection{Overview}
Classical Donaldson--Thomas (DT in short) invariants of a smooth complex projective Calabi--Yau $3$-fold $Y$, introduced in \cite{ThomasThesis}, are integers that virtually count stable torsion free sheaves on $Y$, with fixed Chern character $\gamma$. However, the theory is much richer than what the bare DT numbers
\begin{equation}\label{classicalDT}
\DT(Y,\gamma)\,\in\,\BZ
\end{equation}
can capture: there are extra symmetries subtly hidden in the local structure of the moduli spaces of sheaves giving rise to the classical DT invariants \eqref{classicalDT}. This idea has been present in the physics literature for some time \cite{RefTopVertex,RMQ}. 

These hidden symmetries suggest that there should exist more \emph{refined} invariants, of which the DT numbers \eqref{classicalDT} are just a shadow. These branch out in two main directions:

\begin{itemize}
    \item [$\circ$] \emph{motivic} Donaldson--Thomas invariants, and 
    \item [$\circ$] \emph{K-theoretic} Donaldson--Thomas invariants.
\end{itemize}

For the former, which at date includes a number of interesting subbranches, the papers by Kontsevich and Soibelman \cite{KS1,KS2} are a good starting point, and Szendr\H{o}i's survey \cite{CohDT} contains an extensive bibliography on the subject. For the latter, see some recent developments after Nekrasov--Okounkov \cite{NO16}, such as \cite{Okounkov_Lectures,arbesfeld2019ktheoretic,thomas2018equivariant} and \cite{nekrasov2017magnificent,Magnificent_colors,cao2019ktheoretic} for a generalisation to Calabi--Yau 4-folds. In this paper we deal with K-theoretic DT theory. The relationship between motivic and K-theoretic, which we briefly sketch in \S\,\ref{motivic_comparison}, will be investigated in future work. 

The subtle structure of the DT moduli spaces is most evident in the \emph{local case}, i.e.~when the theory is applied to the simplest Calabi--Yau $3$-fold of all, namely the affine space $\BA^3$. See \cite{BBS,DavisonR} for the rank $1$ motivic DT theory of $\BA^3$, and \cite{Virtual_Quot} for a higher rank version.
This paper solves the K-theoretic Donaldson--Thomas theory of points of $\BA^3$.
In \cite{BR18} it is shown that the main player in the theory, the Quot scheme
\begin{equation}\label{quot}
\Quot_{\BA^3}(\OO^{\oplus r},n)
\end{equation}
of length $n$ quotients of the free sheaf $\OO^{\oplus r}$,
is a global \emph{critical locus}, i.e.~it can be realised as $\set{\dd f = 0}$, for $f$ a regular function on a smooth scheme. This structural result, revealing in bright light the symmetries we were talking about, is used to define the \emph{higher rank K-theoretic DT theory of points} that is the central character in this paper. The rank $1$ theory, corresponding to $\Hilb^n(\BA^3)$, was already defined, and it was solved by Okounkov \cite[\S~3]{Okounkov_Lectures}, proving a conjecture by Nekrasov \cite{NEKRASOV2005261}. Our first main result (Theorem \ref{mainthm:K-theoretic}) can be seen as an upgrade of his calculation, completing the study of the degree $0$ K-theoretic DT theory of $\BA^3$.

\smallbreak
In physics, remarkably, the definition of the K-theoretic DT invariants studied here already existed, and gave rise to a conjecture that our paper --- again, Theorem \ref{mainthm:K-theoretic} --- proves mathematically. More precisely, our formula for the K-theoretic DT partition function $\DT_r^{\KK}$
of $\BA^3$ was first conjectured by Nekrasov \cite{NEKRASOV2005261} for $r=1$ and by Awata and Kanno \cite{MR2545054} for arbitrary $r$  as the partition function of a quiver matrix model describing instantons of a topological $U(r)$ gauge theory on D6 branes.
The reader is referred to \S\,\ref{sec:string_stuff} of this introduction for more background on the physics picture. 

We also study higher rank \emph{cohomological DT invariants} of $\BA^3$. As we show in Corollary \ref{limit K theory cor}, these can be obtained as a suitable limit of the K-theoretic invariants. Motivated by \cite{MR2545054,Nekrasov_M-theory}, a closed formula for their generating function $\DT_r^{\coh}$ was conjectured by Szabo \cite[Conj.~4.10]{Szabo} as a generalisation of the $r=1$ case established by Maulik--Nekrasov--Okounkov--Pandharipande \cite[Thm.~1]{MNOP2}. We prove this conjecture as our Theorem \ref{mainthm:cohomological}. To get there, in \S\,\ref{sec:higher_rank_vertex} we develop a \emph{higher rank topological vertex} formalism based on the combinatorics of $r$-\emph{colored plane partitions},\footnote{In this paper, an $r$-colored plane partition is an $r$-tuple of classical plane partitions, see Definition \ref{def:colored_partition}.} generalising the classical vertex formalism of \cite{MNOP1,MNOP2}.

We pause for a second to explain a key step in this paper. The Quot scheme \eqref{quot}, which gives rise to most of the invariants we study here, is acted on by an algebraic torus
\[
\TT = (\BC^*)^3 \times (\BC^*)^r,
\]
and by their very definition, both the K-theoretic and the cohomological DT invariants depend, a priori, on the sets $t = (t_1,t_2,t_3)$ and $w = (w_1,\ldots,w_r)$ of equivariant parameters of $\TT$. A technical result of this paper, which is proved as Theorem \ref{thm:independence}, states that
\begin{equation}\label{independence_slogan}
\textrm{The K-theoretic DT invariants do not depend on the framing parameters }w.
\end{equation}
This will allow us to take arbitrary limits to evaluate our formulae. We emphasise that this independence, automatic if $r=1$ (see Remark \ref{rmk:trivial_framing_action}), is quite surprising and highly nontrivial if $r>1$.

\subsection{Main results}
We briefly outline here the main results obtained in this paper. 

\subsubsection{K-theoretic DT invariants}
As we mentioned above, the Quot scheme \eqref{quot} is a critical locus, thus it carries a natural symmetric ($\TT$-equivariant, as we prove) perfect obstruction theory in the sense of Behrend--Fantechi \cite{BFinc,BFHilb}. As we recall in \S\,\ref{sec: twisted structure sheaf}, there is also an induced \emph{twisted virtual structure sheaf} $\widehat{\OO}^{\vir} \in K_0^{\TT}(\Quot_{\BA^3}(\OO^{\oplus r},n))$, which is a twist, by a square root of the virtual canonical bundle, of the ordinary virtual structure sheaf $\OO^{\vir}$. The rank $r$ \emph{K-theoretic DT partition function} of the Quot scheme of points of $\BA^3$, encoding the rank $r$ K-theoretic DT invariants of $\BA^3$, is defined as
\[
\DT_r^{\KK}(\BA^3,q,t,w) = \sum_{n\geq 0} q^n \chi(\Quot_{\BA^3}(\OO^{\oplus r},n),\widehat{\OO}^{\vir}) \,\in\, \BZ(\!(t,(t_1t_2t_3)^{\frac{1}{2}},w)\!)\llbracket q \rrbracket,
\]
where the half power is caused by the twist by the chosen square root of the virtual canonical bundle (this choice does not affect the invariants, cf.~Remark \ref{rmk:independence_of_orientation}).

Granting Theorem \ref{thm:independence}, that we stated informally in \eqref{independence_slogan}, we shall write $\DT_r^{\KK}(\BA^3,q,t) = \DT_r^{\KK}(\BA^3,q,t,w)$, ignoring the framing parameters $w$.
In \S\,\ref{sec:proof_of_K-theoretic_thm} we determine a closed formula for this series, proving the conjecture by Awata--Kanno \cite{MR2545054}. This conjecture was checked for low number of instantons in \cite[\S~4]{MR2545054}.

To state our first main result, we need to recall the definition of the \emph{plethystic exponential}. Given an arbitrary power series
\[
f = f(p_1,\ldots,p_e;u_1,\ldots,u_\ell) 
\,\in\,\BQ (p_1,\ldots,p_e)\llbracket u_1,\ldots,u_\ell \rrbracket,
\]
one sets
\begin{equation}\label{def_Exp}
\Exp\left(f\right) = \exp\left(\sum_{n>0}\frac{1}{n}f(p_1^n,\ldots,p_e^n;u_1^n,\ldots,u_\ell^n)\right),
\end{equation}
viewed as an element of $\BQ (p_1,\ldots,p_e)\llbracket u_1,\ldots,u_\ell \rrbracket$. Consider, for a formal variable $x$, the operator $[x] = x^{1/2} - x^{-1/2}$. In \S\,\ref{sec: preliminaries K theory} we consider this operator on $K_0^{\TT}(\pt)$. See  \S\,\ref{subsec:Witten} for its physical interpretation. We are now able to state our first main result.

\begin{thm}[Theorem \ref{thm: K theoretic of points higher rank}]\label{mainthm:K-theoretic}
The rank $r$ K-theoretic DT partition function of $\BA^3$ is given by
\begin{equation}\label{eqn:DT^K}
\DT_r^{\KK}(\BA^3,(-1)^rq,t)=\Exp\left(\mathsf F_r(  q,t_1,t_2,t_3)\right),
\end{equation}
where, setting $\mathfrak{t} = t_1t_2t_3$, one defines
\[ 
\mathsf F_r( q,t_1,t_2,t_3)= \frac{[\mathfrak{t}^r]}{[\mathfrak{t}][\mathfrak{t}^{\frac{r}{2}}q ][\mathfrak{t}^{\frac{r}{2}} q^{-1} ]}\frac{[t_1t_2][t_1t_3][t_2t_3]}{[t_1][t_2][t_3]}.
\]
\end{thm}

The case $r=1$ of Theorem \ref{mainthm:K-theoretic} was proved by Okounkov in \cite{Okounkov_Lectures}. The general case was proposed conjecturally in \cite{MR2545054,BBPT}. Arbesfeld and Kononov informed us during the writing of this paper that they also obtained a different proof of Theorem \ref{mainthm:K-theoretic} \cite{Noah_Yasha}, driven by different motivations.

It is interesting to notice that Formula \eqref{eqn:DT^K} is equivalent to the product decomposition
\begin{equation}\label{eqn:product}
\DT_r^{\KK}(\BA^3,(-1)^rq,t)=\prod_{i=1}^r \DT_1^{\KK}(\BA^3,-q\mathfrak{t}^{\frac{-r-1}{2} +i},t),
\end{equation}
that we obtain in Theorem \ref{thm: r copies of rank 1 K theory}. This  is precisely the product formula \cite[Formula (3.14)]{Magnificent_colors} appearing as a limit of the (conjectural) $4$-fold theory developed by Nekrasov and Piazzalunga.

Formula \eqref{eqn:product} is also related to its motivic cousin: as we observe in \S\,\ref{motivic_comparison}, the motivic partition function $\DT_r^{\mot}$ of the Quot scheme of points of $\BA^3$ (see \cite[Prop.~2.7]{BR18} and the references therein) satisfies the same product formula \eqref{eqn:product}, after the transformation $\mathfrak t^{1/2} \to - \BL^{1/2}$.

\subsubsection{Cohomological DT invariants}
The generating function of \emph{cohomological DT invariants} is defined as
\[
 \DT_r^{\coh}(\mathbb{A}^3,q,s,v) = \sum_{n\geq 0}q^n \int_{[\Quot_{\BA^3}(\OO^{\oplus r},n)]^{\vir}} 1 \,\in\, \mathbb{Q}(\!(s,v)\!)\llbracket q \rrbracket
\]
where $s=(s_1,s_2,s_3)$ and $v = (v_1,\ldots,v_r)$, with $s_i = c_1^{\TT}(t_i)$ and $v_j = c_1^{\TT}(w_j)$ respectively, and the integral is defined in Equation \eqref{def:DT_coh_equivariant_residues} via $\TT$-equivariant residues.
It is a consequence of \eqref{independence_slogan} that $\DT_r^{\coh}(\mathbb{A}^3,q,s,v)$ does not depend on $v$, so we will shorten it as $\DT_r^{\coh}(\mathbb{A}^3,q,s)$. 
In \S\,\ref{sec: cohom reduction} we explain how to recover the cohomological invariants out of the K-theoretic ones. This is the limit formula (Corollary \ref{limit K theory cor})
\[
\DT_r^{\coh}(\BA^3,q,s) = \lim_{b \to 0} \DT_r^{\KK}(\BA^3,q,e^{bs}),
\]
essentially a formal consequence of our explicit expression for the K-theoretic \emph{higher rank vertex} (cf.~\S\,\ref{sec:higher_rank_vertex}) attached to the Quot scheme $\Quot_{\BA^3}(\OO^{\oplus r},n)$.

\begin{thm}[Theorem \ref{thm:cohomological}]\label{mainthm:cohomological}
The rank $r$ cohomological DT partition function of $\BA^3$ is given by
 \[
 \DT_r^{\coh}(\mathbb{A}^3,q,s)=\mathsf M((-1)^rq)^{-r\frac{(s_1+s_2)(s_1+s_3)(s_2+s_3)}{s_1s_2s_3}},
 \]
where $\mathsf M(t) = \prod_{m\geq 1}(1-t^m)^{-m}$ is the MacMahon function.
\end{thm}

The case $r=1$ of Theorem \ref{mainthm:cohomological} was proved by Maulik--Nekrasov--Okounkov--Pandharipande \cite[Thm.~1]{MNOP2}. Theorem \ref{mainthm:cohomological} was conjectured by Szabo in \cite{Szabo} and confirmed for $r\leq 8$ and $n\leq 8$ in \cite{BBPT}. The specialisation $\DT_r^{\coh}(\mathbb{A}^3,q,s)\big|_{s_1+s_2+s_3=0}=\mathsf M((-1)^rq)^r$ was already computed in physics \cite{Cir-Sink-Szabo}.

\subsubsection{Elliptic DT invariants}
In \S\,\ref{sec:elliptic_invariants} we define the \emph{virtual chiral elliptic genus} for any scheme with a perfect obstruction theory, which recovers as a special case the virtual elliptic genus defined in \cite{Fantechi_Gottsche}. By means of this new invariant we introduce a refinement $\DT_r^{\rm ell}$ of the generating series $\DT_r^{\KK}$, providing a mathematical definition of the \emph{elliptic DT invariants} studied in \cite{BBPT}. We propose a conjecture (Conjecture \ref{conj: non dependence under some limits of elliptic}) about the behaviour of $\DT_r^{\rm ell}$ and, granting this conjecture, we obtain a proof of a conjecture formulated by Benini--Bonelli--Poggi--Tanzini (Theorem \ref{thm:elliptic}).

\subsubsection{Global geometry}
So far we have only discussed results concerning \emph{local} geometry. When $X$ is a \emph{projective} toric $3$-fold and $F$ is an equivariant exceptional locally free sheaf, by \cite[Thm.~A]{Virtual_Quot} there is a $0$-dimensional torus equivariant (cf.~Proposition \ref{pot_global_equivariant}) perfect obstruction theory on $\Quot_X(F,n)$. Therefore the higher rank Donaldson--Thomas invariants
\[
\DT_{F,n} = \int_{[\Quot_X(F,n)]^{\vir}}1 \,\in \,\BZ
\]
can be computed via the Graber--Pandharipande virtual localisation formula \cite{GPvirtual}. The next result confirms (in the toric case) a prediction \cite[Conj.~3.5]{Virtual_Quot} for their generating function. Before stating it, recall that an \emph{exceptional sheaf} on a variety $X$ is a coherent sheaf $F \in \Coh X$ such that $\Hom(F,F)=\BC$ (i.e.~$F$ is simple), and $\Ext^i(F,F) = 0$ for $i>0$.

\begin{thm}[Theorem \ref{thm for toric proj}]\label{mainthm:projective_toric}
Let $(X,F)$ be a pair consisting of a smooth projective toric $3$-fold $X$ along with an exceptional equivariant locally free sheaf $F$ of rank $r$. Then
\[
\sum_{n\geq 0} \DT_{F,n}q^n= \mathsf M((-1)^rq)^{r\int_{X}c_3(T_X\otimes K_X)}.
\]
\end{thm} 
The corresponding formula in the Calabi--Yau case was proved in \cite[\S\,3.2]{Virtual_Quot}, whereas the general rank $1$ case was proved in \cite[Thm.~2]{MNOP2} and \cite[Thm.~0.2]{JLI}.

\subsection{Relation to string theory}\label{sec:string_stuff}
An interpretation for Donaldson--Thomas invariants is available also in the context of supersymmetric string theories. In this framework, one is interested in countings of BPS-bound states on (Calabi--Yau) $3$-folds. The interest in studying the BPS sector of string theories lies in the fact that it consists of quantities which are usually protected from quantum corrections, and which can sometimes be studied non-perturbatively. On the other hand, BPS countings have been shown in many occasions to have a precise mathematical interpretation rooted in counting problems in enumerative geometry. In the case of Donaldson--Thomas theory the interpretation is that of a type IIA theory compactified on $X$, so that the four dimensional effective theory has an $\mathcal N=2$ supersymmetry content. BPS states then preserve half of this supersymmetry and are indexed by a charge vector living on a lattice determined by the cohomology (with compact support) of the $3$-fold $\Gamma\cong H^0(X,\BZ)\oplus H^2(X,\BZ)\oplus H^4(X,\BZ)\oplus H^6(X,\BZ)$. These lattices can also be interpreted as the charge lattices of D-branes wrapping $p$-cycles on $X$, where
\[
{\rm D}p \longleftrightarrow H^{6-p}(X,\BZ)\,\cong\, H_p(X,\BZ),
\]
for $p=0,2,4,6$. The Witten index
\[
\Ind_X(\gamma)=\Tr_{\mathcal H^{(\gamma)}_{{\rm BPS}}}(-1)^F
\]
provides a measure of the degeneracy of the BPS states, where the trace is over the fixed charge sector of the single-particle Hilbert space $H_{{\rm BPS}}=\bigoplus_{\gamma\in\Gamma}\mathcal H^{(\gamma)}_{\rm BPS}$ and $F$ is a given one-particle operator acting on the $\gamma$-component of the Hilbert space. The choice of a charge vector of the form $\gamma=(r,0,-\beta,n)$ is then equivalent to a system of (D0-D4-D6)-branes, where $r$ D6-branes wrap the whole $3$-fold $X$.
The index of the theory is then defined via integration over the virtual class of the moduli space of BPS states $\mathcal M_{\rm BPS}^{(\gamma)}(X)$,
\[
\Ind_X(\gamma)=\int_{[\mathcal M_{\rm BPS}^{(\gamma)}(X)]^{\vir}}1.
\]
The moduli space of BPS states on $X=\BA^3$ with charge vector $\gamma=(r,0,0,n)$ is then identified with the Quot scheme $\Quot_{\BA^3}(\OO^{\oplus r},n)$, and the partition function of the theory reproduces the generating function of degree 0 DT invariants of $\BA^3$ 
\[
\sum_{n\ge 0}\Ind_{\BA^3}(r,0,0,n)q^n=\mathsf{DT}_r^{\coh}(\BA^3,q,s).
\]
An equivalent interpretation can be also given in type IIB theories, where the relevant systems will be those of (D(-1)-D3-D5)-branes. The effective theory on the D(-1)-branes, which in this case are $0$-dimensional objects, is a quiver matrix model encoding the critical structure of $\Quot_{\BA^3}(\OO^{\oplus r},n)$. K-theoretic and elliptic versions of DT invariants can also be studied by suitably generalising the D-branes construction. One can study for instance a D0-D6 brane system with $r$ branes on $X\times S^1$, $S^1$ being the worldvolume of the D0-branes. In this case the D0-branes quantum mechanics describe the K-theoretic generalisation of DT theory, so that the supersymmetric partition function computes the equivariant Euler characteristic of the twisted virtual structure sheaf $\widehat{\OO}^{\vir}$ of the Quot scheme identified with the moduli space of BPS vacua of the theory. Analogously, the D1 string theory of a D1-D7 system on the product of $X$ by an elliptic curve, say $X\times T^2$, provides the elliptic generalisation of higher rank DT invariants, \cite{BBPT}, while the superconformal index realises the virtual elliptic genus of the moduli space of BPS states. In this context, a plethystic formula for the generating function of the K-theoretic DT invariants was conjectured in \cite{Nekrasov_M-theory} for rank $1$ and in \cite{MR2545054} for rank $r$. The elliptic generalisation of DT theory was studied in \cite{BBPT}, where a plethystic formula for the virtual elliptic genus was conjectured in the Calabi--Yau specialisation $t_1t_2t_3=1$ and its generalisation $(t_1t_2t_3)^r=1$ (cf. Theorem \ref{thm:elliptic}).

\subsection{Related works on virtual invariants of Quot schemes}
The study of virtual invariants of Quot schemes is an active research area in enumerative geometry. See for instance the seminal work of Marian--Oprea \cite{MR2271296}, where a virtual fundamental class  on the Quot scheme of the trivial bundle on a smooth projective curve was constructed. More recently, there has been a lot of activity on surfaces (partially motivated by Vafa--Witten theory), see e.g.~\cite{Oprea:2019ab,Johnson:2020aa,Lim:2020aa,MR3621431}. In connection to physics, flags of framed torsion free sheaves on $\BP^2$, as a generalisation of Quot schemes of points on $\BA^2$, were studied in \cite{bonelli2019defects,bonelli2019flags}. Many of the results mentioned above use virtual localisation to study virtual invariants of Quot schemes: our paper exploits virtual localisation as well, though in the K-theoretic setup and in the almost unexplored land of $3$-folds.

Generating functions of Euler characteristics of Quot schemes on $3$-folds (for quotients of sheaves of homological dimension at most $1$) were computed by Gholampour--Kool in \cite{Gholampour2017}. Quot schemes of \emph{reflexive} sheaves (of rank $2$) also appeared in the work of Gholampour--Kool--Young \cite{Gholampour2017a}, as fibres of ``double dual maps'', exploited to compute generating functions of Euler characteristics of more complicated moduli spaces of sheaves. Virtual invariants are also defined (again, in the rank $2$ case) via localisation in \cite{Gholampour2017a}. %Interesting rationality conjectures on such generating functions were presented in \cite{Gholampour2017} and confirmed in some examples in \cite{Gholampour2017a}.
In the rank $1$ case, quotients of the ideal sheaf of a smooth curve $C$ in a $3$-fold $Y$ form a closed subscheme of the Hilbert scheme of curves in $Y$, the key player in rank $1$ DT theory. The series of virtual Euler characteristics of the associated Quot scheme was computed by the third author in \cite{LocalDT}, and later upgraded to a $C$-local DT/PT wall-crossing formula \cite{Ricolfi2018}. The motivic refinement of this DT/PT correspondence was established by Davison and the third author in \cite{DavisonR}.

Virtual classes of dimension $0$ of Quot schemes (for locally free sheaves) on projective $3$-folds are constructed under certain assumptions in \cite{Virtual_Quot}. In this paper we mainly work with the virtual class on $\Quot_{\BA^3}(\OO^{\oplus r},n)$ coming from the critical obstruction theory found in \cite{BR18}.

\subsection{Plan of the paper}
In \S\,\ref{sec:background} we recall the basics on perfect obstruction theories, virtual classes, virtual structure sheaves and how to produce virtual invariants out of these data; we review the K-theoretic virtual localisation theorem in \S\,\ref{sec:vir_loc}. Sections \ref{sec: quot scheme local model}\,--\,\ref{sec:elliptic_invariants} are devoted to the ``local Quot scheme'' $\Quot_{\BA^3}(\OO^{\oplus r},n)$. In \S\,\ref{sec: quot scheme local model} we recall its critical structure and we define a $\TT$-action on it, whose fixed locus is parametrised by the finitely many $r$-colored plane partitions (Proposition \ref{prop:fixedlocus_indexed_by_colored_partitions}); we study the equivariant critical obstruction theory on the Quot scheme and prove that the induced virtual class on the $\TT$-fixed locus is trivial (Corollary \ref{cor:trivial pot on fixed locus}). In \S\,\ref{subsec:virtual_invariants_quot} we introduce cohomological and K-theoretic DT invariants of $\Quot_{\BA^3}(\OO^{\oplus r},n)$. In \S\,\ref{sec:higher_rank_vertex} we develop a higher rank vertex formalism which we exploit to write down a formula (Proposition \ref{prop:Tangent^vir}) for the virtual tangent space of a $\TT$-fixed point in the Quot scheme. In \S\,\ref{sec:K-theory} we prove Theorem \ref{mainthm:K-theoretic} as well as Formula \eqref{eqn:product}. In \S\,\ref{sec:cohomological_invariants} we prove Theorem \ref{mainthm:cohomological} and we show that $\DT_r^{\coh}$ does not depend on any choice of possibly nontrivial $(\BC^\ast)^3$-weights on $\OO^{\oplus r}$. In \S\,\ref{sec:elliptic_invariants}, we give a mathematically rigorous definition of a ``chiral'' version of the virtual elliptic genus of \cite{Fantechi_Gottsche} and use it in \S\,\ref{sec:elliptic DT} to define elliptic DT invariants. In \S\,\ref{sec:limits of elliptic DT} we also give closed formulae for elliptic DT invariants in some limiting cases, based on the conjectural independence on the elliptic parameter --- see Conjecture \ref{conj: non dependence under some limits of elliptic} and Remark \ref{rem: elliptic indepedence}. In particular Theorem \ref{thm:elliptic} proves a conjecture recently appeared in the physics literature \cite[Formula (3.20)]{BBPT}. In \S\,\ref{sec:compact_section} we prove Theorem \ref{mainthm:projective_toric} by gluing vertex contributions from the toric charts of a projective toric $3$-fold.

\begin{conventions*}
We work over $\BC$. A \emph{scheme} is a separated scheme of finite type over $\BC$. If $Y$ is a scheme, we let $\derived^{[a,b]}(Y)$ denote the derived category of coherent sheaves on $Y$, whose objects are complexes with vanishing cohomology sheaves outside the interval $[a,b]$. We let $K^0(Y)$ be the K-group of vector bundles on $Y$. When $Y$ carries an action by an algebraic torus $\TT$, we let $K^0_\TT(Y)$ be the K-group of $\TT$-\emph{equivariant} vector bundles on $Y$. Similarly, we let $K_0(Y)$ denote the K-group of \emph{coherent sheaves} on $Y$, and we let $K_0^{\TT}(Y)$ be the K-group of (the abelian category of) $\TT$-\emph{equivariant} coherent sheaves on $Y$. When $Y$ is smooth, the natural $\BZ$-linear map $K^0(Y) \to K_0(Y)$, resp.~$K^0_{\TT}(\pt)$-linear map $K^0_{\TT}(Y) \to K_0^{\TT}(Y)$, is an isomorphism.
Chow groups $A^\ast (Y)$ and cohomology groups $H^\ast(Y)$ are taken with rational coefficients.
\end{conventions*}

\subsection*{Acknowledgements}
We thank Noah Arbesfeld and Yakov Kononov for generously sharing with us their progress on their paper \cite{Noah_Yasha}, for reading a first draft of this paper and for sending us very interesting comments. Thanks to Alberto Cazzaniga for suggesting the comparison between motivic and K-theoretic factorisations. We thank Barbara Fantechi for enlightening discussions about equivariant obstruction theories. We are grateful to Martijn Kool, Giulio Bonelli and Alessandro Tanzini for the many discussions, insightful suggestions and for commenting on a first  draft of this paper. Special thanks to Richard Thomas for generously sharing with us his ideas and insights on equivariant sheaves and Atiyah classes. Finally, many thanks to the anonymous referee for suggesting valuable improvements and corrections.

S.M.~is supported by NWO grant TOP2.17.004. A.R.~was funded by Dipartimenti di Eccellenza.

%%%%%%%%%%%%%%%%%%%%%%%%%%%%%%%%%%%%%%%%%%%%%%%%%%%%%%%%%%%%
%%%%%%%%%%%%%%%%%%%%%%%%%%%%%%%%%%%%%%%%%%%%%%%%%%%%%%%%%%%%
\section{Background material}\label{sec:background}

\subsection{Obstruction theories and virtual classes}
A \emph{perfect obstruction theory} on a scheme $X$, as defined in \cite{LiTian,BFinc}, is the datum of a morphism 
\[
\phi\colon \BE\to\BL_X
\]
in $\derived^{[-1,0]}(X)$, where $\BE$ is a perfect complex of perfect amplitude contained in $[-1,0]$, such that $h^0(\phi)$ is an isomorphism and $h^{-1}(\phi)$ is surjective. Here, $\BL_X = \tau_{\geq -1}L_X^\bullet$ is the cut-off at $-1$ of the full cotangent complex $L_X^\bullet \in \derived^{[-\infty,0]}(X)$ defined by Illusie \cite{Illusie_I}. A perfect obstruction theory is called \emph{symmetric} (see \cite{BFHilb}) if there exists an isomorphism $\theta\colon \BE \simto \BE^\vee[1]$ such that $\theta = \theta^\vee[1]$. 

The \emph{virtual dimension} of $X$ with respect to $(\BE,\phi)$ is the integer $\vd = \rk \BE$. This is just $\rk E^0 - \rk E^{-1}$ if one can write $\BE = [E^{-1}\to E^0]$. 

\begin{remark}\label{rmk:K-class_of_symmetric_pot}
A symmetric obstruction theory $\BE \to \BL_X$  has virtual dimension $0$, and moreover the \emph{obstruction sheaf}, defined as $\Ob = h^1(\BE^\vee)$, is canonically isomorphic to the cotangent sheaf $\Omega_X$. In particular, one has the K-theoretic identity
\[
\BE = h^0(\BE) - h^{-1}(\BE) = \Omega_X - T_X\,\in\,K_0(X).
\]
\end{remark}

A perfect obstruction theory determines a cone
\[
\mathfrak C \into E_1 = (E^{-1})^\vee.
\]
Letting $\iota\colon X \into E_1$ be the zero section of the vector bundle $E_1$, the induced \emph{virtual fundamental class} on $X$ is the refined intersection
\[
[X]^{\vir} = \iota^![\mathfrak C]\, \in\, A_{\vd}(X). 
\]
By a result of Siebert \cite[Thm.~4.6]{Siebert}, the virtual fundamental class depends only on the K-theory class of $\BE$.

\begin{definition}[\cite{GPvirtual,BFHilb}]
\label{def:equivariant_pot}
Let $G$ be an algebraic group acting on a scheme $X$. A $G$-equivariant perfect obstruction theory on $X$ is a choice of lift of a perfect obstruction theory $\phi\colon \BE \to \BL_X$ to the $G$-equivariant derived category $\derived^{[-1,0]}(\Coh_X^G)$.
\end{definition}

\subsection{From obstruction theories to virtual invariants}
On a proper scheme $X$ with a perfect obstruction theory, one can define virtual enumerative invariants by
\[
\int_{[X]^{\vir}}\alpha \,\in\, \BQ,
\]
where $\alpha\in A^i (X)$. These intersection numbers are going to vanish if $i \neq \vd$.

On the K-theoretic side, it was observed in \cite[\S\,5.4]{BFinc} that a perfect obstruction theory $\BE \to \BL_X$ not only induces a virtual fundamental class, but also a \emph{virtual structure sheaf} 
\[
\OO_X^{\vir} = [\mathbf L \iota^\ast \OO_{\mathfrak C}]\,\in\, K_0(X).
\]
Its construction first appeared in \cite{kontsevich_94,kapranov_2009} in the context of dg-manifolds and  in  \cite{BFinc,Fantechi_Gottsche} in the language of perfect obstruction theories. More recently, Thomas gave a description of $\OO_X^{\vir}$  in terms of the \emph{$K$-theoretic Fulton class}, showing that it only depends on the K-theory class of $\BE$ \cite[Cor.~4.5]{Thomas_Kth_Fulton_Class}. If $\pi\colon X \to \pt$ is proper, one can use $\OO_X^{\vir}$ and K-theoretic pushfoward $\pi_\ast = \chi(X,-)$ to define virtual invariants by
\[
\chi^{\vir}(X,V)=\chi(X, V\otimes \OO_X^{\vir})\, \in\, K_0(\pt) =  \mathbb{Z},\quad V\in K^0(X).
\]
A virtual version \cite[Cor.~3.6]{Fantechi_Gottsche} of the Hirzebruch--Riemann--Roch theorem holds: one has
\[
\chi^{\vir}(X,V)={\int_{[X]^{\vir}}}  \ch(V)\cdot \td(T_X^{\vir}),
\]
where, setting $E_i = (E^{-i})^\vee$, one defines the \emph{virtual tangent bundle} of $X$ by the formula 
\[
T^{\vir}_X = \BE^{\vee} = E_0 - E_1\,\in\,K^0(X).
\]

%%%%%%%%%%%%%%%%%%%%%%%%%%%%%%%%%%%%%%%%%%%%%
\subsection{Torus representations and their weights}
Let $\TT = (\BC^*)^g$ be an algebraic torus, with character lattice $\widehat{\TT} = \Hom(\TT,\BC^\ast) \cong \BZ^g$. Let $K_0^{\TT}(\pt)$ be the K-group of the category of $\TT$-representations. Any finite dimensional $\TT$-representation $V$ splits as a sum of $1$-dimensional representations called the \emph{weights} of $V$. Each weight corresponds to a character $\mu \in \widehat{\TT}$, and in turn each character corresponds to a monomial $t^\mu = t_1^{\mu_1}\cdots t_g^{\mu_g}$ in the coordinates of $\TT$. The map
\begin{equation}\label{eqn:trace}
\tr\colon K_0^{\TT}(\pt) \to \BZ \left[t^\mu \mid \mu \in \widehat{\TT}\right],\quad V\mapsto \tr_V,
\end{equation}
sending the class of a $\TT$-module to its decomposition into weight spaces is a ring isomorphism, where tensor product on the left corresponds to the natural multiplication on the right. We will therefore sometimes identify a (virtual) $\TT$-module with its character. Sometimes, to ease notation, we shall write $\tr(V)$ instead of $\tr_V$.

\begin{example}
Let $V=\sum_\mu t^\mu$ be a $\TT$-module. Define $\Lambda^\bullet_p V = \sum_{i=0}^{\rk V}p^i\Lambda^i V$ to be the \emph{total wedge of} $V$. We shall write $\Lambda^\bullet V = \Lambda^\bullet_{-1} V$. As shown for instance in \cite[Ex.~2.1.5]{Okounkov_Lectures}, its trace satisfies
\begin{equation*}
    \tr_{\Lambda^\bullet_{-p} V}=\prod_\mu\,(1-pt^\mu).
\end{equation*}
\end{example}

\subsection{Virtual normal bundle and virtual tangent space}
\label{subsec:Nvir}
Let $\TT = (\BC^\ast)^g$ be an algebraic torus. If $Y$ is a scheme carrying the trivial $\TT$-action, any $\TT$-equivariant coherent sheaf $B \in \Coh Y$ admits a direct sum decomposition $B = \bigoplus_{\mu} B^{\mu}$ into eigensheaves, where $\mu \in \widehat \TT \cong \BZ^g$ ranges over the characters of $\TT$. The $\TT$-\emph{fixed part} and the $\TT$-\emph{moving part} of $B$ are defined as
\[
B^{\fix} = B^0,\quad B^{\mov} = \bigoplus_{\mu\neq 0}B^{\mu}.
\]
This definition extends to complexes of coherent sheaves.

If $X$ is a scheme carrying a $\TT$-action and a $\TT$-equivariant perfect obstruction theory $\BE \to \BL_{X}$, and $Y \subset X^{\TT}$ is a component of the fixed locus, then the \emph{virtual normal bundle} of $Y\subset X$ is the complex
\[
N^{\vir}_{Y/X}=T_X^{\vir}\big|_{Y}^{\mov} = \BE^\vee \big|_{Y}^{\mov}.
\]
If $p \in X^{\TT}$ is a fixed point, the \emph{virtual tangent space} of $X$ at $p$ is $T_p^{\vir} = \BE^\vee \big|_{p} \in K_0^{\TT}(\pt)$.

\subsection{Virtual localisation}\label{sec:vir_loc}
A very useful tool to compute virtual K-theoretic invariants is the \emph{virtual localisation theorem}. Its first version was proven in equivariant Chow theory in \cite{GPvirtual}, and a K-theoretic version appeared in \cite[\S~7]{Fantechi_Gottsche} and \cite[\S~3]{Qu_virtual}. Suppose that $X$ carries an action of an algebraic torus $\TT$ and a $\TT$-equivariant perfect obstruction theory. Then $\OO_X^{\vir}$ is naturally an element of $K_0^{\TT}(X)$ and if $X$ is proper one can define $\chi^{\vir}(X,-) = \chi(X,-\otimes \OO_X^{\vir})$ by means of the equivariant pushforward $\chi(X,-)\colon K_0^{\TT}(X)\to K_0^{\TT}(\pt)$. The K-theoretic virtual localisation formula states that, for an arbitrary element $V\in K_\TT^0(X)$  of the $\TT$-equivariant K-theory, one has the identity
\begin{equation}\label{K-theoretic_localisation}
\chi^{\vir}(X,V)=\chi^{\vir}\left(X^\TT, \frac{V|_{X^\TT}}{\Lambda^\bullet N^{\vir,\vee}} \right)\,\,\in\,\, K^\TT_0(\pt)\left[\frac{1}{1-t^{\mu}}\, \Bigg{|}\, \mu \in \widehat{\TT}\right],
\end{equation}
where $X^\TT\subset X$ is the $\TT$-fixed locus, $N^{\vir} \in K^0_{\TT}(X^\TT)$ is K-theory class of the virtual normal bundle, i.e.~of the $\TT$-moving part of $T_X^{\vir}=\BE^\vee$ restricted to $X^{\TT}$.

More than just being a powerful theorem, the localisation formula allows one to define invariants for \emph{quasiprojective} $\TT$-varieties, provided that they have proper $\TT$-fixed locus: if this is the case, one defines $\chi^{\vir}(X,V)$ to be the right hand side of \eqref{K-theoretic_localisation}. 
If $X$ is proper, this definition coincides with the usual one thanks to the localisation theorem.

\begin{remark}
We thank  Noah Arbesfeld for pointing out to us that for quasiprojective schemes, in contrast to the case of equivariant \emph{cohomology}, one can directly define the equivariant Euler characteristic of a  coherent sheaf via its character as a torus representation (without invoking virtual localisation), provided that the weight spaces of the sheaf are finite dimensional.
\end{remark}

%%%%%%%%%%%%%%%%%%%%%%%%%%%%%%%%%%%%%%%%%%%%%%%%%%%%%%%%%%%%%%
\section{The local Quot scheme: critical and equivariant structure}\label{sec: quot scheme local model}

%%%%%%%%%%%%%%%%%%%%%%%%%%%%%%%%%%%%%%%%%%%%%%%%%%%%%%%%%%%%%%
\subsection{Overview}
In this section we start working on the local Calabi--Yau $3$-fold $\BA^3$. Fix integers $r\geq 1$ and $n\geq 0$. Our focus will be on the \emph{local Quot scheme}
\[
\Quot_{\BA^3}(\mathscr O^{\oplus r},n),
\]
whose points correspond to short exact sequences
\[
0\to S\to \mathscr O^{\oplus r}\to T \to 0
\]
where $T$ is a $0$-dimensional $\OO_{\BA^3}$-module with $\chi(T)=n$.

We shall use the following notation throughout.

\begin{notation}
If $F$ is a locally free sheaf on a variety $X$, and $F\onto T$ is a surjection onto a $0$-dimensional sheaf of length $n$, with kernel $S\subset F$, we denote by 
\[
[S] \in \Quot_{X}(F,n)
\]
the corresponding point in the Quot scheme. 
\end{notation}

In this section, we will:

\begin{itemize}
    \item [$\circ$] recall from \cite{BR18} the description of the Quot scheme as a critical locus (\S\,\ref{sec:critical_structure_Quot}),
    \item [$\circ$] describe a $\TT$-action (for $\TT = (\BC^*)^3\times (\BC^*)^r$ a torus of dimension $3+r$) on $\Quot_{\BA^3}(\mathscr O^{\oplus r},n)$, with isolated fixed locus consisting of direct sums of monomial ideals (\S\,\ref{sec:torus_actions}), 
    \item [$\circ$] reinterpret the fixed locus $\Quot_{\BA^3}(\mathscr O^{\oplus r},n)^{\TT}$ in terms of colored partitions (\S\,\ref{sec: combinatorial description of fixed locus}),
    \item [$\circ$] prove that the critical perfect obstruction theory on $\Quot_{\BA^3}(\mathscr O^{\oplus r},n)$ is $\TT$-equivariant (Lemma \ref{lemma:T_equivariance_of_POT}), and that the induced $\TT$-fixed obstruction theory on the fixed locus is trivial (Corollary \ref{cor:trivial pot on fixed locus}).
\end{itemize}
The content of this section is the starting point for the definition (see \S\,\ref{subsec:virtual_invariants_quot}) of virtual invariants on $\Quot_{\BA^3}(\mathscr O^{\oplus r},n)$, as well as our construction (see \S\,\ref{sec:higher_rank_vertex}) of the \emph{higher rank vertex formalism}.

%%%%%%%%%%%%%%%%%%%%%%%%%%%%%%%%%%%%%%%%%%%%%%%%%%%%%%%%%%%%%%
\subsection{The critical structure on the Quot scheme}\label{sec:critical_structure_Quot}
Let $V$ be an $n$-dimensional complex vector space. Consider the space $R_{r,n} = \Rep_{(n,1)}(\widetilde{\mathsf L}_3)$ of $r$-framed $(n,1)$-dimensional representations of the $3$-loop quiver $\mathsf L_3$, depicted in Figure \ref{fig:3loopquiver_framed}. The notation ``$(n,1)$'' means that the main vertex (the one belonging to the $3$-loop quiver, labelled ``$0$'' in the figure) carries a copy of $V$, whereas the framing vertex (labelled ``$\infty$'') carries a copy of $\BC$.

\begin{figure}[ht]
\begin{tikzpicture}[>=stealth,->,shorten >=2pt,looseness=.5,auto]
  \matrix [matrix of math nodes,
           column sep={3cm,between origins},
           row sep={3cm,between origins},
           nodes={circle, draw, minimum size=7.5mm}]
{ 
|(A)| \infty & |(B)| 0 \\         
};
\tikzstyle{every node}=[font=\small\itshape]
\path[->] (B) edge [loop above] node {$A_1$} ()
              edge [loop right] node {$A_2$} ()
              edge [loop below] node {$A_3$} ();

\node [anchor=west,right] at (-0.2,0.1) {$\vdots$};
\node [anchor=west,right] at (-0.3,0.95) {$u_1$};              
\node [anchor=west,right] at (-0.3,-0.85) {$u_r$};              
\draw (A) to [bend left=25,looseness=1] (B) node [midway,above] {};
\draw (A) to [bend left=40,looseness=1] (B) node [midway] {};
\draw (A) to [bend right=35,looseness=1] (B) node [midway,below] {};
\end{tikzpicture}
\caption{The $r$-framed $3$-loop quiver $\widetilde{\mathsf L}_3$.}\label{fig:3loopquiver_framed}
\end{figure}
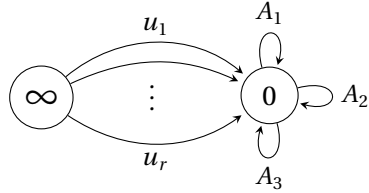

We have that $R_{r,n}$ is an affine space of dimension $3n^2+rn$, with an explicit description as
\begin{align*}
R_{r,n} &= \Set{(A_1,A_2,A_3,u_1,\ldots,u_r)|A_j \in \End(V),\,u_i \in V} \\
&=\End(V)^{\oplus 3}\oplus V^{\oplus r}.
\end{align*}
By \cite[Prop.~2.4]{BR18}, there exists a stability parameter $\theta$ on the $3$-loop quiver such that $\theta$-stable framed  representations $(A_1,A_2,A_3,u_1,\ldots,u_r) \in R_{r,n}$ are precisely those satisfying the condition:
\[
\textrm{the vectors }u_1,\ldots,u_r \in V \textrm{ jointly generate }(A_1,A_2,A_3) \in \Rep_n (\mathsf L_3).
\]
Imposing this stability condition on $R_{r,n}$ we obtain an open subscheme
\[
U_{r,n} \subset R_{r,n}
\]
on which $\GL(V)$ acts freely by the rule
\[
g \cdot (A_1,A_2,A_3,u_1,\ldots,u_r) = (gA_1g^{-1},gA_2g^{-1},gA_3g^{-1},gu_1,\ldots,gu_r). 
\]
The quotient
\begin{equation}\label{def:NCquot}
\NCQuot_{r}^n = U_{r,n}/\GL(V)
\end{equation}
is a smooth quasiprojective variety of dimension $2n^2+rn$. In \cite{BR18} the scheme $\NCQuot_{r}^n$ is referred to as the \emph{non-commutative Quot scheme}, by analogy with the \emph{non-commutative Hilbert scheme} \cite{Nori1}, i.e.~the moduli space of left ideals of codimension $n$ in $\BC\langle x_1,x_2,x_3\rangle$ (which of course exists for an arbitrary number of free variables).

\smallbreak
On $R_{r,n}$ one can define the function
\[
h_n\colon R_{r,n} \to \BA^1, \quad (A_1,A_2,A_3,u_1,\ldots,u_r) \mapsto \Tr A_1[A_2,A_3],
\]
induced by the superpotential $\mathsf W=A_1[A_2,A_3]$ on the $3$-loop quiver. Note that this function
\begin{itemize}
    \item is symmetric under cyclic permutations of $A_1$, $A_2$ and $A_3$, and
    \item does not touch the vectors $u_i$, which are only used to define its domain.
\end{itemize} 
Moreover, $h_n|_{U_{r,n}}$ is $\GL(V)$-invariant, and thus descends to a regular function
\begin{equation}\label{eqn:superpotential_ncquot}
f_n \colon \NCQuot_{r}^n \to \BA^1.
\end{equation}

\begin{prop}[{\cite[Thm.~2.6]{BR18}}]\label{prop:SPOT_A^3}
There is an identity of closed subschemes
\[
\Quot_{\BA^3}(\mathscr O^{\oplus r},n) = \crit(f_n) \subset \NCQuot_{r}^n.
\]
In particular, $\Quot_{\BA^3}(\mathscr O^{\oplus r},n)$ carries a symmetric perfect obstruction theory.
\end{prop}  

We use the notation $\crit(f)$ for the zero scheme $\set{\dd f = 0}$, for $f$ a function on a smooth scheme. The embedding of the Quot scheme inside a non-commutative quiver model had appeared (conjecturally, and in a slightly different language) in the physics literature~\cite{Cir-Sink-Szabo}.

Every critical locus $\crit(f)$ has a canonical symmetric obstruction theory, determined by the Hessian complex attached to the function $f$. It will  be referred to as the \emph{critical obstruction theory} throughout. In the case of $\mathrm{Q} = \Quot_{\BA^3}(\mathscr O^{\oplus r},n)$, this symmetric obstruction theory is the morphism 
\begin{equation}\label{symmetric_POT_quot}
\begin{tikzpicture}[baseline=(current  bounding  box.center)]
\node (N1) at (-1.98,0.95) {$\mathbb E_{\crit}$};
\node (N2) at (-1.37,0.94) {$=$};
\node (N3) at (-1.97,-0.88) {$\BL_{\mathrm{Q}}$};
\node (N4) at (-1.38,-0.88) {$=$};
\node (O1) at (-0.1,0.93) {$\big[T_{\NCQuot_{r}^n}\big|_{\mathrm{Q}}$};
\node (O2) at (2.99,0.93) {$\Omega_{\NCQuot_{r}^n}\big|_{\mathrm{Q}}\big]$};
\node (O3) at (-0.1,-0.88) {$\big[\mathscr I/\mathscr I^2$};
\node (O4) at (2.99,-0.88) {$\Omega_{\NCQuot_{r}^n}\big|_{\mathrm{Q}}\big]$};
\path[commutative diagrams/.cd, every arrow, every label]
(N1) edge node[swap] {$\phi$} (N3)
(O1) edge node {$\mathsf{Hess}(f_n)$} (O2)
(O1) edge node[swap] {$(\dd f_n)^\vee|_{\mathrm{Q}}$} (O3)
(O3) edge node {$\dd$} (O4)
(O2) edge node {$\mathrm{id}$} (O4);
\end{tikzpicture}
\end{equation}
in $\derived^{[-1,0]}(\mathrm{Q})$, where we represented the truncated cotangent complex by means of the exterior derivative $\dd$ constructed out of the ideal sheaf $\mathscr I \subset \OO_{\NCQuot^n_r}$ of the inclusion $\mathrm{Q} \into \NCQuot^n_r$. 

\begin{remark}
As proved by Cazzaniga and the third author in \cite{cazzaniga2020framed}, for any integer $m\geq 3$, the Quot scheme $\Quot_{\BA^m}(\OO^{\oplus r},n)$ is canonically isomorphic to the moduli space of \emph{framed sheaves} on $\BP^m$, i.e.~the moduli space of pairs $(E,\varphi)$ where $E$ is a torsion free sheaf on $\BP^m$ with Chern character $(r,0,\ldots,0,-n)$ and $\varphi\colon E|_D \simto \OO_D^{\oplus r}$ is an isomorphism, for $D \subset \BP^m$ a fixed hyperplane.
\end{remark}

%%%%%%%%%%%%%%%%%%%%%%%%%%%%%%%%%%%%%%%%%%%%%%%%%%%%

\subsection{Torus actions on the local Quot scheme}\label{sec:torus_actions}
In this section we define a torus action on the Quot scheme.
Set
\begin{equation}\label{tori_T1_and_T2}
\BT_1 = (\BC^\ast)^3, \quad \BT_2 = (\BC^\ast)^r,\quad \TT = \BT_1 \times \BT_2.
\end{equation}
The torus $\BT_1$ acts on $\BA^3$ by the standard action
\begin{equation}\label{standard_action}
t\cdot x_i = t_ix_i,
\end{equation}
and this action lifts to an action on $\Quot_{\BA^3}(\mathscr O^{\oplus r},n)$. 
At the same time, the torus $\BT_2 = (\BC^\ast)^r$ acts on the Quot scheme by scaling the fibres of $\OO^{\oplus r}$. Thus we obtain a $\TT$-action on $\Quot_{\BA^3}(\mathscr O^{\oplus r},n)$.

\begin{remark}\label{lemma:compact_fixed_locus_Quot}
The fixed locus $\Quot_{\BA^3}(\OO^{\oplus r},n)^{\BT_1}$ is proper. Indeed, a $\BT_1$-invariant surjection $\OO^{\oplus r}\onto T$ necessarily has the quotient $T$  entirely supported at the origin $0\in \BA^3$. Hence
\[
\Quot_{\BA^3}(\OO^{\oplus r},n)^{\BT_1} \into \Quot_{\BA^3}(\OO^{\oplus r},n)_0
\]
sits inside the \emph{punctual Quot scheme} as a closed subscheme. But $\Quot_{\BA^3}(\OO^{\oplus r},n)_0$ is proper, since it is a fibre of the Quot-to-Chow morphism $\Quot_{\BA^3}(\OO^{\oplus r},n)\to \Sym^n\BA^3$, which is a proper morphism.
\end{remark}

We recall, verbatim from \cite[Lemma 2.10]{BR18}, the description of the full $\TT$-fixed locus induced by the product action on the local Quot scheme.

\begin{lemma}\label{lemma:T_fixed_points}
There is an isomorphism of schemes
\begin{equation}\label{eqn:T-fixedlocus}
\Quot_{\BA^3}(\mathscr O^{\oplus r},n)^{\TT} = \coprod_{n_1+\cdots+n_r=n}\prod_{i=1}^r \Hilb^{n_i}(\BA^3)^{\BT_1}.
\end{equation}
In particular, the $\TT$-fixed locus is isolated and compact. Moreover, letting $\TT_0 \subset \TT$ be the subtorus defined by $t_1t_2t_3 = 1$, one has a scheme-theoretic identity
\begin{equation}\label{eqn:T0_fixed_locus}
\Quot_{\BA^3}(\mathscr O^{\oplus r},n)^{\TT_0} = \Quot_{\BA^3}(\mathscr O^{\oplus r},n)^{\TT}.
\end{equation}
\end{lemma}

\begin{proof}
The main result proved by Bifet in \cite{Bifet} (in greater generality) implies that 
\begin{equation}\label{Bifet_theorem}
\Quot_{\BA^3}(\mathscr O^{\oplus r},n)^{\BT_2} = \coprod_{n_1+\cdots+n_r=n}\prod_{i=1}^r \Hilb^{n_i}(\BA^3).
\end{equation}
The isomorphism \eqref{eqn:T-fixedlocus} follows by taking $\BT_1$-invariants. Since $\Hilb^{k}(\BA^3)^{\BT_1}$ is isolated (a disjoint union of reduced points, each corresponding to a monomial ideal of colength $k$), the first claim follows. Let now $\BT_0 \subset \BT_1$ be the subtorus defined by $t_1t_2t_3=1$, so that $\TT_0 = \BT_0 \times \BT_2$. Equation \eqref{eqn:T0_fixed_locus} follows combining Equation \eqref{Bifet_theorem} and the isomorphism $\Hilb^{k}(\BA^3)^{\BT_1} \cong \Hilb^k(\BA^3)^{\BT_0}$ proved in \cite[Lemma 4.1]{BFHilb}.
\end{proof}

\begin{remark}
The $\TT$-action on $\Quot_{\BA^3}(\OO^{\oplus r},n)$ just described can be seen as the restriction of a $\TT$-action on the larger space $\NCQuot^n_r$. Indeed, a torus element $\mathbf t = (t_1,t_2,t_3,w_1,\ldots,w_r) \in \TT$ acts on $ (A_1,A_2,A_3,u_1,\ldots,u_r) \in \NCQuot^n_r$ by 
\begin{equation}\label{T-action_on_NCQUOT}
\mathbf t\cdot P = (t_1A_1,t_2A_2,t_3A_3,w_1u_1,\ldots,w_ru_r).
\end{equation}
The critical locus $\crit(f_n)$ is $\TT$-invariant, and the induced action is precisely the one we described earlier in this section.
\end{remark}

\subsection{Combinatorial description of the $\TT$-fixed locus}\label{sec: combinatorial description of fixed locus}
The $\TT$-fixed locus $\Quot_{\BA^3}(\OO^{\oplus r},n)^\TT$ is described purely in terms of $r$-\emph{colored plane partitions} of size $n$, as we now explain.

We first recall the definition of a partition of arbitrary dimension.

\begin{definition}\label{def:partitions_arbitrary_dim}
Let $d\geq 1$ and $n \geq 0$ be integers. A $(d-1)$-dimensional partition of $n$ is a collection of $n$ points $\CA=\set{\mathbf a_1,\ldots,\mathbf a_n}\subset \BZ_{\geq 0}^{d}$ with the following property: if $\mathbf a_i = (a_{i1},\ldots,a_{id}) \in \CA$, then whenever a point $\mathbf y = (y_1,\ldots,y_d)\in\BZ_{\geq 0}^{d}$ satisfies $0\leq y_j\leq a_{ij}$ for all $j=1,\ldots,d$, one has $\mathbf y \in \CA$. The integer $n = \lvert \mathcal A \rvert$ is called the \emph{size} of the partition. 
\end{definition}

There is a bijective correspondence between the sets of
\begin{itemize}
    \item [$\circ$] $(d-1)$-dimensional partitions of size $n$,
    \item [$\circ$] $(\BC^\ast)^d$-fixed points of $\Hilb^n(\BA^d)$, and
    \item [$\circ$] monomial ideals $I \subset \BC[x_1,\ldots,x_d]$ of colength $n$.
\end{itemize}

We will be interested in the case $d=3$, corresponding by definition to \emph{plane partitions}. These can be visualised (cf.~Figure \ref{fig:2D_partition}) as configurations of $n$ boxes stacked in the corner of a room (with gravity pointing in the $(-1,-1,-1)$ direction).

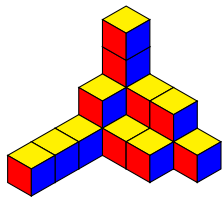
\begin{figure}[ht]
\captionsetup{width=0.68\textwidth}
\centering
\begin{tikzpicture}[scale=0.36] 
\planepartition{{4,2,2,1},{2,1,1},{1},{1},{1}}
\end{tikzpicture}
\caption{\small{A plane partition of size $n=16$.}} \label{fig:2D_partition}
\end{figure}

\begin{example}
If $\mathcal A$ is a $(d-1)$-dimensional partition of size $n$ as in Definition \ref{def:partitions_arbitrary_dim}, the associated monomial ideal is
\[
I_{\mathcal A} = \Braket{x_1^{i_1}\cdots x_d^{i_d}|(i_1,\ldots,i_d) \in \BZ^d_{\geq 0}\setminus \mathcal A}\,\subset\,\BC[x_1,\ldots,x_d].
\]
For instance, if $d=3$, in the case of the plane partition pictured in Figure \ref{fig:2D_partition}, the associated monomial ideal of colength $16$ is generated by the monomials shaping the staircase of the partition, and is thus equal to
\[
\Braket{x_3^4,x_1x_3^2,x_1^2x_3,x_1^5,x_1^2x_2,x_1x_2x_3,x_2x_3^2,x_1x_2^3,x_2^3x_3,x_2^4}\,\subset\, \BC[x_1,x_2,x_3].
\]
\end{example}

Here is an alternative definition of plane partitions.

\begin{definition}\label{def:plane_partition}
A (finite) \emph{plane partition} is a sequence $\pi=\set{\pi_{ij}|i,j\geq 0} \subset \BZ_{\geq 0}$ such that $\pi_{ij}=0$ for $i,j\gg 0$ and
\[ 
\pi_{ij}\geq \pi_{i+1,j}, \quad  \pi_{ij}\geq \pi_{i,j+1}\quad \textrm{for all } i,j\geq 0.
\]
\end{definition}

\begin{definition}\label{def:colored_partition}
An \emph{$r$-colored plane partition} is a tuple $\overline{\pi}=(\pi_1,\ldots,\pi_r)$, where each $\pi_i$ is a plane partition.
\end{definition}

Denote by
\[
|\pi|=\sum_{i,j\geq 0}\pi_{ij}
\]
the \emph{size} of a plane partition (i.e.~the number $n$ in Definition \ref{def:partitions_arbitrary_dim}) and by $|\overline{\pi}|=\sum_{i=1}^r|\pi_i|$ the size of an $r$-colored plane partition.

In the light of Definition \ref{def:plane_partition}, the monomial ideal associated to a plane partition $\pi$ is
\[
I_\pi = \Braket{x_1^ix^j_2x^{\pi_{ij}}_3 | i,j\geq 0}\,\subset\, \BC[x_1,x_2,x_3].
\]
It is clear that the colength of the ideal $I_{\pi}$ is $|\pi|$.

\begin{remark}
A general plane partition may have infinite legs, each shaped by (i.e.~asymptotic to) a standard ($1$-dimensional) partition, or Young diagram. We are not concerned with infinite plane partitions here, since we only deal with quotients $\OO^{\oplus r} \onto T$ with finite support.
\end{remark}

\begin{prop}\label{prop:fixedlocus_indexed_by_colored_partitions}
There is a bijection between $\TT$-fixed points $[S]\in\Quot_{\BA^3}(\OO^{\oplus r},n)^\TT$ and $r$-colored plane partitions $\overline{\pi}$ of size $n$.
\end{prop}

\begin{proof}
For $r = 1$ this is well known: as we recalled above, monomial ideals $I\subset \BC[x_1,x_2,x_3]$ are in bijective correspondence with plane partitions. 
Similarly, to each $r$-colored plane partition $\overline{\pi}=(\pi_1,\ldots,\pi_r)$ there corresponds a subsheaf $S_{\overline{\pi}}=\bigoplus_{i=1}^r I_{\pi_i} \subset \OO^{\oplus r}$. But these are the $\TT$-fixed points of the Quot scheme by Lemma \ref{lemma:T_fixed_points}.
\end{proof}

%%%%%%%%%%%%%%%%%%%%%%%%%%%%%%%%%%%%%%%%%%%%%%%%%%%%%%%
\subsubsection{Computing the trace of a monomial ideal}
Recall the map \eqref{eqn:trace} sending a torus representation to its weight space decomposition. Consider the $3$-dimensional torus $\BT_1$ acting on the coordinate ring $R = \BC[x_1,x_2,x_3]$ of $\BA^3$. Then we have
\[
\tr_{R} = \sum_{\textrm{\ding{114}}\,\,\in\, \BZ_{\geq 0}^3}t^{\textrm{\ding{114}}} =\sum_{(i,j,k) \,\in\, \BZ_{\geq 0}^3}t_1^it_2^jt_3^k = \frac{1}{(1-t_1)(1-t_2)(1-t_3)}.
\]
For a cyclic monomial ideal $\mathfrak m_{abc} = x_1^ax_2^bx_3^c\cdot R \subset R$, one has
\[
\tr_{\mathfrak m_{abc}} = \sum_{i\geq a}\sum_{j\geq b}\sum_{k\geq c} t_1^it_2^jt_3^k = \frac{t_1^at_2^bt_3^c}{(1-t_1)(1-t_2)(1-t_3)}.
\]
More generally, for a monomial ideal $I_{\pi}\subset \BC[x_1,x_2,x_3]$, one has
\begin{equation}\label{eqn:trace_ideal}
    \tr_{I_\pi} = \sum_{(i,j,k)\,\notin\,\pi}t_1^it_2^jt_3^k.
\end{equation}
These are the building blocks needed to compute $\tr_S$ for an arbitrary sheaf $S = \bigoplus_{i=1}^r I_{\pi_i}$ corresponding to a $\TT$-fixed point $[S] \in \Quot_{\BA^3}(\OO^{\oplus r},n)^{\TT}$.

%%%%%%%%%%%%%%%%%%%%%%%%%%%%%%%%%%%%%%%%%%%%%%%%%%%%%%%
\subsection{The \texorpdfstring{$\TT$}{}-fixed obstruction theory}
Recall from \cite[Prop.~1]{GPvirtual} that a torus equivariant obstruction theory on a scheme $Y$ induces a canonical perfect obstruction theory, and hence a virtual fundamental class, on each component of the torus fixed locus. In this subsection we show that the reduced isolated locus
\[
\Quot_{\BA^3}(\OO^{\oplus r},n)^{\TT} \into \Quot_{\BA^3}(\OO^{\oplus r},n)
\]
carries the trivial $\TT$-fixed perfect obstruction theory, so the induced virtual fundamental class agrees with the actual ($0$-dimensional)  fundamental class.

We first need to check the equivariance (cf.~Definition \ref{def:equivariant_pot}) of the critical obstruction theory $\BE_{\crit}$ obtained in Proposition \ref{prop:SPOT_A^3}. In fact, this follows from the general fact that the critical obstruction theory on $\crit(f) \subset Y$, for $f$ a function on a smooth scheme $Y$, acted on by an algebraic torus $\TT$, is naturally $\TT$-equivariant as soon as $f$ is $\TT$-homogeneous. However, for the sake of completeness, we include a direct proof below for the case at hand.

\begin{lemma}\label{lemma:T_equivariance_of_POT}
The critical obstruction theory on $\mathrm{Q} = \Quot_{\BA^3}(\OO^{\oplus r},n)$  is $\TT$-equivariant. 
\end{lemma}

\begin{proof}
We start with two observations:
\begin{enumerate}
    \item The potential $f_n = \Tr A_1[A_2,A_3]$ recalled in \eqref{eqn:superpotential_ncquot} is homogeneous (of degree $3$) in the matrix coordinates of the non-commutative Quot scheme.\label{obs1}
    \item The potential $f_n$ satisfies the relation
    \begin{equation}\label{equivariance_of_f}
    f_n(t\cdot P) = t_1t_2t_3\cdot f_n(P)
    \end{equation}
    for every $t = (t_1,t_2,t_3)\in \BT_1$ and $P\in \NCQuot^n_r$.\label{obs2}
\end{enumerate}
 Fix a point $p \in \mathrm{Q} = \crit (f_n) \subset \NCQuot^n_r$. Then, setting $N = 2n^2+rn = \dim \NCQuot^n_r$, let $x_1,\ldots,x_N$ be local holomorphic coordinates of $\NCQuot^n_r$ around $p$. Let the torus $\TT$ act on these coordinates as prescribed by Equation \eqref{T-action_on_NCQUOT}, i.e.~$t_1$ (resp.~$t_2$ and $t_3$) rescales each $x_k$ corresponding to the entries of the first (resp.~second and third) matrix, and $w_l$ rescales the coordinates of the vector $u_l$, for $l=1,\ldots,r$. Formally, for a matrix coordinate $x_k$, we set
\[
(t_1,t_2,t_3,w_1,\ldots,w_r) \cdot x_k = t_{\ell(k)} x_k
\]
where $\ell(k)\in \set{1,2,3}$ depends on whether $x_k$ comes from an entry of $A_1$, $A_2$ or $A_3$. We also have to prescribe an action on tangent vectors and $1$-forms. For a matrix coordinate $x_k$, we set
\begin{equation}\label{eqn:T-action_forms_vectors}
\begin{split}
    (t_1,t_2,t_3,w_1,\ldots,w_r) \cdot \frac{\partial}{\partial x_k} &= \frac{t_1t_2t_3}{t_{\ell(k)}} \frac{\partial}{\partial x_k}\\
    (t_1,t_2,t_3,w_1,\ldots,w_r) \cdot \dd x_k &= t_{\ell(k)} \dd x_k.
\end{split}
\end{equation}
If $x_k$ comes from a vector component of the $l$-th vector, we set
\begin{equation}\label{T2actiondimerda}
\begin{split}
    (t_1,t_2,t_3,w_1,\ldots,w_r) \cdot \frac{\partial}{\partial x_k} &= w_l^{-1}\frac{\partial}{\partial x_k}\\
    (t_1,t_2,t_3,w_1,\ldots,w_r) \cdot \dd x_k &=w_l \dd x_k.
\end{split}
\end{equation}
However, the $\BT_2$-action \eqref{T2actiondimerda} will be invisible in the Hessian since the function $f_n$ does not touch the vectors.

The Hessian can be seen as a section
\[
\mathsf{Hess}(f_n) \in \Gamma\left(\mathrm{Q},T^\ast_{\NCQuot^n_r}\big|_\mathrm{Q}\otimes T^\ast_{\NCQuot^n_r}\big|_\mathrm{Q}\right).
\]
In checking the equivariance relation
\[
\mathbf t\cdot \mathsf{Hess}(f_n)(\xi) = \mathsf{Hess}(f_n)(\mathbf t\cdot\xi), \quad \mathbf t \in \TT,
\]
we may ignore local coordinates $x_k$ corresponding to vector entries, because the Hessian is automatically equivariant in these coordinates (equivariance translates into the identity $0=0$). 

So, let us fix an $x_k$ coming from one of the matrices. The $(i,j)$-component of the Hessian applied to $\partial / \partial x_k$ is given by
\[
\mathsf{Hess}_{ij}(f_n)\left(\frac{\partial}{\partial x_k} \right) =
\frac{\partial^2 f_n}{\partial x_i\partial x_j}(x_1,\ldots,x_N)\,\dd x_j.
\]
This will vanish unless $k \in \set{i,j}$. Without loss of generality we may assume $k=i$. In this case we obtain, up to a sign convention,
\begin{equation}\label{equivariance_Hess_LHS}
\mathsf{Hess}_{ij}(f_n)\left(\frac{t_1t_2t_3}{t_{\ell(k)}} \frac{\partial}{\partial x_k} \right) = \frac{t_1t_2t_3}{t_{\ell(k)}}
\frac{\partial^2 f_n}{\partial x_k\partial x_j}(x_1,\ldots,x_N)\,\dd x_j.
\end{equation}
On the other hand, combining the observations \eqref{obs1} and \eqref{obs2} with  \eqref{eqn:T-action_forms_vectors}, we obtain
\begin{align*}
    \mathbf t\cdot \mathsf{Hess}_{ij}(f_n)\left(\frac{\partial}{\partial x_k} \right)%&= \mathbf t\cdot \left(\frac{\partial^2 f_n}{\partial x_k\partial x_j}(x_1,\ldots,x_N)\,\dd x_j\right) \\
    &=\frac{\partial^2 f_n}{(\partial t_{\ell(k)}x_k)(\partial t_{\ell(j)}x_j)}(t_{\ell(1)}x_1,\ldots,t_{\ell(N)}x_N)t_{\ell(j)}\dd x_j\\
    &=\frac{t_1t_2t_3}{t_{\ell(k)}t_{\ell(j)}}t_{\ell(j)}\frac{\partial^2 f_n}{\partial x_k\partial x_j}(x_1,\ldots,x_N)\,\dd x_j,%\\
    %&=\frac{t_1t_2t_3}{t_{\ell(k)}}\frac{\partial^2 f_n}{\partial x_k\partial x_j}(x_1,\ldots,x_N)\,\dd x_j,
\end{align*}
which agrees with the right hand side of Equation \eqref{equivariance_Hess_LHS}. Thus we conclude that the Hessian complex is $\TT$-equivariant, as well as the morphism \eqref{symmetric_POT_quot} to the cotangent complex. This finishes the proof.
\end{proof}

The property \eqref{equivariance_of_f} of $f_n$ exhibits the differential $\dd f_n$ as a $\GL_3$-equivariant section
\[
\dd f_n \otimes \mathfrak{t}^{-1} \colon \OO_{\NCQuot^n_r} \to \Omega_{\NCQuot^n_r}\otimes \mathfrak{t}^{-1},
\]
where $\mathfrak{t}^{-1} = (t_1t_2t_3)^{-1}$ is the determinant representation of $\BC^3 = \bigoplus_{1\leq i\leq 3}t_i^{-1}\cdot \BC$. Therefore, explicitly, the morphism in $\derived^{[-1,0]}(\Coh^{\TT}_{\mathrm{Q}})$ lifting the critical obstruction theory \eqref{symmetric_POT_quot} is
\begin{equation}\label{equivariant_symmetric_POT_quot}
\begin{tikzpicture}[baseline=(current  bounding  box.center)]
\node (O1) at (-0.1,0.93) {$ \big[\mathfrak t\otimes T_{\NCQuot_{r}^n}\big|_{\mathrm{Q}}$};
\node (O2) at (2.99,0.93) {$\Omega_{\NCQuot_{r}^n}\big|_{\mathrm{Q}}\big]$};
\node (O3) at (-0.1,-0.88) {$\big[\mathscr I/\mathscr I^2$};
\node (O4) at (2.99,-0.88) {$\Omega_{\NCQuot_{r}^n}\big|_{\mathrm{Q}}\big]$};
\path[commutative diagrams/.cd, every arrow, every label]
(O1) edge node {$\mathsf{Hess}(f_n)$} (O2)
(O1) edge node[swap] {$(\dd f_n)^\vee|_{\mathrm{Q}}$} (O3)
(O3) edge node {$\dd$} (O4)
(O2) edge node {$\mathrm{id}$} (O4);
\end{tikzpicture}
\end{equation}
so that, in particular, the equivariant K-theory class of the virtual tangent bundle attached to the (equivariant) perfect obstruction theory \eqref{equivariant_symmetric_POT_quot} is
\begin{equation}\label{eqn:virtual_tg_quot}
 T_{\mathrm{Q}}^{\vir} = T_{\NCQuot_{r}^n}\big|_{\mathrm{Q}}-\Omega_{\NCQuot_{r}^n}\big|_{\mathrm{Q}}\otimes\mathfrak{t}^{-1} \,\in\,K^0_{\TT}(\mathrm{Q}).
\end{equation}
This fact will be recalled and used in Propositions \ref{prop:Tangent^vir} and \ref{restriction of class in K theory}.

\smallbreak
Lemma \ref{lemma:T_equivariance_of_POT} implies the existence of a ``$\TT$-fixed'' obstruction theory
\begin{equation}\label{T-fixed_pot}
\BE_{\crit}\big|_{\Quot_{\BA^3}(\OO^{\oplus r},n)^{\TT}}^{\fix} \to \BL_{\Quot_{\BA^3}(\OO^{\oplus r},n)^{\TT}}
\end{equation}
on the fixed locus of the Quot scheme. We proved in Lemma \ref{lemma:T_fixed_points} that this fixed locus is $0$-dimensional, isolated and reduced. The next result will imply that the virtual fundamental class induced by \eqref{T-fixed_pot} on the fixed locus agrees with the actual fundamental class.

\begin{prop}\label{prop:tangents_are_T_movable}
Let $[S] \in \Quot_{\BA^3}(\OO^{\oplus r},n)^{\TT}$ be a torus fixed point. The deformations and obstructions of $\Quot_{\BA^3}(\OO^{\oplus r},n)$ at $[S]$ are entirely $\TT$-movable. In particular, the virtual tangent space at $[S]$ can be written as
\begin{equation}\label{tg_vir}
T^{\vir}_S = \BE_{\crit}^\vee\big|^{\mov}_{[S]} = T_{\mathrm{Q}}\big|_{[S]} - \Omega_{\mathrm{Q}}\big|_{[S]} \otimes\mathfrak{t}^{-1} \in K^0_{\TT}(\pt).
\end{equation}
\end{prop}

\begin{proof}
The perfect obstruction theory $\BE_{\crit}$ on $\mathrm{Q}=\Quot_{\BA^3}(\OO^{\oplus r},n)$, made explicit in Diagram \eqref{equivariant_symmetric_POT_quot}, satisfies $\BE_{\crit} \cong \BE_{\crit}^\vee[1]\otimes \mathfrak t$. Its equivariant K-theory class is therefore 
\[
\BE_{\crit} = \Omega_{\mathrm{Q}}-T_{\mathrm{Q}} \otimes \mathfrak t \,\in\, K_0^{\TT}(\mathrm{Q}).
\]
We know by Equation \eqref{eqn:T0_fixed_locus} in Lemma \ref{lemma:T_fixed_points} that no power of $\mathfrak t$ is a weight of $T_{\mathrm Q}|_p$ for any fixed point $p\in \mathrm{Q}^{\TT}$, which implies that
\begin{equation}\label{no_fixed_tangents}
    \left(T_{\mathrm{Q}}\big|_p \otimes \mathfrak t\right)^{\fix} = 0, \quad \Omega_{\mathrm Q}\big|_p^{\fix} = 0.
\end{equation}
\begin{comment}
Every such fixed point inherits a perfect obstruction theory whose associated virtual fundamental class is of dimension at most $\dim \set{p}=0$. Explicitly, the obstruction theory is given by the complex 
\[
\BE_{\crit}\big|_{p}^{\fix},
\]
whose K-theory class is
\[
\Omega_{\mathrm{Q}}\big|_p^{\fix} - (T_{\mathrm{Q}}\big|_p \otimes \mathfrak t)^{\fix} = \Omega_{\mathrm{Q}}\big|_p^{\fix}
\]
thanks to the vanishing \eqref{no_fixed_tangents}.
For this to have rank at most $0$, we must have
\[
\Omega_{\mathrm{Q}}\big|_p^{\fix} = 0.
\]
So there are neither fixed tangents nor fixed obstructions at $p$, i.e.~the tangent and cotangent spaces are entirely $\TT$-movable, as claimed.
\end{comment}
The claim follows.
\end{proof}

\begin{corollary}\label{cor:trivial pot on fixed locus}
There is an identity
\[
\left[\Quot_{\BA^3}(\OO^{\oplus r},n)^{\TT}\right]^{\vir} = \left[\Quot_{\BA^3}(\OO^{\oplus r},n)^{\TT}\right]\in A_0^{\TT}(\Quot_{\BA^3}(\OO^{\oplus r},n)^{\TT}).
\]
\end{corollary}

%%%%%%%%%%%%%%%%%%%%%%%%%%%%%%%%%%%%%%%%%%%%%%%%%%%%%
\section{Invariants attached to the local Quot scheme}\label{sec:invariants}
In this section we introduce cohomological and K-theoretic DT invariants of $\BA^3$, the main object of study of this paper, starting from the Quot scheme $\Quot_{\BA^3}(\OO^{\oplus r},n)$ studied in the previous section. We first need to introduce some notation and terminology.

\subsection{Some notation}
Recall the tori $\BT_1 = (\BC^\ast)^3$ and $\BT_2 = (\BC^\ast)^r$ from \eqref{tori_T1_and_T2}.
We let $t_1,t_2,t_3$ and $w_1,\ldots,w_r$ be the generators of the representation rings $K^0_{\BT_1}(\pt)$ and $K^0_{\BT_2}(\pt)$, respectively. Then one can write the equivariant cohomology rings of $\BT_1$ and $\BT_2$ as
\[
H^\ast_{\BT_1} = \BQ[s_1,s_2,s_3],\quad H^\ast_{\BT_2} = \BQ[v_1,\ldots,v_r],
\]
where $s_i = c_1^{\BT_1}(t_i)$ and $v_j=c_1^{\BT_2}(w_j)$.
For a virtual $\TT$-module $V \in K^0_{\TT}(\pt)$, we let 
\[
\tr_{V} \in \BZ(\!(t_1,t_2,t_3,w_1,\ldots,w_r)\!)
\]
denote its character, i.e.~its decomposition into weight spaces. We denote by $\overline{(\,\cdot\,)}$ the involution defined on $\BZ(\!(t_1,t_2,t_3,w_1,\ldots,w_r)\!)$ by
\[
\overline P(t_1,t_2,t_3,w_1,\ldots,w_r) = P(t_1^{-1},t_2^{-1},t_3^{-1},w_1^{-1},\ldots,w_r^{-1}).
\]

\subsubsection{Twisted virtual structure sheaf}\label{sec: twisted structure sheaf} 

For any scheme $X$ endowed with a perfect obstruction theory $\BE \to \BL_X$, define as in \cite[Def.~3.12]{Fantechi_Gottsche}, the \emph{virtual canonical bundle}
\[
\mathcal{K}_{X,\vir}=\det \BE = \det(T_X^{\vir})^\vee.
\]
This is just $\det E^0 \otimes (\det E^{-1})^\vee$ if $\BE = E^0 - E^{-1} \in K^0(X)$. We will simply write $\mathcal{K}_{\vir}$ when $X$  is clear from the context.

\begin{lemma}\label{prop:there_is_a_square_root}
Let $A$ be a smooth variety equipped with a regular function $f\colon A\to \BA^1$, and let $X= \crit(f) \subset A$ be the critical locus of $f$, with its critical (symmetric) perfect obstruction theory $\BE_{\crit} \to \BL_{X}$. Then $\mathcal{K}_{X,\vir} \in \Pic(X)$ admits a square root, i.e.~there exists a line bundle
\[
\mathcal{K}_{X,\vir}^{\frac{1}{2}}\in \Pic(X) 
\]
whose second tensor power equals $\det \BE_{\crit}$.
\end{lemma}

\begin{proof}
The K-theory class of the critical perfect obstruction theory is
\[
\BE_{\crit} = \Omega_A\big|_X - T_A\big|_{X},
\]
and by definition one has
\[
\mathcal{K}_{X,\vir}=\frac{\det \Omega_A|_X}{\det {T_A}|_{X}}=\frac{\det \Omega_A|_X}{(\det \Omega_A|_X)^{-1}}=K_A\big|_{X}\otimes K_A\big|_{X}.\qedhere
\]
\end{proof}

Let $X$ be a scheme endowed with a perfect obstruction theory, and let $\OO_X^{\vir} \in K_0(X)$ be the induced virtual structure sheaf. Assume the virtual canonical bundle admits a square root.
Following \cite{NO16}, we define the \emph{twisted} (or \emph{symmetrised}) \emph{virtual structure sheaf} as
\[
\widehat{\OO}_X^{\vir}=\OO_X^{\vir}\otimes \mathcal{K}_{X,\vir}^{\frac{1}{2}}.
\]
In case $X$ carries a torus action, we will see in Remark \ref{rmk: weight of twisted sheaf} that $\widehat{\OO}_X^{\vir}$ acquires a canonical weight.

%%%%%%%%%%%%%%%%%%%%%%%%%%%%%%%%%%%%%%%%%%%%%%%%%%%
\subsection{Classical enumerative invariants}
Degree $0$ Donaldson--Thomas invariants of various flavours have been computed in \cite{LEPA,JLI,BFHilb,RGKummer,MNOP2,Virtual_Quot,BBS,cazzaniga2020higher}. This paper is concerned with the general theory of \emph{Quot schemes}, hence in the (virtual) enumeration of $0$-dimensional quotients of locally free sheaves on $3$-folds. 

The naive (topological) Euler characteristic of the Quot scheme $\Quot_{\BA^3}(\mathscr O^{\oplus r},n)$ is computed via the Gholampour--Kool formula \cite[Prop.~2.3]{Gholampour2017}
\[
\sum_{n\geq 0}e(\Quot_{\BA^3}(\mathscr O^{\oplus r},n))q^n = \mathsf M(q)^r,
\]
where $\mathsf M(q) = \prod_{m\geq 1}(1-q^m)^{-m}$
is the MacMahon function, the generating function for the number of plane partitions of non-negative integers.
On the other hand, the Behrend weighted Euler characteristic of the Quot scheme can be computed via the formula 
\begin{equation}\label{eqn:chi_vir_quot_affinespace}
\sum_{n\geq 0}e_{\vir}(\Quot_{\BA^3}(\mathscr O^{\oplus r},n))q^n = \mathsf M((-1)^rq)^r,
\end{equation}
proved in \cite[Cor.~2.8]{BR18}.
For a complex scheme $Z$ of finite type over $\BC$, we have set $e_{\vir}(Z) = e(Z,\nu_Z)$, where $\nu_Z$ is Behrend's constructible function \cite{Beh}. See \cite[Thm.~A]{BR18} for a proof of the analogue of \eqref{eqn:chi_vir_quot_affinespace} for an arbitrary pair $(Y,F)$ consisting of a smooth $3$-fold $Y$ along with a locally free sheaf $F$ on it. It was shown by Toda \cite{Toda2} that, on a projective Calabi--Yau $3$-fold $Y$, the wall-crossing factor in the higher rank DT/PT correspondence is precisely $\mathsf M((-1)^rq)^{re(Y)}$. The relationship between Toda's wall-crossing formula \cite{Toda2} and the Gholampour--Kool's formula for Euler characteristics of Quot schemes on $3$-folds \cite{Gholampour2017} was clarified in \cite{BR18} via a Hall algebra argument.

%%%%%%%%%%%%%%%%%%%%%%%%%%%%%%%%%%%%%%%%%%%%%%%%%%%
\subsection{Virtual invariants of the Quot scheme}\label{subsec:virtual_invariants_quot}
The scheme $\Quot_{\BA^3}(\mathscr O^{\oplus r},n)$ is not proper, but carries a torus action with proper fixed locus. Thus we may define virtual invariants via equivariant residues, by setting
\begin{equation}
\label{def:DT_coh_equivariant_residues}
\int_{[\Quot_{\BA^3}(\mathscr O^{\oplus r},n)]^{\vir}}1\defeq \sum_{[S]\in \Quot_{\BA^3}(\mathscr O^{\oplus r},n)^{\TT}}\frac{1}{e^{\TT}(T_S^{\vir})}\in \mathbb{Q}(\!(s,v)\!),
\end{equation}
where $s=(s_1,s_2,s_3) $ and $v=(v_1,\ldots,v_r)$ are the equivariant parameters of the torus $\TT$ and $T_S^{\vir}$ is the virtual tangent space \eqref{tg_vir}. The sum runs over all $\TT$-fixed points $[S]$, which are isolated, reduced and with the trivial perfect obstruction theory induced from the critical obstruction theory on the Quot scheme (cf.~Corollary \ref{cor:trivial pot on fixed locus}). We refer to these invariants as (degree $0$) \emph{cohomological rank $r$ DT invariants}, as they take value in (an extension of) the fraction field $\mathbb{Q}(s,v)$ of the $\TT$-equivariant cohomology ring $H^\ast_{\TT}$. We will study their generating function
\begin{equation}\label{eqn:cohomological_DT_series}
\DT_r^{\coh}(\BA^3,q,s,v)=\sum_{n\geq 0} q^n\int_{[\Quot_{\BA^3}(\mathscr O^{\oplus r},n)]^{\vir}}1 \in \mathbb{Q}(\!(s,v)\!)\llbracket q \rrbracket
\end{equation}
in \S\,\ref{sec:cohomological_invariants}. On the other hand,
K-theoretic invariants arise as natural refinements of their cohomological counterpart. Naively, one would like to study the virtual holomorphic Euler characteristic
\begin{align*}
\chi(\Quot_{\mathbb{A}^3}(\OO^{\oplus r},n), \OO^{\vir})&=\sum_{[S]\in \Quot_{\BA^3}(\mathscr O^{\oplus r},n)^{\TT}}\tr\left(\frac{1}{\Lambda^\bullet T_S^{\vir,\vee}} \right) \in \BZ(\!(t,w)\!),
\end{align*}
where $t = (t_1,t_2,t_3)$, $w=(w_1,\ldots, w_r)$, and via the trace map $\tr$ we identify a (possibly infinite-dimensional) virtual $\TT$-module with its decomposition into weight spaces. It turns out that guessing a closed formula for these invariants is incredibly difficult and, after all, not what one should look at. Instead, Nekrasov--Okounkov \cite{NO16} teach us that we should focus our attention on
\begin{equation}\label{eqn: K theoretic invariants}
\chi(\Quot_{\mathbb{A}^3}(\OO^{\oplus r},n), \widehat{\OO}^{\vir}) 
=  \sum_{[S]\in \Quot_{\BA^3}(\mathscr O^{\oplus r},n)^{\TT}}\tr\left(\frac{\mathcal{K}_{\vir}^{\frac{1}{2}}}{\Lambda^\bullet T_S^{\vir,\vee}}\right)  \in \BZ(\!(t, (t_1t_2t_3)^{\frac{1}{2}},w)\!),\footnote{In theory all equivariant weights could appear with half powers. However, by Remark \ref{rmk: weight of twisted sheaf},   $(t_1t_2t_3)^{\frac{1}{2}}$ is the only one that truly appears.}
\end{equation}
where the \emph{twisted} \emph{virtual structure sheaf} $\widehat{\OO}^{\vir}$ is defined in \S\,\ref{sec: twisted structure sheaf} --- a square root of the virtual canonical bundle exists by Lemma \ref{prop:there_is_a_square_root} and Proposition \ref{prop:SPOT_A^3}. The generating function of rank $r$ K-theoretic DT invariants
\begin{equation}\label{eqn:K_partition_function}
\DT^{\KK}_r(\BA^3,q,t,w)=\sum_{n\geq 0} q^n\chi(\Quot_{\mathbb{A}^3}(\OO^{\oplus r},n), \widehat{\OO}^{\vir})\in \BZ(\!(t, (t_1t_2t_3)^{\frac{1}{2}},w)\!)\llbracket q \rrbracket
\end{equation}
will be studied in \S\,\ref{sec:K-theory}.

\begin{remark}\label{rmk: weight of twisted sheaf}
To be precise, we should replace the torus $\TT$ with its double cover $\TT_{\mathfrak t}$, the minimal cover  of $\TT$ where the character  $\mathfrak{t}^{-1/2}$ is defined, as in \cite[\S\,7.1.2]{NO16}. Then $\mathcal K_{\vir}^{1/2}$ is a $\TT_{\mathfrak t}$-equivariant sheaf with character $\mathfrak{t}^{-(\dim \NCQuot_{r}^n)/2}$. To ease the notation, we keep denoting the torus acting as $\TT$.
\end{remark}

\begin{remark}\label{rmk:independence_of_orientation}
As remarked in \cite{NO16,arbesfeld2019ktheoretic},  choices of square roots of $\mathcal{K}_{\vir}$ differ by a 2-torsion element in the Picard group, which implies that  $\chi(\Quot_{\mathbb{A}^3}(\OO^{\oplus r},n), \widehat{\OO}^{\vir})$ does not depend on such choices of square roots. Thus there is no ambiguity in \eqref{eqn:K_partition_function}.
\end{remark}

%%%%%%%%%%%%%%%%%%%%%%%%%%%%%%%%%%%%%%%%%%%%%%%%%%%%%%%%%%%%%%%%%%
%%%%%%%%%%%%%%%%%%%%%%%%%%%%%%%%%%%%%%%%%%%%%%%%%%%%%%%%%%%%%%%%%%
\section{Higher rank vertex on the local Quot scheme}\label{sec:higher_rank_vertex}

%%%%%%%%%%%%%%%%%%%%%%%%%%%%%%%%%%%%%%%%%%%%%%%%%%%%%%%%%%%%%%%%%%
\subsection{The virtual tangent space of the local Quot scheme}\label{sec: virtual tangent space}
By Lemma \ref{lemma:T_fixed_points}, we can represent the sheaf corresponding to a $\TT$-fixed point 
\[
[S] \in \Quot_{\BA^3}(\OO^{\oplus r},n)^{\TT}
\]
as a direct sum of ideal sheaves
\[
S = \bigoplus_{\alpha=1}^r\mathscr I_{Z_\alpha} \,\subset\,\OO^{\oplus r},
\]
with $Z_\alpha \subset \BA^3$ a finite subscheme of length $n_\alpha$ supported at the origin, and satisfying $n = \sum_{1\leq \alpha\leq r} n_\alpha$. In this subsection we derive a formula for the character of
\[
\chi(\OO^{\oplus r},\OO^{\oplus r})-\chi(S,S) \in K^0_{\TT}(\pt),
\]
where for $F$ and $G$ in $K^0_{\TT}(\pt)$, we set 
\[
\chi(F,G) =\mathbf{R}\Hom(F,G)=\sum_{i\geq 0}(-1)^{i}\Ext^{i}(F,G).
\]
Our method follows the approach of \cite[\S~4.7]{MNOP1}. Moreover, we show in Proposition \ref{prop:Tangent^vir} that such character agrees with the virtual tangent space $T_S^{\vir}$ induced by the critical obstruction theory.

\smallbreak
Let $\mathsf Q_\alpha$ be the $\BT_1$-character of the  $\alpha$-summand of $\OO^{\oplus r}/S$, i.e.~(cf.~Equation \eqref{eqn:trace_ideal})
\[
\mathsf Q_\alpha = \tr_{\OO_{Z_\alpha}}= \sum_{(i,j,k) \,\in\, \pi_\alpha}t_1^{i}t_2^{j}t_3^{k},
\]
where $\pi_\alpha \subset \BZ_{\geq 0}^3$ is the plane partition corresponding to the monomial ideal $\mathscr I_{Z_\alpha} \subset R = \BC[x,y,z]$.
Let $\mathsf P_\alpha(t_1,t_2,t_3)$ be the Poincar\'e polynomial of $\mathscr I_{Z_\alpha}$. This can be computed via a $\BT_1$-equivariant free resolution
\[
0 \to E_{\alpha,s} \to \cdots \to E_{\alpha,1} \to E_{\alpha,0} \to \mathscr I_{Z_\alpha} \to 0.
\]
Writing 
\[
E_{\alpha,i} = \bigoplus_{j}R(d_{\alpha,ij}),\quad d_{\alpha,ij} \in \BZ^3,
\]
one has, independently of the chosen resolution, the formula
\[
\mathsf P_\alpha(t_1,t_2,t_3) = \sum_{i,j}(-1)^it^{d_{\alpha,ij}}.
\]
By \cite[\S\,4.7]{MNOP1} we know that there is an identity
\begin{equation}\label{eqn_Q_and_P}
\mathsf Q_\alpha = \frac{1+\mathsf P_\alpha}{(1-t_1)(1-t_2)(1-t_3)}.
\end{equation}
For each $1\leq \alpha,\beta\leq r$, we can compute
\begin{align*}
\chi(\mathscr I_{Z_\alpha},\mathscr I_{Z_\beta}) &= \sum_{i,j,k,l}(-1)^{i+k} \Hom_R(R(d_{\alpha,ij}),R(d_{\beta,kl})) \\
&=\sum_{i,j,k,l}(-1)^{i+k} R(d_{\beta,kl}-d_{\alpha,ij}),
\end{align*}
which immediately yields the identity
\[
\tr_{\chi(\mathscr I_{Z_\alpha},\mathscr I_{Z_\beta})} = \frac{\mathsf P_\beta \overline{\mathsf P}_\alpha}{(1-t_1)(1-t_2)(1-t_3)} \in \BZ(\!(t_1,t_2,t_3)\!).
\]
It follows that, as a $\TT$-representation, one has
\begin{align*}
    \chi(S,S) 
    &= \chi\left(\sum_\alpha w_\alpha\otimes \mathscr I_{Z_\alpha},\sum_\beta w_\beta\otimes \mathscr I_{Z_\beta}\right) \\
    &= \sum_{1\leq \alpha,\beta\leq r}\chi(w_\alpha\otimes \mathscr I_{Z_\alpha},w_\beta\otimes \mathscr I_{Z_\beta}),
\end{align*}
which yields
\begin{equation}\label{eqn:trace_chi_FF}
\tr_{\chi(S,S)} = \sum_{1\leq \alpha,\beta\leq r} \frac{w_\alpha^{-1}w_\beta\cdot \mathsf P_\beta \overline{\mathsf P}_\alpha}{(1-t_1)(1-t_2)(1-t_3)}.
\end{equation}
On the other hand,
\begin{equation}\label{eqn:trace_chi_OO}
\tr_{\chi(\OO^{\oplus r},\OO^{\oplus r})} = \sum_{1\leq \alpha,\beta\leq r} \frac{w_\alpha^{-1}w_\beta}{(1-t_1)(1-t_2)(1-t_3)}.
\end{equation}
Combining Formulae \eqref{eqn:trace_chi_FF} and \eqref{eqn:trace_chi_OO} with Formula \eqref{eqn_Q_and_P} yields the following result.

\begin{prop}\label{prop:character_virtual_tangent}
Let $[S] \in \Quot_{\BA^3}(\OO^{\oplus r},n)^{\TT}$ be a torus fixed point.
There is an identity
\begin{equation}\label{eqn:trace_T_vir}
\tr_{\chi(\OO^{\oplus r},\OO^{\oplus r})-\chi(S,S)} = \sum_{1\leq \alpha,\beta\leq r}w_\alpha^{-1}w_\beta\left(\mathsf Q_\beta-\frac{\overline{\mathsf Q}_\alpha}{t_1t_2t_3}+\frac{(1-t_1)(1-t_2)(1-t_3)}{t_1t_2t_3}\mathsf Q_\beta \overline{\mathsf Q}_\alpha \right)
\end{equation}
in $\BZ(\!(t_1,t_2,t_3,w_1,\ldots,w_r)\!)$.
\end{prop}

For every $\TT$-fixed point $[S]$ we define associated \emph{vertex} terms
\begin{equation}\label{eqn:vertex_terms}
\mathsf V_{ij}=w_i^{-1}w_j\left(\mathsf Q_j-\frac{\overline{\mathsf Q}_i}{t_1t_2t_3}+\frac{(1-t_1)(1-t_2)(1-t_3)}{t_1t_2t_3}\mathsf Q_j \overline{\mathsf Q}_i \right)
\end{equation}
for every $i,j=1,\ldots, r$. It is immediate to see that for $r=1$ (forcing $i=j$) we recover the vertex formalism developed in \cite{MNOP1}.

Proposition \ref{prop:character_virtual_tangent} can then be restated as
\[
\tr_{\chi(\OO^{\oplus r},\OO^{\oplus r})}-\tr_{\chi(S,S)} = \sum_{1\leq i,j\leq r} \mathsf V_{ij}.
\]
We now relate this to the virtual tangent space (cf.~\S\,\ref{subsec:Nvir}) $T_{S}^{\vir}$ of a point $[S] \in \Quot_{\BA^3}(\OO^{\oplus r},n)^{\TT}$.

\begin{prop}\label{prop:Tangent^vir}
Let $[S] \in \Quot_{\BA^3}(\OO^{\oplus r},n)^{\TT}$ be a $\TT$-fixed point. Let $T^{\vir}_S = \BE_{\crit}^\vee\big|_{[S]}$ be the virtual tangent space induced by the $\TT$-equivariant critical obstruction theory. Then there are identities
\begin{align*}
T^{\vir}_S &= \chi(\OO^{\oplus r},\OO^{\oplus r}) - \chi(S,S)=\sum_{1\leq i,j\leq r} \mathsf V_{ij}\in K_0^{\TT}(\pt).
\end{align*}
\end{prop}

\begin{proof}
Let $\NCQuot_{r}^n$ be the non-commutative Quot scheme defined in \eqref{def:NCquot}. The superpotential $f_n\colon \NCQuot_{r}^n \to \BA^1$ defined in \eqref{eqn:superpotential_ncquot} is equivariant with respect to the character $(t_1,t_2,t_3)\mapsto t_1t_2t_3$, so it gives rise to a $\GL_3$-equivariant section
\begin{equation}\label{GL3_equiv_section}
\OO_{\NCQuot_{r}^n} \xrightarrow{\dd f_n \otimes \mathfrak{t}^{-1}} \Omega_{\NCQuot_{r}^n}\otimes \mathfrak{t}^{-1}
\end{equation}
where, starting from the representation $\BC^3 = \bigoplus_{1\leq i\leq 3}t_i^{-1}\cdot \BC \in K^0_{\BT_1}(\pt)$, we have set
\[
\mathfrak{t}^{-1} = \det \BC^3 = (t_1t_2t_3)^{-1}.
\]
Here, and throughout this proof, we are identifying a representation with its own character via the isomorphism \eqref{eqn:trace}. The zero locus of the section \eqref{GL3_equiv_section} is our Quot scheme
\[
\mathrm Q=\Quot_{\BA^3}(\OO^{\oplus r},n),
\]
endowed with the $\TT$-action described in \S\,\ref{sec:torus_actions}.
According to Equation \eqref{eqn:virtual_tg_quot}, the virtual tangent space computed with respect to the critical $\TT$-equivariant obstruction theory on $\mathrm Q$ is
\begin{equation}\label{tg_vir_definition}
T^{\vir}_S = \left(T_{\NCQuot_{r}^n} - \Omega_{\NCQuot_{r}^n}\otimes \mathfrak{t}^{-1}\right)\big|_{[S]}.
\end{equation}
The tangent space to the smooth scheme $\NCQuot_{r}^n$ can be written, around a point $S \into \OO^{\oplus r} \onto V$, as
\begin{equation}\label{eqn:tangents_NCQUOT}
T_{\NCQuot_{r}^n}\big|_{[S]} = (\BC^3-1)\otimes (\overline{V}\otimes V) + \bigoplus_{\alpha=1}^r \Hom(w_\alpha \BC,V),
\end{equation}
where 
\[
\bigoplus_{\alpha=1}^r \Hom(w_\alpha \BC,V) = \bigoplus_{\alpha=1}^r w_\alpha^{-1}\otimes V
\]
represents the $r$ framings on the $3$-loop quiver. Let $V$ be written as a direct sum of structure sheaves
\[
V = \bigoplus_{\alpha=1}^r \OO_{Z_\alpha}, 
\]
where the $\alpha$-summand has $\TT$-character $w_\alpha \mathsf Q_\alpha$.
Substituting
\begin{align*}
    \BC^3 -1 &= t_1^{-1}+t_2^{-1}+t_3^{-1} - 1 = \frac{t_1t_2+t_1t_3+t_2t_3-t_1t_2t_3}{t_1t_2t_3}\\
    V &= \sum_{\alpha=1}^r w_\alpha \mathsf Q_\alpha 
\end{align*}
into Formula \eqref{eqn:tangents_NCQUOT} yields
\[
T_{\NCQuot_{r}^n}\big|_{[S]} = \frac{t_1t_2+t_1t_3+t_2t_3-t_1t_2t_3}{t_1t_2t_3} \sum_{1\leq \alpha,\beta\leq r} w_\alpha^{-1}w_\beta \overline{\mathsf Q}_\alpha \mathsf Q_\beta+\sum_{1\leq \alpha,\beta\leq r}w_\alpha^{-1}w_\beta \mathsf Q_\beta,
\]
and hence
\[
\left(\Omega_{\NCQuot_{r}^n}\otimes \mathfrak{t}^{-1}\right) \big|_{[S]} =
\frac{t_1+t_2+t_3-1}{t_1t_2t_3} \sum_{1\leq \alpha,\beta\leq r} w_\alpha^{-1}w_\beta \overline{\mathsf Q}_\alpha \mathsf Q_\beta + \sum_{1\leq \alpha,\beta\leq r}w_\alpha^{-1}w_\beta \frac{\overline{\mathsf Q}_\alpha}{t_1t_2t_3},
\]
which by Formula \eqref{tg_vir_definition} yields
\[
T^{\vir}_S = \sum_{1\leq \alpha,\beta\leq r} w_\alpha^{-1}w_\beta\left(\mathsf Q_\beta - \frac{\overline{\mathsf Q}_\alpha}{t_1t_2t_3} + \frac{(1-t_1)(1-t_2)(1-t_3)}{t_1t_2t_3}\overline{\mathsf Q}_\alpha \mathsf Q_\beta \right).
\]
The right hand side is shown to be equal to $\chi(\OO^{\oplus r},\OO^{\oplus r}) - \chi(S,S)$ in Proposition \ref{prop:character_virtual_tangent}.
\end{proof}

\subsection{A small variation of the vertex formalism}\label{sec: variation vertex}
All locally free sheaves on $\mathbb{A}^3$ are trivial, but this is not true equivariantly. For example, we have $K_{\mathbb{A}^3}=\OO_{\mathbb{A}^3}\otimes t_1t_2t_3\in K^0_{\BT_1}({\mathbb{A}^3})$, even though the ordinary canonical bundle is trivial. Consider 
\begin{equation}\label{eqn:sheaf_with_equiv_weights}
F=\bigoplus_{i=1}^r \OO_{\mathbb{A}^3}\otimes \lambda_i\in K^0_{\BT_1}({\mathbb{A}^3})
\end{equation}
where $\lambda = (\lambda_i)_i$ are weights of the $\BT_1$-action, i.e.~monomials in the representation ring of $\BT_1$. Let $[S] \in \Quot_{\BA^3}(F,n)^{\TT}$ be a $\TT$-fixed point. It decomposes as $S = \bigoplus_{i=1}^r\mathscr I_{Z_i}\otimes \lambda_i\in K_{\BT_1}^0 (\BA^3 )$, where the weights $\lambda_i$ are naturally inherited from $F$. This generalises the discussion in \S\,\ref{sec: virtual tangent space}, which can be recovered by setting all weights $\lambda_i$ to be trivial. Just as in Proposition \ref{prop:Tangent^vir}, we can compute
\[
T^{\vir}_{S,\lambda}=\chi(F,F)-\chi(S,S)\in  K_{\TT}^0 (\pt).
\]
We find
 \begin{align*}
     T^{\vir}_{S,\lambda}&=\chi\left(\bigoplus_{i=1}^r \OO_{\mathbb{A}^3}\otimes \lambda_i w_i,\bigoplus_{j=1}^r \OO_{\mathbb{A}^3}\otimes \lambda_j w_j \right)-\chi\left(\bigoplus_{j=1}^r \mathscr I_{Z_i}\otimes \lambda_i w_i, \bigoplus_{i=1}^r\mathscr I_{Z_j}\otimes \lambda_j w_j \right).
   \end{align*}

Therefore we derive the same expression for the vertex formalism as before, just substituting $w_i$ with $\lambda_i w_i$. 
This variation will be crucial in \S\,\ref{section on compact invariants} where we indentify the Quot scheme of the local model with the restriction of $\Quot_X(F,n)$ to an open toric chart, with $X$ a toric projective 3-fold and $F$ a $\BT_1$-equivariant exceptional locally free  sheaf. 

Define, for $\lambda = (\lambda_1,\ldots,\lambda_r)$ as above and $F$ as in \eqref{eqn:sheaf_with_equiv_weights}, the equivariant integral
\begin{equation}\label{eqn:lambda_invariants}
\int_{[\Quot_{\BA^3}( F,n)]^{\vir}}1\defeq \sum_{[S]\in \Quot_{\BA^3}(\mathscr O^{\oplus r},n)^{\TT}}\frac{1}{e^{\TT}(T_{S,\lambda}^{\vir})}\in \mathbb{Q}(\!(s,v)\!),
\end{equation}
and let 
\begin{equation}\label{eqn:cohomological_DT_series_lambda}
\DT_r^{\coh}(\BA^3,q,s,v)_\lambda = \sum_{n\geq 0}q^n \int_{[\Quot_{\BA^3}( F,n)]^{\vir}}1
\end{equation}
be the generating function of the invariants \eqref{eqn:lambda_invariants}.
We shall see (cf.~Corollary \ref{cor:independence_on_lambda_w}) that this expression does not depend on the equivariant weights $\lambda_i$.

\section{The higher rank K-theoretic DT partition function}\label{sec:K-theory}

\subsection{Symmetrised exterior algebras and brackets}\label{sec: preliminaries K theory}
We recall some constructions is equivariant K-theory which will be used to prove Theorem \ref{mainthm:K-theoretic}. For a recent and more complete reference, the reader may consult \cite[\S~2]{Okounkov_Lectures}.

Let $\TT$ be a torus, $V=\sum_\mu t^\mu$ a $\TT$-module. Assume that $\det(V)$ is a  square in $K_{\TT}^0(\pt)$. Define the \emph{symmetrised exterior algebra} of $V$ as
\[
\widehat{\Lambda}^\bullet V \defeq \frac{\Lambda^\bullet V}{\det(V)^{\frac{1}{2}}}.
\]
It satisfies the relation
\[ 
\widehat{\Lambda}^\bullet V^{\vee}=(-1)^{\rk V}\widehat{\Lambda}^\bullet V.
\]
Define the operator $[\,\cdot\,]$ by
\[
[t^{\mu}]=t^{\frac{\mu}{2}}-t^{-\frac{\mu}{2}}.
\]
One can compute
\[
\tr( \widehat{\Lambda}^\bullet V^\vee)=\prod_{\mu}\frac{1-t^{-\mu}}{t^{-\frac{\mu}{2}}}=\prod_{\mu}\,(t^{\frac{\mu}{2}}-t^{-\frac{\mu}{2}})=\prod_{\mu}\,[t^{\mu}],
\]
For a virtual $\TT$-representation $V=\sum_{\mu}t^{\mu} -\sum_{\nu}t^{\nu}$, where the weight $\nu=0$ never appears, we extend $  \widehat{\Lambda}^\bullet$ and $[\,\cdot\,]$ by linearity and  find
%Extending $  \widehat{\Lambda}^\bullet$ and $[\,\cdot\,]$ by linearity to the whole of  $K^0_{\TT}(\pt)$, for any virtual $\TT$-representation $V=\sum_{\mu}t^{\mu} -\sum_{\nu}t^{\nu}$ we find
\[
\tr( \widehat{\Lambda}^\bullet V^\vee)=\frac{\prod_{ \mu}[t^{\mu}]}{\prod_{ \nu}[t^{\nu}]}.
\]

\subsection{Proof of Theorem \ref{mainthm:K-theoretic}}\label{sec:proof_of_K-theoretic_thm}
By the description of the $\TT$-fixed locus $\Quot_{\mathbb{A}^3}(\OO^{\oplus r},n)^{\TT}$ given in \S\,\ref{sec: combinatorial description of fixed locus},  every colored plane partition $ \overline{\pi}=(\pi_1,\dots, \pi_r)$ corresponds to a unique $\TT$-fixed point $S = \bigoplus_{i=1}^r\mathscr I_{Z_i}$, for which we defined in Equation \eqref{eqn:vertex_terms} the vertex terms  $\mathsf V_{ij}$ by
\[
\mathsf V_{ij}=w_i^{-1}w_j\left(\mathsf Q_j-\frac{\overline{\mathsf Q}_i}{t_1t_2t_3}+\frac{(1-t_1)(1-t_2)(1-t_3)}{t_1t_2t_3}\mathsf Q_j \overline{\mathsf Q}_i \right)
\]
with notation as in \S\,\ref{sec: virtual tangent space}. The generating function \eqref{eqn:cohomological_DT_series} of higher rank cohomological DT invariants can be rewritten in a purely combinatorial fashion as
\begin{align*}
    \DT_r^{\coh}(\BA^3, q,s,v)=\sum_{\overline{\pi}} q^{\lvert \overline \pi\rvert} \prod_{i,j=1}^r e^{\TT}(-\mathsf V_{ij}).
\end{align*}
Similarly, the generating function \eqref{eqn:K_partition_function} of the K-theoretic invariants can be rewritten as
\begin{align*}
    \DT_r^{\KK}(\BA^3, q,t,w)=\sum_{\overline{\pi}} q^{\lvert \overline \pi\rvert} \prod_{i,j=1}^r [-\mathsf V_{ij}].
\end{align*}

A closed formula for $\DT_1^{\KK}(\BA^3, q,t,w)$ was conjectured in \cite{NEKRASOV2005261} and  has recently been proven by Okounkov.
\begin{theorem}[{\cite[Thm.~3.3.6]{Okounkov_Lectures}}]\label{thm: Okounkov rank 1}
The rank $1$ K-theoretic DT partition function of $\BA^3$ is given by
  \begin{align*}
    \DT_1^{\KK}(\BA^3,-q,t,w)=\Exp\left(\mathsf F_1( q,t_1,t_2,t_3)\right)
\end{align*}
where, setting $\mathfrak{t}=t_1t_2t_3$, one defines
\[
\mathsf F_1( q,t_1,t_2,t_3)= \frac{1}{[\mathfrak{t}^{\frac{1}{2}}q ][\mathfrak{t}^{\frac{1}{2}} q^{-1}]}\frac{[t_1t_2][t_1t_3][t_2t_3]}{[t_1][t_2][t_3]}.
\]
\end{theorem}

\begin{remark}\label{rmk:trivial_framing_action}
It is clear from the expression of the vertex in rank 1 that there is no dependence on $w_1$. As pointed out to us by N.~Arbesfeld, this can in fact be seen as a shadow of the fact that $\BT_2 = \BC^\ast$ acts trivially on $\NCQuot^n_1$ and on $\dd f_n$. 
Surprisingly, the same phenomenon occurs in the higher rank case (cf.~Theorem \ref{thm:independence}). 
\end{remark}
We devote the rest of this section to proving a generalisation of Theorem \ref{thm: Okounkov rank 1} to higher rank.
\begin{theorem}\label{thm: K theoretic of points higher rank}
The rank $r$ K-theoretic DT partition function of $\BA^3$ is given by
 \begin{align*}
    \DT_r^{\KK}(\BA^3,(-1)^r q,t,w)=\Exp\left(\mathsf F_r(  q,t_1,t_2,t_3)\right),
\end{align*}
where, setting $\mathfrak t = t_1t_2t_3$, one defines
\[ 
\mathsf F_r( q,t_1,t_2,t_3)= \frac{[\mathfrak{t}^r]}{[\mathfrak{t}][\mathfrak{t}^{\frac{r}{2}}q ][\mathfrak{t}^{\frac{r}{2}} q^{-1} ]}\frac{[t_1t_2][t_1t_3][t_2t_3]}{[t_1][t_2][t_3]}.
\]
\end{theorem}

\begin{remark}
This result was conjectured in \cite{MR2545054} by Awata and Kanno, who also proved it $\bmod q^4$, i.e.~up to $3$ instantons. The conjecture was confirmed numerically up to some order by Benini--Bonelli--Poggi--Tanzini \cite{BBPT}.
\end{remark}

The proof of Theorem \ref{thm: K theoretic of points higher rank} will follow essentially by taking suitable limits of the weights $w_i$. To perform such limits, we prove the slogan \eqref{independence_slogan}, already anticipated   in the Introduction.

\begin{theorem}\label{thm:independence}
The generating function $\DT_r^{\KK}(\BA^3,q,t,w)$ is independent of the weights $w_1,\dots,w_r$.
\end{theorem}

\begin{proof}
The $n$-th coefficient of $\DT^{\KK}_r(\BA^3,q,t,w)$ is a sum of contributions
\[
[-T_{\overline{\pi}}^{\vir}],\qquad \lvert \overline{\pi}\rvert = n.
\]
A simple manipulation shows that
\begin{equation}\label{eqn:poles_showing_up}
[-T_{\overline{\pi}}^{\vir}] = A(t)\prod_{1\leq i<j\leq r}\frac{\prod_{\mu_{ij}}w_i-w_j t^{\mu_{ij}}}{\prod_{\nu_{ij}}w_i-w_j t^{\nu_{ij}}} = A(t)\prod_{1\leq i<j\leq r}\frac{\prod_{\mu_{ij}}(1-w_i^{-1}w_j t^{\mu_{ij}})}{\prod_{\nu_{ij}}(1-w_i^{-1}w_j t^{\nu_{ij}})},
\end{equation}
where $A(t)\in \BQ(\!(t_1,t_2,t_3, (t_1t_2t_3)^{\frac{1}{2}})\!) $ and the number of weights $\mu_{ij}$ and $\nu_{ij}$ is the same. Thus, $\DT^{\KK}_r(\BA^3,q,t,w)$ is a homogeneous rational expression of total degree 0 with respect to the variables $w_1,\ldots,w_r$. We aim to show that $\DT^{\KK}_r(\BA^3,q,t,w)$ has  no poles of the form $1-w_i^{-1}w_j t^{\nu_{ij}}$, implying that it is a degree $0$ polynomial in the $w_i$, hence constant in the $w_i$. This generalises the strategy of \cite[\S\,4]{MR2545054}.

Set $\mathsf w =  w_i^{-1}w_jt^\nu$ for fixed $i<j$ and $\nu \in \widehat{\BT}_1$. To see that $1-\mathsf w$ is not a pole, we use \cite[Prop.~3.2]{arbesfeld2019ktheoretic}, which asserts the following: if $M$ is a quasiprojective $\TT$-scheme with a $\TT$-equivariant perfect obstruction theory and proper (nonempty) fixed locus, then for any $V \in K_0^{\TT}(M)$, the only poles of the form $(1-\mathsf w)$ that \emph{may} appear in $\chi(M,V \otimes \OO^{\vir})$ arise from \emph{noncompact weights} $\mathsf w \in \widehat{\TT}$. A weight $\mathsf w$ is called compact if the fixed locus $M^{\TT_{\mathsf w}} \subset M$ is proper, where $\TT_{\mathsf w}$ is the maximal torus in $\ker(\mathsf w) \subset \TT$, and is called noncompact otherwise \cite[Def.~3.1]{arbesfeld2019ktheoretic}.

We of course want to apply \cite[Prop.~3.2]{arbesfeld2019ktheoretic} to $M = \mathrm{Q} = \Quot_{\BA^3}(\OO^{\oplus r},n)$, $\TT = \BT_1\times \BT_2$ and $V = \mathcal K_{\vir}^{1/2}$. By Equation \eqref{eqn:poles_showing_up}, our goal is to prove that $\mathsf w = w_i^{-1}w_jt^\nu$ is a compact weight for all $i<j$ and $\nu \in \widehat{\BT}_1$.

First of all, we observe that $\TT_{\mathsf w} = \ker (\mathsf w)$. Indeed $\mathsf w$ is not a product of powers of weights of $\TT$, hence 
\[
\OO(\ker \mathsf w) = \OO(\TT)/(w_i^{-1}w_jt^\nu-1) \cong \BC\left[t_1^{\pm 1},t_2^{\pm 1},t_3^{\pm 1},w_1^{\pm 1},\ldots,w_{i-1}^{\pm 1},w_{i+1}^{\pm 1},\ldots,w_r^{\pm 1}\right],
\]
which shows that $\ker(\mathsf w)$ is itself a torus (of dimension $3+r-1$). Next, consider the automorphism $\tau_\nu\colon \TT \simto \TT$ defined by
\[
(t_1,t_2,t_3,w_1,\ldots,w_r) \mapsto (t_1,t_2,t_3,w_1,\ldots,w_it^{-\nu},\ldots,w_j,\ldots,w_r).
\]
It maps $\TT_{\mathsf w}\subset \TT$ isomorphically onto the subtorus $\BT_1 \times \set{w_i=w_j} \subset \TT$. This yields an inclusion of tori
\begin{equation}\label{inclusions_tori}
\BT_1 \simto \BT_1 \times \Set{(1,\ldots,1)} \into \tau(\TT_{\mathsf w}).
\end{equation}
We consider the action $\sigma_\nu \colon \TT \times \mathrm{Q} \to \mathrm{Q}$ where $\BT_1$ translates the support of the quotient sheaf in the usual way, the $i$-th summand of $\OO^{\oplus r}$ gets scaled by $w_it^\nu$ and all other summands by $w_k$ for $k\neq i$. In other words, in terms of the matrix-and-vectors description of $\mathrm{Q}$, we set
\[
\sigma_\nu(\mathbf t,(A_1,A_2,A_3,u_1,\ldots,u_r)) = (t_1A_1,t_2A_2,t_3A_3,w_1u_1,\ldots,w_it^\nu u_i,\ldots,w_ru_r),
\]
just a variation of Equation \eqref{T-action_on_NCQUOT} in the $i$-th vector component. Then, upon restricting this action to $\TT_{\mathsf w}$, we have a commutative diagram
\[
\begin{tikzcd}[row sep=large]
\TT_{\mathsf w} \times \mathrm{Q}\arrow{r}{\sigma}\arrow[swap]{d}{\tau_\nu\times\id} & \mathrm{Q}\arrow[equal]{d} \\
\tau_\nu(\TT_{\mathsf w}) \times \mathrm{Q}\arrow{r}{\sigma_\nu} & \mathrm{Q}
\end{tikzcd}
\]
where $\sigma$ is the restriction of the usual action \eqref{T-action_on_NCQUOT}. This diagram induces a natural isomorphism $\mathrm{Q}^{\TT_{\mathsf w}}\simto\mathrm{Q}^{\tau_\nu(\TT_{\mathsf w})}$, which combined with \eqref{inclusions_tori}
yields an inclusion
\[
\mathrm{Q}^{\TT_{\mathsf w}}\simto\mathrm{Q}^{\tau_\nu(\TT_{\mathsf w})} \into \mathrm{Q}^{\BT_1},
\]
where $\mathrm{Q}^{\BT_1}$ is the fixed locus with respect to the action $\sigma_\nu$. But by the same reasoning as in Remark \ref{lemma:compact_fixed_locus_Quot}, this fixed locus is proper (because, again, a $\BT_1$-fixed surjection $\OO^{\oplus r}\onto T$ necessarily has the quotient $T$ entirely supported at the origin $0 \in \BA^3$). Thus $\mathsf w$ is a compact weight, and the result follows.
\end{proof}

%\begin{comment}
\begin{remark}\label{rmk:rigidity_principle}
%We devote Appendix \ref{sec:framing_independence} to the proof of this crucial fact. 
After a first draft of this work was already finished, we were informed of an alternative way to prove Theorem \ref{thm:independence}, which, in a nutshell, goes as follows: one exploits the (proper) Quot-to-Chow morphism $\Quot_{\BA^3}(\OO^{\oplus r},n) \to \Sym^n \BA^3$ to express the K-theoretic DT invariants as equivariant holomorphic Euler characteristics on $\Sym^n \BA^3$, where the framing torus $\BT_2$ is acting trivially on the symmetric product. One concludes by an application of Okounkov's \emph{rigidity principle} \cite[\S\,2.4.1]{Okounkov_Lectures}. This strategy will be carried out  in  \cite{Noah_Yasha}.
\end{remark}
%\end{comment}

Thanks to Theorem \ref{thm:independence}, we may now  specialise $w_1,\ldots,w_r$ to arbitrary values and take arbitrary limits. We set  $w_i=L^{i}$ for $i=1,\ldots,r$ and compute the limit for $L\to \infty$.
\begin{lemma}\label{lemma: lim in K theory Vij}
Let $i< j$. Then we have
\[
\lim_{L\to \infty}[-\mathsf V_{ij}][-\mathsf V_{ji}]\big|_{w_i=L^{i}}= (-\mathfrak{t}^{\frac{1}{2}})^{|\pi_j|-|\pi_i|}.
\]
\end{lemma}
\begin{proof}

Notice that all monomials in $\mathsf V_{ij}$ are of the form $w_i^{-1}w_j \lambda$ for $\lambda$ a monomial in $t_1,t_2,t_3$. Then
\[
[w_i^{-1}w_j\lambda]\big|_{w_i=L^{i}}=(L^{j-i}\lambda)^{\frac{1}{2}}(1- L^{i-j}\lambda^{-1}).
\]
Write $\mathsf Q_i=\sum_{\mu}t^{\mu}$ and $\mathsf Q_j=\sum_{\nu}t^{\nu}$. Taking limits, we obtain
\begin{align*}
 \lim_{L\to \infty}[-\mathsf V_{ij}]\big|_{w_i=L^{i}}&=\lim_{L\to \infty}[-w_i^{-1}w_j(\mathsf Q_j- \overline{\mathsf Q}_i\mathfrak{t}^{-1}+\overline{\mathsf Q}_i \mathsf Q_j\mathfrak{t}^{-1}(1-t_1)(1-t_2)(1-t_3)  )  ]\big|_{w_i=L^{i}} \\
 &=\lim_{L\to \infty} L^{\frac{j-i}{2}(|\pi_i|-|\pi_j|)}\frac{\prod_\mu (t^{-\frac{\mu}{2}}\mathfrak{t}^{-\frac{1}{2}})}{\prod_{\nu} t^{\frac{\nu}{2}}}.
\end{align*}
Similarly, we obtain
\begin{align*}
 \lim_{L\to \infty}[-\mathsf V_{ji}]\big|_{w_i=L^{i}}&= (-1)^{\rk (-\mathsf V_{ji})}\lim_{L\to \infty}[-\overline{\mathsf V}_{ji}]\big|_{w_i=L^{i}} \\      
 &=(-1)^{|\pi_i|-|\pi_j|}\lim_{L\to \infty}[-w_i^{-1}w_j(\overline{\mathsf Q}_i- {\mathsf Q}_j\mathfrak{t}-\overline{\mathsf Q}_i \mathsf Q_j(1-t_1)(1-t_2)(1-t_3)  )  ]\big|_{w_i=L^{i}}\\
 &=(-1)^{|\pi_i|-|\pi_j|}\lim_{L \to \infty} L^{\frac{j-i}{2}(|\pi_j|-|\pi_i|)}\frac{\prod_\nu (t^{\frac{\nu}{2}}\mathfrak{t}^{\frac{1}{2}})}{\prod_{\mu} t^{-\frac{\mu}{2}}}.
\end{align*}
We conclude, as required, that
\[
\lim_{L\to \infty}[-\mathsf V_{ij}][-\mathsf V_{ji}]\big|_{w_i=L^{i}}= (-\mathfrak{t}^{\frac{1}{2}})^{|\pi_j|-|\pi_i|}.\qedhere
\]
\end{proof}

\begin{lemma}\label{lemma: combinatorical trick}
Let $x$ be a 
variable and $c_i\in \mathbb{Z}$, for $i=1,\ldots,r$. Then we have
\[
\prod_{1\leq i<j\leq r}x^{c_j-c_i}=\prod_{i=1}^r x^{(-r-1+2i)c_i}.
\]
\end{lemma}

\begin{proof}
The assertion holds for $r=1$ as the productory on the left hand side is empty. Assume it holds for $r-1$. Then we have:
\begin{align*}
\prod_{1\leq i<j\leq r}x^{c_j-c_i}&= \prod_{1\leq i<j\leq r-1}x^{c_j-c_i}\prod_{i=1}^{r-1}x^{c_r-c_i}\\
&= x^{(r-1)c_r}\prod_{i=1}^{r-1} x^{(-r-1+2i)c_i}\\
&= \prod_{i=1}^r x^{(-r-1+2i)c_i}. \qedhere
\end{align*}
\end{proof}

Combining Lemma \ref{lemma: lim in K theory Vij} with Lemma \ref{lemma: combinatorical trick} we can express the rank $r$ K-theoretic DT theory of $\BA^3$ as a product of $r$ copies of the rank $1$ K-theoretic DT theory. This product formula already appeared  as a limit of the (conjectural) $4$-fold theory developed by Nekrasov and Piazzalunga \cite[Formula (3.14)]{Magnificent_colors}.\footnote{Typo warning: N. Piazzalunga kindly pointed out to us that in \cite[Formula (3.14)]{Magnificent_colors} one should read `$\frac{N+1}{2}-l$' instead of  `$N+1-2l$'.}
\begin{theorem}\label{thm: r copies of rank 1 K theory}
There is an identity
\[
\DT_r^{\KK}(\BA^3,(-1)^rq,t,w)=\prod_{i=1}^r \DT_1^{\KK}(\BA^3,-q\mathfrak{t}^{\frac{-r-1}{2} +i},t).
\]
\end{theorem}
\begin{proof}
Set $w_i=L^i$. The generating series $\DT_r^{\KK}(\BA^3,q,t,w) $ can be computed in the limit $L\to \infty$:
\begin{align*}
   \lim_{L\to \infty} \DT_r^{\KK}(\BA^3,q,t,w) &= \lim_{L\to \infty}\sum_{\overline{\pi}}q^{|\overline{\pi}|}\prod_{i,j=1}^r [-\mathsf V_{ij}]\\
    &= \lim_{L\to \infty}\sum_{\overline{\pi}}\prod_{i=1}^r q^{|\pi_i|}[-\mathsf V_{ii}]\prod_{1\leq i<j\leq r}[-\mathsf V_{ij}][-\mathsf V_{ji}]\\
    &= \sum_{\overline{\pi}}\prod_{i=1}^r q^{|\pi_i|}[-\mathsf V_{ii}]\prod_{1\leq i<j\leq r}(-\mathfrak{t}^{\frac{1}{2}})^{|\pi_j|-|\pi_i|}\\
    &= \sum_{\overline{\pi}}\prod_{i=1}^r q^{|\pi_i|}[-\mathsf V_{ii}]\prod_{i=1}^r (-\mathfrak{t}^{\frac{1}{2}})^{(-r-1+2i)|\pi_i|}\\
     &= \sum_{\overline{\pi}}\prod_{i=1}^r[-\mathsf V_{ii}] q^{|\pi_i|}(-1)^{(r+1)|\pi_i|}\mathfrak{t}^{(\frac{-r-1}{2} +i)|\pi_i|}\\
    &=\sum_{\overline{\pi}}\prod_{i=1}^r[-\mathsf V_{ii}] \left( (-1)^{(r+1)}q\mathfrak{t}^{\frac{-r-1}{2} +i} \right)^{|\pi_i|}\\
     &=\prod_{i=1}^r \DT_1^{\KK}(\BA^3,(-1)^{(r+1)}q\mathfrak{t}^{\frac{-r-1}{2} +i}, t). \qedhere
\end{align*}
\end{proof}
We can now prove Theorem \ref{thm: K theoretic of points higher rank} (i.e.~Theorem \ref{mainthm:K-theoretic} from the Introduction).
\begin{proofof}{Theorem \ref{thm: K theoretic of points higher rank}}

Define 
\[      G_{r,i}(q,t_1,t_2,t_3)=\mathsf F_1(q\mathfrak{t}^{\frac{-r-1}{2} +i}, t_1,t_2,t_3).\]
 We have
\begin{align*}
 \DT_1^{\KK}(\BA^3,-q\mathfrak{t}^{\frac{-r-1}{2} +i},t)
 &=\exp{\left(\sum_{n\geq 1}\frac{1}{n}\frac{1}{[\mathfrak{t}^{\frac{n}{2}}q^n\mathfrak{t}^{n(\frac{-r-1}{2} +i)}][\mathfrak{t}^{\frac{n}{2}}q^{-n}\mathfrak{t}^{n(\frac{r+1}{2} -i)}]}\frac{[t^n_1t^n_2][t^n_1t^n_3][t^n_2t^n_3]}{[t^n_1][t^n_2][t^n_3]}  \right)} \\
 &=\Exp\left(G_{r,i}(q,t_1,t_2,t_3)\right).
\end{align*}
By Theorem \ref{thm: r copies of rank 1 K theory} and Theorem \ref{thm: Okounkov rank 1} it is enough to show that $\mathsf F_r=\sum_{i=1}^r G_{r,i} $, or equivalently
\[
\sum_{i=1}^r \frac{1}{[\mathfrak{t}^{\frac{1}{2}}q\mathfrak{t}^{\frac{-r-1}{2} +i}][\mathfrak{t}^{\frac{1}{2}}q^{-1}\mathfrak{t}^{\frac{r+1}{2} -i}]} = \frac{[\mathfrak{t}^r]}{[\mathfrak{t}][\mathfrak{t}^{\frac{r}{2} }q][\mathfrak{t}^{\frac{r}{2}}q^{-1}]}.
\]
It is easy to check this is true for $r=1,2$. Let now $r\geq 3 $: we perform induction separately on even and odd cases. Assume the claimed identity holds for $r-2$. In both cases we have
\begin{align*}
\sum_{i=1}^r&\frac{1}{[\mathfrak{t}^{\frac{1}{2}}q\mathfrak{t}^{\frac{-r-1}{2} +i}][\mathfrak{t}^{\frac{1}{2}}q^{-1}\mathfrak{t}^{\frac{r+1}{2}-i}]}\\
&=\sum_{i=1}^{r-2} \frac{1}{[\mathfrak{t}^{\frac{1}{2}}q\mathfrak{t}^{\frac{-(r-2)-1}{2} +i}][\mathfrak{t}^{\frac{1}{2}}q^{-1}\mathfrak{t}^{\frac{(r-2)+1}{2} -i}]}+ \frac{1}{[q\mathfrak{t}^{-\frac{r}{2} +1}][q^{-1}\mathfrak{t}^{\frac{r}{2}}]}+\frac{1}{[q\mathfrak{t}^{\frac{r}{2} }][q^{-1}\mathfrak{t}^{-\frac{r}{2}+1}]}\\
&= \frac{[\mathfrak{t}^{r-2}]}{[\mathfrak{t}][\mathfrak{t}^{\frac{r-2}{2}}q][\mathfrak{t}^{\frac{r-2}{2}}q^{-1}]}- \frac{1}{[\mathfrak{t}^{\frac{r-2}{2}}q^{-1}][\mathfrak{t}^{\frac{r}{2}}q^{-1}]}-\frac{1}{[\mathfrak{t}^{\frac{r}{2} }q][\mathfrak{t}^{\frac{r-2}{2}}q]}\\
&= \frac{1}{[\mathfrak{t}][\mathfrak{t}^{\frac{r}{2} }q][\mathfrak{t}^{\frac{r}{2}}q^{-1}]}\cdot \frac{[\mathfrak{t}^{r-2}][\mathfrak{t}^{\frac{r}{2} }q][\mathfrak{t}^{\frac{r}{2}}q^{-1}]-[\mathfrak{t}][\mathfrak{t}^{\frac{r}{2} }q][\mathfrak{t}^{\frac{r-2}{2}}q]-[\mathfrak{t}][\mathfrak{t}^{\frac{r}{2} }q^{-1}][\mathfrak{t}^{\frac{r-2}{2}}q^{-1}] }{[\mathfrak{t}^{\frac{r-2}{2}}q][\mathfrak{t}^{\frac{r-2}{2}}q^{-1}]}\\
&= \frac{[\mathfrak{t}^r]}{[\mathfrak{t}][\mathfrak{t}^{\frac{r}{2} }q][\mathfrak{t}^{\frac{r}{2}}q^{-1}]}
\end{align*}
by which we conclude the proof.
\end{proofof}

%%%%%%%%%%%%%%%%%%%%%%%%%%%%%%%%%%%%%%%%%%%%%%%%%
\subsection{Comparison with motivic DT invariants}\label{motivic_comparison}
Let $f\colon U\to \BA^1$ be a regular function on a smooth scheme $U$, and let $\hat\mu$ be the group of all roots of unity. The critical locus $Z = \crit (f) \subset U$ inherits a canonical \emph{virtual motive} \cite{BBS}, i.e.~a $\hat\mu$-equivariant motivic class
\[
[Z]_{\vir} = -\BL^{-\frac{\dim U}{2}}\left[\phi_f\right]\,\in\,\CM^{\hat\mu}_{\BC} = K_0^{\hat\mu}(\Var_{\BC})\bigl[\BL^{-\frac{1}{2}}\bigr]
\]
such that $e [Z]_{\vir} = e_{\vir}(Z)$, where $e_{\vir}(-)$ is Behrend weighted Euler characteristic and the Euler number specialisation prescribes $e(\BL^{-1/2}) = -1$. The motivic class $[\phi_f]$ is the (absolute) motivic vanishing cycle class introduced by Denef and Loeser \cite{DenefLoeser1}.

The virtual motive of $\Quot_{\BA^3}(\OO^{\oplus r},n) = \crit(f_n)$, with respect to the critical structure of Proposition \ref{prop:SPOT_A^3}, was computed in \cite[Prop.~2.3.6]{ThesisR}. The result is as follows. Let $\DT_r^{\mot}(\BA^3,q)\in \CM_{\BC} \llbracket q \rrbracket$ be the generating function of the virtual motives $[\Quot_{\BA^3}(\OO^{\oplus r},n)]_{\vir}$. Then one has
\begin{equation}\label{sbidigudi}
\DT_r^{\mot}(\BA^3,q) 
= \prod_{m\geq 1}\prod_{k=0}^{rm-1}\left(1-\BL^{2+k-\frac{rm}{2}}q^m\right)^{-1}.
\end{equation}
The case $r=1$ was computed in \cite{BBS}. The general proof of Formula \eqref{sbidigudi} is obtained in a similar fashion in \cite{ThesisR,Cazzaniga_Thesis}, and via a wall-crossing technique in \cite{cazzaniga2020higher}.
Moreover, it is immediate to verify that $\DT_r^{\mot}$ satisfies a product formula analogous to the one proved in Theorem \ref{thm: r copies of rank 1 K theory} for the K-theoretic invariants: we have
\begin{equation}\label{eqn:motivic_factorisation}
    \DT_r^{\mot}(\BA^3,q) 
    =\prod_{i=1}^r \DT_1^{\mot}\left(\BA^3,q \BL^{\frac{-r-1}{2}+i}\right).
\end{equation}
In particular, up to the substitution $\mathfrak t^{\frac{1}{2}} \to -\BL^{\frac{1}{2}}$, the factorisation \eqref{eqn:motivic_factorisation} is equivalent to the K-theoretic one (Theorem \ref{thm: r copies of rank 1 K theory}). 
As observed in \cite[\S~4]{Virtual_Quot}, the (signed) motivic partition function admits an expression in terms of the motivic exponential, namely
\begin{equation}\label{eqn:motivic_exp}
\DT_r^{\mot}(\BA^3,(-1)^rq) 
= \Exp\left(\frac{(-1)^rq \BL^{\frac{3}{2}}}{\bigl(1-(-1)^rq\BL^{\frac{r}{2}}\bigr)\bigl(1-(-1)^rq\BL^{-\frac{r}{2}}\bigr)}\frac{\BL^{\frac{r}{2}}-\BL^{-\frac{r}{2}}}{\BL^{\frac{1}{2}}-\BL^{-\frac{1}{2}}} \right).
\end{equation}

Given their structural similarities, we believe it is an interesting problem to compare the K-theoretic partition function with the motivic one.

It is worth noticing that Formula \eqref{eqn:motivic_exp} can be recovered from the factorisation \eqref{eqn:motivic_factorisation}, just as we discovered in the K-theoretic case during the proof of Theorem \ref{thm: K theoretic of points higher rank}. This fact follows immediately from the properties of the plethystic exponential.

\begin{remark}
A virtual motive for $\Quot_Y(F,n)$ was defined in \cite[\S\,4]{Virtual_Quot} for every locally free sheaf $F$ on a $3$-fold $Y$. Just as in the case of the naive motives of the Quot scheme \cite{ricolfi2019motive}, the resulting partition function $\DT_r^{\mot}(Y,q)$ only depends on the motivic class $[Y] \in K_0(\Var_{\BC})$ and on $r = \rk F$. See also \cite{Cazzaniga:2020aa} for calculations of motivic higher rank DT and PT invariants in the presence of nonzero curve classes: the generating function $\DT_r^{\mot}(Y,q)$, computed easily starting with Formula \eqref{sbidigudi}, is precisely the DT/PT wall-crossing factor.
\end{remark}

%%%%%%%%%%%%%%%%%%%%%%%%%%%%%%%%%%%%%%%%%%%%%%%%%%%%%%%%%%%%%%%%%
%%%%%%%%%%%%%%%%%%%%%%%%%%%%%%%%%%%%%%%%%%%%%%%%%%%%%%%%%%%%%%%%%
\section{The higher rank cohomological DT partition function}\label{sec:cohomological_invariants}

%%%%%%%%%%%%%%%%%%%%%%%%%%%%%%%%%%%%%%%%%%%%%%%%%%%%%%%%%%%%%%%%%
\subsection{Cohomological reduction}\label{sec: cohom reduction}
One should think of K-theoretic invariants as refinements of the cohomological ones, as by taking suitable limits one  fully recovers $\DT_r^{\coh}(\mathbb{A}^3,q,s)$ from $\DT^{\KK}_r(\mathbb{A}^3,q,t)$. We make this precise in the remainder of this section.

Let $\TT\cong (\BC^\ast)^g$  be an algebraic torus and let $t_1,\dots, t_g$ be its coordinates.
Recall that the Chern character gives a natural transformation from (equivariant) K-theory to the  (equivariant) Chow group with rational coefficients by sending $t_i\mapsto e^{s_i}$, where  $s_i=c_1^\TT(t_i)$. We can formally extend it to
\[
\begin{tikzcd}
\mathbb{Z}[t_1^{\pm 1},\dots,t_g^{\pm 1}]\arrow{r}{\ch}\arrow{d}& \BQ \llbracket s_1,\dots,s_g\rrbracket\arrow{d}\\
 \mathbb{Z}[t_1^{\pm b},\dots,t_g^{\pm b} | b\in \BC]\arrow{r}{ \ch} & \BC\llbracket s_1,\dots,s_g\rrbracket
\end{tikzcd}
\]
by sending $t_i^b\mapsto e^{b s_i}$, where $b\in \BC$.

In \S\,\ref{sec: preliminaries K theory} we defined the symmetrised transformation $[t^\mu]=t^{\frac{\mu}{2}}-t^{-\frac{\mu}{2}}$. We set $[\ch(t^{b \mu})]= e^{\frac{b  \mu\cdot s}{2}}- e^{-\frac{b  \mu\cdot s}{2}}$ as an expression in rational cohomology, which enjoys the following \emph{linearisation} property:
\begin{align*}
    [\ch(t^{b \mu})]&= e^{\frac{b  \mu\cdot s}{2}}(1-e^{-b  \mu \cdot s})= b e^{\TT}(t^{\mu} )+ o(b^2).
\end{align*}
In other words, $ e^{\TT}(\,\cdot\,)$ is the first-order approximation of $[\,\cdot\,]$ in $\TT$-equivariant Chow groups. For a virtual representation $V=\sum_\mu t^{\mu}- \sum_\nu t^\nu\in K^\TT_0(\pt)$, denote by $V^b=\sum_\mu t^{b\mu}- \sum_\nu t^{b\nu} $ the virtual representation where we formally substitute each weight $t^\mu$ with $t^{b \mu}$. We have the identity 
\begin{align*}
    [\ch(V^b)]=\frac{\prod_\mu[\ch(t^{b\mu})]}{\prod_\nu[\ch(t^{b\nu})]}  =b^{\rk V} \frac{\prod_{\mu} ( e^{\TT}(t^{\mu} ) + o(b))}{\prod_{\nu} ( e^{\TT}(t^{\nu} ) + o(b))}.
\end{align*}
If $\rk V=0$, by taking the limit for $b\to 0$ we conclude
\begin{align}
    \lim_{b\to 0}\, [\ch(V^b)]= e^{\TT}(V).
\end{align}

It is clear from the definition of $\ch(\,\cdot\,)$ and $[\,\cdot\,]$ that these two transformations commute with each other. This proves the following relation between K-theoretic invariants and cohomological invariants of the local model.

\begin{corollary}\label{limit K theory cor}
There is an identity
\[
\DT_r^{\coh}(\mathbb{A}^3,q,s,v)=\lim_{b\to 0}\DT_r^{\KK}(\mathbb{A}^3,q,e^{b s},e^{bv}).
\]
\end{corollary}

\begin{proof}
Follows from the description of the generating series of K-theoretic invariants as 
\[
\DT_r^{\KK}(\mathbb{A}^3,q,t,w) =\sum_{n\geq 0}q^n\sum_{[S]\in \Quot_{\BA^3}(\mathscr O^{\oplus r},n)^{\TT}} [-T^{\vir}_S]
\]
and by noticing that $\rk T_S^{\vir}=0$.
\end{proof}

Thanks to the $v$-independence, we can now rename
\[
\DT_r^{\coh}(\BA^3,q,s) = \DT_r^{\coh}(\BA^3,q,s,v).
\]
We are ready to prove Theorem \ref{mainthm:cohomological} from the Introduction.
\begin{theorem}\label{thm:cohomological}
The rank $r$ cohomological DT partition function of $\BA^3$ is given by
 \[
 \DT_r^{\coh}(\mathbb{A}^3,q,s)=\mathsf M((-1)^rq)^{-r\frac{(s_1+s_2)(s_1+s_3)(s_2+s_3)}{s_1s_2s_3}}.
 \]
\end{theorem}

\begin{proof}
By Corollary    \ref{limit K theory cor} and Theorem \ref{thm: K theoretic of points higher rank}, we just need to compute the limit
\[
    \lim_{b\to 0}\DT_r^{\KK}\left(\mathbb{A}^3,(-1)^rq,e^{b s}\right)= \lim_{b\to 0}\Exp\left(\mathsf F_r(  q,t_1^b,t_2^b,t_3^b)\right).
\]
Denote for ease of notation  $\mathfrak{s}=c_1^{\TT}(\mathfrak{t})=s_1+s_2+s_3$. By the definition of plethystic exponential, recalled in \eqref{def_Exp}, we have
\begin{multline*}
    \lim_{b\to 0}\Exp\left(\mathsf F_r(q,t_1^b,t_2^b,t_3^b)\right) \\
    =\exp \sum_{k\geq 1}\frac{1}{k}\left( \lim_{b\to 0} \frac{[e^{bkr\mathfrak{s}  }]}{[ e^{bk\mathfrak{s}  }][e^{\frac{bkr}{2}\mathfrak{s}  } q^k    ][ e^{\frac{bkr}{2}\mathfrak{s}  } q^{-k}    ]}\frac{[e^{bk (s_1+s_2)}][e^{bk (s_1+s_3)}][e^{bk (s_2+s_3)}]}{[e^{bk s_1}][e^{bk s_2}][e^{bk s_3}]} \right).
\end{multline*}
We have
\begin{align*}
    \lim_{b\to 0}\frac{[e^{bk (s_1+s_2)}][e^{bk (s_1+s_3)}][e^{bk (s_2+s_3)}]}{[e^{bk s_1}][e^{bk s_2}][e^{bk s_3}]}=
\frac{(s_1+s_2)(s_1+s_3)(s_2+s_3)}{s_1s_2s_3},
\end{align*}
and
\begin{align*}
\lim_{b\to 0} \frac{[e^{bkr\mathfrak{s}  }]}{[ e^{bk\mathfrak{s}  }][e^{\frac{bkr}{2}\mathfrak{s}  }  q^k    ][ e^{\frac{bkr}{2}\mathfrak{s}  }  q^{-k}    ]}&= \frac{r}{[q^{k} ][ q^{-k} ]}
    = -r\cdot \frac{q^{k}}{(1-q^{k})^2}.
\end{align*}
Recall the plethystic exponential form of the MacMahon function
\begin{align*}
 \mathsf M(q)= \prod_{n\geq 1} (1-q^n)^{-n}=\mathrm{Exp}\bigg(\frac{q}{(1-q)^2} \bigg).
\end{align*}
We conclude
\begin{align*}
      \lim_{b\to 0} \DT_r^{\KK}(\mathbb{A}^3,(-1)^r q,  e^{b s})&= \exp\left(-r\cdot \frac{(s_1+s_2)(s_1+s_3)(s_2+s_3)}{s_1s_2s_3} \sum_{k\geq 1} \frac{1}{k} \frac{q^{k}}{(1-q^{k})^2}\right)\\
      &= \mathsf M(q)^{-r\frac{(s_1+s_2)(s_1+s_3)(s_2+s_3)}{s_1s_2s_3}}.\qedhere
\end{align*}
%The proof is complete.
\end{proof}

Thus we proved Szabo's conjecture \cite[Conj.~4.10]{Szabo}.

\begin{remark}
The specialisation
\[
\DT_r^{\coh}(\A^3,q,s)\big|_{s_1+s_2+s_3 = 0} = \mathsf M((-1)^rq)^r,
\]
recovering Formula \eqref{eqn:chi_vir_quot_affinespace},
was already known in physics, see e.g.~\cite{Cir-Sink-Szabo}.
\end{remark}

We end this section with a small variation of Theorem \ref{mainthm:cohomological}.

\begin{corollary}\label{cor:independence_on_lambda_w}
Fix an $r$-tuple $\lambda = (\lambda_1,\ldots,\lambda_r)$ of $\BT_1$-equivariant line bundles on $\BA^3$. Then there is an identity
\[
\DT_r^{\coh}(\BA^3,q,s) = \DT_r^{\coh}(\BA^3,q,s,v)_\lambda,
\]
where the right hand side was defined in \eqref{eqn:cohomological_DT_series_lambda}.
\end{corollary}

\begin{proof}
We have 
\[
T_{S,\lambda}^{\vir}=\sum_{i,j}\lambda_i^{-1}\lambda_j \mathsf V_{ij}.
\]
Let $\mathsf V_{ij}=\sum_\mu w_i^{-1}w_j t^\mu $ be the decomposition into weight spaces. A monomial in $T_{S,\lambda}^{\vir}$ is of the form $ \lambda_i^{-1}\lambda_jw_i^{-1}w_jt^{\mu}$ and its Euler class is
\begin{align*}
e^{\TT}( \lambda_i^{-1}\lambda_jw_i^{-1}w_jt^{\mu})&=\mu\cdot s + v_j+c^{\TT}_1(\lambda_j)-v_i-c^{\TT}_1(\lambda_i)\\
&= \mu\cdot s + \overline{v}_j-\overline{v}_i
\end{align*}  
where we define $  \overline{v}_i= v_i+c^{\TT}_1(\lambda_i)$. We conclude that
\[
\DT_r^{\coh}(\BA^3,q,s,v)_\lambda=\DT_r^{\coh}(\BA^3,q,s,\overline{v}),
\]
which does not depend on $\overline{v}$ by Theorem \ref{thm:independence}.
\end{proof}

\begin{example}\label{example r=2, n=1}
Set $r=2$, $n=1$, so that the only $\TT$-fixed points in $\Quot_{\BA^3}(\OO^{\oplus 2},1)$ are the direct sums of ideal sheaves
\[
S_1 = \mathscr I_{\pt}\oplus \OO \subset \OO^{\oplus 2},\quad S_2 = \OO \oplus \mathscr I_{\pt} \subset \OO^{\oplus 2},
\]
where $\pt = (0,0,0) \in \BA^3$ is the origin.
One computes
\begin{align*}
    T^{\vir}_{S_1} &= 1-\frac{1}{t_1t_2t_3}+\frac{(1-t_1)(1-t_2)(1-t_3)}{t_1t_2t_3} - w_1^{-1}w_2\frac{1}{t_1t_2t_3} + w_2^{-1}w_1 \\
    T^{\vir}_{S_2} &= 1-\frac{1}{t_1t_2t_3}+\frac{(1-t_1)(1-t_2)(1-t_3)}{t_1t_2t_3}- w_2^{-1}w_1\frac{1}{t_1t_2t_3} + w_1^{-1}w_2.
\end{align*}
Therefore, the cohomological DT invariant is
\[
\int_{[\Quot_{\BA^3}(\OO^{\oplus 2},1)]^{\vir}} 1 = e^{\TT}(-T^{\vir}_{S_1}) + e^{\TT}(-T^{\vir}_{S_2}).
\]
The part that could possibly depend on the framing parameters $v_1$ and $v_2$ is, in fact, constant:
\[
e^{\TT}\left(w_1^{-1}w_2\frac{1}{t_1t_2t_3} - w_2^{-1}w_1\right) + e^{\TT}\left(w_2^{-1}w_1\frac{1}{t_1t_2t_3} - w_1^{-1}w_2\right)
= \frac{-v_1+v_2-\mathfrak s}{v_1-v_2} + \frac{v_1-v_2-\mathfrak s}{v_2-v_1} = -2.
\]
Let $\lambda_1$ and $\lambda_2$ be two $\BT_1$-equivariant line bundles. After the substitutions $w_i \to w_i\lambda_i$, and setting $\overline v_i = v_i + c_1(\lambda_i)$, the final sum of Euler classes depending on $\overline v$ becomes
\[
2\frac{\overline v_2-\overline v_1}{\overline v_1-\overline v_2} = -2.
\]
\end{example}

\section{Elliptic Donaldson--Thomas invariants}\label{sec:elliptic_invariants}
\subsection{Chiral elliptic genus}
In \cite{BBPT} an elliptic (string-theoretic) generalisation of the $K$-theoretic DT invariants is given. In  physics the invariants computed in loc.~cit.~are obtained as the superconformal index of a D1-D7 brane system on a type IIB  $\mathcal N=1$ supersymmetric background, where $r$ D7-branes wrap the product of a $3$-fold by a torus, i.e.~$X_3\times T^2$, while $n$ D1-branes wrap $T^2$. The connection with enumerative geometry is then given via BPS-bound states counting, as brane systems often provide interesting constructions of relevant moduli spaces. In the case at hand, when $X_3=\BA^3$, the D1-D7 brane system considered has a BPS moduli space which can be naturally identified with the moduli space parametrising quotients of length $n$ of $\OO_{\BA^3}^{\oplus r}$. The superconformal index is usually identified in the physics literature with the elliptic genus of such a moduli space. This however does not coincide with the usual notion of (virtual) elliptic genus, as the coupling to the D7-branes breaks half of the chiral supersymmetry, thus leading to an effective $2d$ $\mathcal N=(0,2)$ GLSM on $T^2$ for the D1-brane dynamics, whose Witten index generalises the K-theoretic DT invariants of $\BA^3$ and provide a sort of chiral (or $1/2$ BPS) version of the elliptic genus of the Quot scheme. In this section we give a mathematical definition of the elliptic invariants computed in \cite{BBPT}, and show that our definition leads precisely to the computation of the same quantities found in \cite[\S\,3]{BBPT}.

\smallbreak
Let $X$ be a scheme carrying a perfect obstruction theory $\BE\to \BL_X$ of virtual dimension $\vd=\rk \BE$.

\begin{definition}
If $F$ is a rank $r$ vector bundle on $X$ we define
\begin{equation}\label{E_1/2}
\mathcal E_{1/2}(F)=\bigotimes_{n\ge 1}\Sym_{p^n}^\bullet\left(F\oplus F^\vee\right)\in 1+p\cdot K^0(X)\llbracket p\rrbracket
\end{equation}
where the total symmetric algebra $\Sym^\bullet_p(F)=\sum_{i\ge 0}p^i[S^iF]\in K^0(X)\llbracket p \rrbracket$ satisfies $\Sym^\bullet_p(F)=1/\Lambda^\bullet_{-p}(F)$. Note that $\mathcal E_{1/2}$ defines a homomorphism from the additive group $K^0(X)$ to the multiplicative group $1+p\cdot K^0(X)\llbracket p\rrbracket$. Set
\begin{equation}
\ELL_{1/2}(F;p)=(-p^{-\frac{1}{12}})^{\rk F}\ch\left(\mathcal E_{1/2}(F)\right)\cdot\td(F)\in A^*(X)\llbracket p\rrbracket[p^{\pm \frac{ 1}{12}}],
\end{equation}
where $\td(-)$ is the Todd class, so that $\ELL_{1/2}(-;p)$ extends to a group homomorphism from $K^0(X)$ to the multiplicative group of units in $A^*(X)\llbracket p\rrbracket[p^{\pm\frac{1}{12}}]$.
\end{definition}
We can then define the virtual chiral elliptic genus as follows. 

\begin{definition}\label{half bps elliptic genus}
Let $X$ be a proper scheme with a perfect obstruction theory and $V\in K^0(X)$. The \emph{virtual chiral elliptic genus} is defined as 
\[
    \Ell_{1/2}^{\vir}(X,V;p)=(-p^{-\frac{1}{12}})^{\vd}\chi^{\vir}\left(X,\mathcal E_{1/2}(T_X^{\vir})\otimes V\right)\in\BZ \llbracket p\rrbracket[p^{\pm\frac{1}{12}}].
\]
By the virtual Riemann--Roch theorem of \cite{Fantechi_Gottsche} we can also compute the virtual chiral elliptic genus as
\[
\Ell_{1/2}^{\vir}(X,V;p)=\int_{[X]^{\vir}}\ELL_{1/2}(T_X^{\vir};p)\cdot\ch(V).
\]
\end{definition}

\begin{remark}
One may give a  more general definition by adding a ``mass deformation'' and defining $\mathcal E_{1/2}^{(y)}(F)$ for $F\in K^0(X)$ as
\[
\mathcal E_{1/2}^{(y)}(F;p)=\bigotimes_{n\ge 1}\Sym_{y^{-1}p^n}^\bullet\left(F\right)\otimes \Sym_{yp^n}^\bullet\left(F^\vee\right)\in 1+p\cdot K^0(X)[y,y^{-1}]\llbracket p \rrbracket,
\]
so we recover the standard definition of virtual elliptic genus  by taking $\mathcal E(F)=\mathcal E_{1/2}^{(1)}(F;p)\otimes\mathcal E_{1/2}^{(y)}(-F;p)$, cf.~\cite[\S~6]{Fantechi_Gottsche}.
\end{remark}
\iffalse\begin{SM}
We have:
$$\bigwedge_{y}(F)=S_{-y}(-F) $$
and so
$$E_{1/2}^{(1)}(F;p)E_{1/2}^{(y)}(-F;p)= \bigotimes_{n\ge 1}\bigwedge_{-y^{-1}p^n}^\bullet\left(F\right)\otimes \bigwedge_{-yp^n}^\bullet\left(F^\vee\right) \otimes \Sym_{p^n}^\bullet\left(F\oplus F^\vee\right)$$
\end{SM}\fi

\begin{prop}
Let $X$ be a proper scheme with a perfect obstruction theory and let $V\in K^0(X)$. Then the virtual chiral elliptic genus $\Ell_{1/2}^{\vir}(X,V;p)$ is deformation invariant.
\end{prop}
\begin{proof}
The statement follows directly from Definition \ref{half bps elliptic genus} and \cite[Theorem 3.15]{Fantechi_Gottsche}.
\end{proof}

\iffalse
\begin{prop}
Let $X$ be equipped with a symmetric perfect obstruction theory $[E^{-1}\to E^0]$. Then $\Ell_{1/2}^{\vir}(X,(\mathcal K_X^{\vir})^{1/2};p)$ is $SL(2,\BZ)$-invariant.
\end{prop}
\begin{proof}
Let $p=e^{2\pi i\tau}$, $\Im\tau>0$. The modular group $SL(2,\BZ)$ acts as
\[
\tau\mapsto\frac{a\tau+b}{c\tau+d},\qquad \begin{pmatrix}
a & b\\
c & d
\end{pmatrix}\in SL(2,\BZ).
\]
In order to prove the statement, let $F$ be any vector bundle on $X$, with Chern roots $f_i$, $i=1,\dots,r$. Assume $\det F^\vee$ admits a square root in $K^0(X)$, so we have
\[
\ELL_{1/2}(F;p)\cdot\ch(-p^{2/3}\det F^\vee)^{1/2}=\prod_{\ell=1}^rf_\ell\frac{\eta(p)}{\theta\left(\frac{f_\ell}{2\pi i},p\right)},
\]
where $\theta(y,z)$ denotes the Jacobi theta function
\[
\theta(y,z)=-iz^{1/8}(y^{1/2}-y^{-1/2})\prod_{n=1}^\infty(1-z^n)(1-yz^n)(1-y^{-1}z^n).
\]
It is then known that $\theta(y,z)$ is a Jacobi form of weight $1/2$ and index $1/2$ and that $\eta(z)$ is a modular form of weight $1/2$ with character. Indeed we have
\begin{align*}
&\eta(\tau+1)=e^{i\pi/12}\eta(\tau)\\
&\eta(-1/\tau)=\sqrt{-i\tau}\eta(\tau)
\end{align*}
and
\begin{align*}
&\theta(z|\tau+1)=e^{i\pi/4}\theta(z|\tau)\\
&\theta(z/\tau|-1/\tau)=-i\sqrt{-i\tau}e^{\pi i z^2/\tau}\theta(z|\tau)
\end{align*}
Moreover, as $T_X^{\vir}=[E_0-E_1]$ we have
$\ELL_{1/2}(T_X^{\vir};p)=\ELL_{1/2}(E_0;p)/\ELL_{1/2}(E_1;p)$. The fact that the perfect obstruction theory on $X$ is symmetric is enough to conclude.
\end{proof}\fi

Let now $V=\sum_\mu t^\mu$ be a $\TT$-module as in \S\,\ref{sec: preliminaries K theory}. The trace of its symmetric algebra is given by
\[
\tr\left(\Sym^\bullet_p(V)\right)=\tr\left(\frac{1}{\Lambda^\bullet_{-p}(V)}\right)=\prod_\mu\frac{1}{1-pt^\mu}.
\]
Let us now assume as in \S\,\ref{sec: preliminaries K theory} that $\det V$ is a square in $K_{\TT}^0(\pt)$ and $\mu=0$ is not a weight of $V$. We can then compute the trace of the symmetric product in \eqref{E_1/2} as
\begin{align*}
\tr\left(\bigotimes_{n\ge 1}\Sym_{p^n}^\bullet\left(V\oplus V^\vee\right)\right)&=\prod_{\mu}\prod_{n\ge 1}\frac{1}{(1-p^nt^\mu)(1-p^nt^{-\mu})},
\end{align*}
\begin{align*}
\nonumber\tr\left(\frac{\bigotimes_{n\ge 1}\Sym_{p^n}^\bullet\left(V\oplus V^\vee\right)}{\Lambda^\bullet V^\vee}\right)&=\prod_{\mu}\frac{1}{1-t^{-\mu}}\prod_{n\ge 1}\frac{1}{(1-p^nt^\mu)(1-p^nt^{-\mu})}\\
&=\left(-ip^{1/8}\phi(p)\right)^{\rk V}\prod_\mu\frac{t^{\mu/2}}{\theta(p;t^\mu)},\label{trace_thetas}
\end{align*}
where $\phi(p)$ is the Euler function, i.e.~$\phi(p)=\prod_n(1-p^n)$, and $\theta(p;y)$ denotes the Jacobi theta function
\[
\theta(p;y)=-ip^{1/8}(y^{1/2}-y^{-1/2})\prod_{n=1}^\infty(1-p^n)(1-yp^n)(1-y^{-1}p^n).
\]
Combining everything together we  get the identity
\begin{align*}
(-p^{-\frac{1}{12}})^{\rk V}\tr\left(\frac{\bigotimes_{n\ge 1}\Sym_{p^n}^\bullet\left(V\oplus V^\vee\right)\otimes\det (V^\vee)^{1/2}}{\Lambda^\bullet V^\vee}\right)&=(-p^{-\frac{1}{12}})^{\rk V}\left(-ip^{1/8}\phi(p)\right)^{\rk V}\prod_\mu\frac{1}{\theta(p;t^\mu)}\\
&=\prod_\mu i\frac{\eta(p)}{\theta(p;t^\mu)},
\end{align*}
where $\eta(p)$ is the Dedekind eta function
\[
\eta(p)=p^{\frac{1}{24}}\prod_{n\ge 1}(1-p^n).
\]
For a  virtual $\TT$-representation $ V=\sum_\mu t^\mu-\sum_\nu t^\nu\in K_{\TT}^0(\pt)$ where  $\mu=0$ is not a weight of $V$, we compute  
\begin{align*}
(-p^{-\frac{1}{12}})^{\rk  V}\tr\left(\frac{\bigotimes_{n\ge 1}\Sym_{p^n}^\bullet\left( V\oplus V^\vee\right)\otimes\det ( V^\vee)^{1/2}}{\Lambda^\bullet V^\vee}\right)=     (i\cdot \eta(p))^{\rk V} \frac{\prod_{\nu}\theta(p;t^\nu)}{\prod_{\mu}\theta(p;t^\mu)}.
\end{align*}
For the remainder of the section we set $p=e^{2\pi i \tau}$, with $\tau\in \mathbb{H}=\set{\tau\in \BC | \Im(\tau)>0}$.   Denoting $\theta(\tau|z)=\theta(e^{2\pi i\tau};e^{2\pi iz})$,  $\theta$ enjoys the  modular behaviour %\cite{Mumford_Tata_I}
\[
\theta(\tau|z+a+b\tau)=(-1)^{a+b}e^{-2\pi ibz}e^{-i\pi b^2\tau}\theta(\tau|z),\quad a,b\in\BZ.
\]

Analogously to the measure $[\,\cdot\,]$ for K-theoretic invariants, we  define the \emph{elliptic measure}
\[
\theta[V]=     (i\cdot \eta(p))^{-\rk V} \frac{\prod_{\mu}\theta(p;t^\mu)}{\prod_{\nu}\theta(p;t^\nu)},
\]
which satisfies $\theta[\overline V]=(-1)^{\rk V}\theta[V]$. Notice that, if $\rk V=0$, the elliptic measure refines both $[\,\cdot\,]$ and $e^\TT(\,\cdot\,)$
\[\theta[V]\xrightarrow{p\to 0} [V]\xrightarrow{b\to 0} e^{\TT}(V)\]
where the second limit was dicussed in \S\,\ref{sec: cohom reduction}.
\begin{remark}
The definition we gave for virtual chiral elliptic genus is reminiscent of what is commonly known in physics as the elliptic genus (or superconformal index) of a $2d$ $\mathcal N=(0,2)$ sigma model. Our definition actually matches the one in \cite{KawaiMohri} for $\mathcal N=(0,2)$ Landau--Ginzburg models. Indeed, in this case we are given an $n$-dimensional (compact) K\"ahler manifold $X$ together with a holomorphic vector bundle $E\to X$ such that $c_1(E)-c_1(T_X)=0\mod 2$. If we then consider the K-theory class $[V]=[T_X]-[E]$, the superconformal index of \cite{KawaiMohri} can be written as
\[
\boldsymbol{\mathcal I}(X,E;p)=(-ip^{-\frac{1}{12}})^{\rk V}\chi\left(X,\mathcal E_{1/2}(V)\otimes\det(V^\vee)^{\frac{1}{2}}\right)
\]
and in terms of the Chern roots $v_i$, $w_j$ of $T_X$ and $E$, respectively, we also have (cf.~\cite{FrancoGhimLeeSeong})
\[
\boldsymbol{\mathcal I}(X,E;p)=\int_X\prod_{i=1}^r\frac{\theta(\tau|\frac{w_i}{2\pi i})}{\eta(p)}\prod_{j=1}^n\frac{v_j\eta(p)}{\theta(\tau|\frac{-v_j}{2\pi i})}.
\]
\end{remark}

\subsection{Elliptic DT invariants}\label{sec:elliptic DT}
\begin{definition}\label{def:elliptic DT of A^3}
The generating series of elliptic DT  invariants $\DT_r^{\rm ell}(\BA^3,q,t,w;p)$ is defined as
\[
\DT_r^{\rm ell}(\BA^3,q,t,w;p)=\sum_{n\ge 0}q^n\Ell^{\vir}_{1/2}(\Quot_{\BA^3}(\OO^{\oplus r},n),\mathcal K_{\vir}^{\frac{1}{2}};p)\in \BZ(\!(t,\mathfrak t^{\frac{1}{2}},w)\!)\llbracket p,q \rrbracket.
\]
Being that $\Quot_{\BA^3}(\OO^{\oplus r},n)$ is not projective, but nevertheless carries the action of an algebraic torus $\TT$ with proper $\TT$-fixed locus, we  define the invariants  by means of virtual localisation, as we explained in \S\,\ref{sec:vir_loc}.
\end{definition}
At each $\TT$-fixed point $[S]\in\Quot_{\BA^3}(\OO^{\oplus r},n)^{\TT}$, the localised contribution is
\[
\tr\left(\frac{\bigotimes_{n\ge 1}\Sym_{p^n}^\bullet\left(T_S^{\vir}\oplus T_S^{\vir,\vee}\right)}{\widehat{\Lambda}^\bullet T_S^{\vir,\vee}}\right)
\]
from which we deduce that we can recover the K-theoretic invariants $\chi(\Quot_{\BA^3}(\OO^{\oplus r},n),\widehat{\OO}^{\vir})$  in the limit  $p\to 0$.
As for K-theoretic invariants, we have
\begin{align*}
    \DT_r^{\rm ell}(\BA^3,q,t,w;p)\,&=\,\sum_{n\geq 0}q^n\sum_{[S]\in \Quot_{\BA^3}(\mathscr O^{\oplus r},n)^{\TT}}\theta[-T^{\vir}_S]\\
    \,&=\,\sum_{\overline\pi}q^{\lvert \overline\pi\rvert}\prod_{i,j=1}^r\theta[-\mathsf V_{ij}],
\end{align*}
where $\overline{\pi}$ runs over all $r$-colored plane partitions.

\smallbreak
Contrary to the case of K-theoretic and cohomological invariants, there exists no conjectural closed formula for elliptic DT invariants yet, even for the rank 1 case. Moreover, the generating series  depends on the equivariant parameters of the framing torus, as shown in the following example. 
\begin{example}\label{example dependence elliptic}
Consider $Q_1^3=\Quot_{\BA^3}(\OO^{\oplus 3},1)$, whose only $\TT$-fixed points are
\[
S_1=\mathscr I_\pt\oplus\OO\oplus\OO\subset\OO^{\oplus 3},\quad S_2= \OO\oplus\mathscr I_\pt\oplus\OO\subset\OO^{\oplus 3},\quad S_3=\OO\oplus\OO\oplus\mathscr I_\pt\subset\OO^{\oplus 3},
\]
with $\pt=(0,0,0)\in\BA^3$ as in Example \ref{example r=2, n=1}.  We have
\begin{align*}
    T^{\vir}_{S_1} &= 1-\frac{1}{t_1t_2t_3}+\frac{(1-t_1)(1-t_2)(1-t_3)}{t_1t_2t_3} - w_1^{-1}w_2\frac{1}{t_1t_2t_3} + w_2^{-1}w_1 - w_1^{-1}w_3\frac{1}{t_1t_2t_3} + w_3^{-1}w_1 \\
    T^{\vir}_{S_2} &= 1-\frac{1}{t_1t_2t_3}+\frac{(1-t_1)(1-t_2)(1-t_3)}{t_1t_2t_3}- w_2^{-1}w_1\frac{1}{t_1t_2t_3} + w_1^{-1}w_2 - w_2^{-1}w_3\frac{1}{t_1t_2t_3} + w_3^{-1}w_2 \\
    T^{\vir}_{S_3} &= 1-\frac{1}{t_1t_2t_3}+\frac{(1-t_1)(1-t_2)(1-t_3)}{t_1t_2t_3}- w_3^{-1}w_1\frac{1}{t_1t_2t_3} + w_1^{-1}w_3 - w_3^{-1}w_2\frac{1}{t_1t_2t_3} + w_2^{-1}w_3
\end{align*}
by which we may compute the corresponding elliptic invariant. Set $w_j=e^{2\pi iv_j}$ and $t_\ell=e^{2\pi is_\ell}$, so that
\begin{multline*}
    \Ell^{\vir}_{1/2}(Q_1^3,\mathcal K_{\vir}^{\frac{1}{2}},t,w;p)=\diamondsuit\cdot\left( \frac{\theta(\tau|v_2-v_1-s)\theta(\tau|v_3-v_1-s)}{\theta(\tau|v_1-v_2)\theta(\tau|v_1-v_3)}\right.\\
    \left.+\frac{\theta(\tau|v_1-v_2-s)\theta(\tau|v_3-v_2-s)}{\theta(\tau|v_2-v_1)\theta(\tau|v_2-v_3)}+
    \frac{\theta(\tau|v_1-v_3-s)\theta(\tau|v_2-v_3-s)}{\theta(\tau|v_3-v_1)\theta(\tau|v_3-v_2)}\right),
\end{multline*}
where $s=s_1+s_2+s_3$, with the overall factor
\[
    \diamondsuit=\frac{\theta(\tau|s_1+s_2)\theta(\tau|s_1+s_3)\theta(\tau|s_2+s_3)}{\theta(\tau|s_1)\theta(\tau|s_2)\theta(\tau|s_3)}.
\]
Moreover, by evaluating residues in $v_i-v_j=0$ one can realise that $\Ell^{\vir}_{1/2}(Q_1^3,\mathcal K_{\vir}^{1/2},t,w;p)$ has no poles in $v_i$. Indeed
\[
\Res_{v_1-v_2=0}\Ell^{\vir}_{1/2}(Q_1^3,\mathcal K_{\vir}^{\frac{1}{2}};p)=\diamondsuit\cdot\left(\frac{\theta(\tau|-s)\theta(\tau|v_3-v_2-s)}{\theta(\tau|v_2-v_3)}-\frac{\theta(\tau|-s)\theta(\tau|v_3-v_2-s)}{\theta(\tau|v_2-v_3)}\right)=0,
\]
and the same occurs for any other pole involving the $v_i$'s. However, this does not imply the independence of the elliptic invariants from $v$, as we now suggest.
%\smallbreak

Set $\overline{v}_i=v_i + a_i+b_i\tau$, with $a_i, b_i\in \BZ$, for $i=1,2,3$. Applying the quasi-periodicity of theta functions, we get
\begin{multline*}
     \Ell^{\vir}_{1/2}(Q_1^3,\mathcal K_{\vir}^{\frac{1}{2}},t,\overline{w};p)=\\
     \frac{\diamondsuit}{\theta(\tau|\overline{v}_1-\overline{v}_2)\theta(\tau|\overline{v}_1-\overline{v}_3)\theta(\tau|\overline{v}_2-\overline{v}_3)}\cdot\left(\theta(\tau|\overline{v}_2-\overline{v}_1-s)\theta(\tau|\overline{v}_3-\overline{v}_1-s)\theta(\tau|\overline{v}_2-\overline{v}_3) \right.\\
     \left. -\theta(\tau|\overline{v}_1-\overline{v}_2-s)\theta(\tau|\overline{v}_3-\overline{v}_2-s)\theta(\tau|\overline{v}_1-\overline{v}_3)+\theta(\tau|\overline{v}_1-\overline{v}_3-s)\theta(\tau|\overline{v}_2-\overline{v}_3-s)\theta(\tau|\overline{v}_1-\overline{v}_2) \right)\\
     =\frac{\diamondsuit}{\theta(\tau|v_1-v_2)\theta(\tau|v_1-v_3)\theta(\tau|v_2-v_3)}\cdot\left(e^{2\pi is (b_2+b_3-2b_1)}\theta(\tau|v_2-v_1-s)\theta(\tau|v_3-v_1-s)\theta(\tau|v_2-v_3) \right.\\
     \left. -e^{2\pi is (b_1+b_3-2b_2)}\theta(\tau|v_1-v_2-s)\theta(\tau|v_3-v_2-s)\theta(\tau|v_1-v_3) \right.\\ 
     \left.
     +e^{2\pi is (b_1+b_2-2b_3)}\theta(\tau|v_1-v_3-s)\theta(\tau|v_2-v_3-s)\theta(\tau|v_1-v_2) \right).
\end{multline*}
Notice that for general values of $s$ the above expression is different from $\Ell^{\vir}_{1/2}(Q_1^3,\mathcal K_{\vir}^{1/2},t,w;p)$. However, if we specialise  $s\in \frac{1}{3}\BZ$, we see that in the previous example $ \Ell^{\vir}_{1/2}(Q_1^3,\mathcal K_{\vir}^{1/2},t,w;p)$ becomes constant and periodic with respect to $v$ on the lattice $\BZ+ 3\tau \BZ$ and is holomorphic in $v$, from which we conclude that it is constant on $v$ under this specialisation. Therefore, by choosing $w_j=e^{2\pi i \frac{j}{3}}$ to be third roots of unity, one can show
\[
 \left.\Ell^{\vir}_{1/2}(Q_1^3,\mathcal K_{\vir}^{\frac{1}{2}},t,w;p)\right|_{\mathfrak{t}=e^{2\pi i \frac{k}{3}}}= \begin{cases} (-1)^{m+1}3, & \mbox{if } k=3m, \quad m\in \BZ \\ 0, & \mbox{if } k\notin 3\BZ. \end{cases}
\]
\end{example}

\subsection{Limits of elliptic DT invariants}\label{sec:limits of elliptic DT}
Even if  a closed formula for the higher rank generating series of elliptic DT invariants is not available, we can still study its behaviour by looking at some particular limits of the variables $p,t_i,w_j$. 

\smallbreak
It is easy to see that, under the  Calabi--Yau restriction $\mathfrak{t}=1$, the generating series of elliptic DT invariants does not carry any more refined information than the cohomological one; in particular, we have no more dependence on the framing parameters $w_j$ and the elliptic parameter $p$. We generalise this phenomenon in the following setting. Denote by $\TT_k\subset\BT_1$ the subtorus where  $\mathfrak t^{\frac{1}{2}}=e^{\pi ik/r}$, $k\in\BZ$. Define by
\[\DT_{r,k}^{\rm ell}(\BA^3,q,t,w;p)=\left.\DT_{r}^{\rm ell}(\BA^3,q,t,w;p)\right|_{\TT_k}\]
the restriction of the generating series to the subtorus  $\TT_k\subset\BT_1$, which is well-defined as no powers of the Calabi--Yau weight appear in the vertex terms \eqref{eqn:vertex_terms} by Lemma \ref{lemma:T_fixed_points}. 

\begin{prop}\label{corollary to elliptic macmahon prop}
If $k=rm\in r\BZ$, then
\[
\DT_{r, k}^{\rm ell}(\BA^3,q,t,w;p)=\mathsf M((-1)^{r(m+1)}q)^{r}.
\]
In particular, the dependence on $t_i, w_j$ and $p$ drops.
\end{prop}
\begin{proof}
Let $S\in \Quot_{\BA^3}(\OO^{\oplus r},n)^{\TT}$. Denote $T_S^{\vir}=T^{nc}_{S}-\overline{T^{nc}_{S}}\mathfrak{t}^{-1}$ as in Equation \eqref{eqn:virtual_tg_quot}, where $T^{nc}_{S}$ is the tangent space of $\NCQuot_{r}^n$ at $S$. Denote by $T^{nc}_{S,l}$ the sub-representation of $ T^{nc}_{S} $ corresponding to $\mathfrak{t}^l$, with $l\in \mathbb{Z}$. As there are no powers of the Calabi--Yau weight in $T_S^{\vir}$, we have an identity $T^{nc}_{S,l}=\overline{T^{nc}_{S,-l-1}}\mathfrak{t}^{-1}$. Set 
\[W=T^{nc}_{S}-\sum_{n\in \BZ}T^{nc}_{S,n}. \] 
We have that
\[T^{nc}_{S}-\overline{T^{nc}_{S}}\mathfrak{t}^{-1}=W-\overline{W}\mathfrak{t}^{-1}\]
and, in particular, neither $W$ nor $\overline{W}\mathfrak{t}^{-1}$ contain constant terms and powers of the Calabi-Yau weight.
Using the quasi-periodicity of the theta function $\theta(\tau|z)$, we have
\begin{align*}
    \theta[-T_S^{\vir}]=\frac{\theta[\overline{W}\mathfrak{t}^{-1}]}{\theta[W]}=(-1)^{m\rk W}\frac{\theta[\overline{W}]}{\theta[W]}=(-1)^{\rk W(m+1)}.
\end{align*}
We conclude by noticing that
\[
\rk W=\rk T_S^{nc}=rn \, \mod{2}.\qedhere
\]
\end{proof}

Motivated by Example \ref{example dependence elliptic} and Proposition \ref{corollary to elliptic macmahon prop}, we propose the following conjecture.

\begin{conjecture}\label{conj: non dependence under some limits of elliptic}
The series $\DT_{r,k}^{\rm ell}(\BA^3,q,t,w;p)$ does not depend on the elliptic parameter $p$.
\end{conjecture}
\begin{remark}\label{rem: independence of p then equal to K theoretic}
Notice that the independence from the elliptic parameter $p$ implies that we can reduce our invariants to the K-theoretic ones by setting $p=0$, i.e.
\[ \DT_{r,k}^{\rm ell}(\BA^3,q,t,w;p)=\DT_{r}^{ \KK}\left.(\BA^3,q,t)\right|_{\TT_k},\]
which in particular do not depend on the framing parameters.
\end{remark}

Assuming Conjecture \ref{conj: non dependence under some limits of elliptic}, we derive a closed expression for $\DT_{r,k}^{\rm ell}(\BA^3,q,t,w;p)$, which was conjectured in \cite[Equation (3.20)]{BBPT}, motivated by string-theoretic phenomena.

\begin{theorem}\label{thm:elliptic}
Assume Conjecture \ref{conj: non dependence under some limits of elliptic} holds and let $k\in \BZ$. Then there is an identity
\begin{equation*}\label{eqn:conjecture elliptic}
\DT_{r,k}^{\rm ell}(\BA^3,q,t,w;p)=\mathsf M\left((-1)^{kr}((-1)^rq)^{\frac{r}{\gcd(k,r)}}\right)^{\gcd(k,r)}.
\end{equation*}
\end{theorem}

\begin{proof}
Assuming Conjecture \ref{conj: non dependence under some limits of elliptic}, by Remark \ref{rem: independence of p then equal to K theoretic} we just have to prove the result for K-theoretic invariants. By Theorem \ref{thm: K theoretic of points higher rank},
\begin{multline*}
    \DT_r^{\KK}(\BA^3,(-1)^r q,t) \\
    =\exp{\sum_{n\geq 1}\frac{1}{n}\frac{(1-t_1^{-n}t_2^{-n})(1-t_1^{-n}t_3^{-n})(1-t_2^{-n}t_3^{-n})}{(1-t_1^{-n})(1-t_2^{-n})(1-t_3^{-n})}  \frac{1-\mathfrak{t}^{-rn}}{1-\mathfrak{t}^{-n}}\frac{1}{(1-\mathfrak{t}^{-\frac{rn}{2}}q^{-n})(1-\mathfrak{t}^{-\frac{rn}{2}}q^{n})}}.
\end{multline*}
Assume now that $\mathfrak{t}^{\frac{1}{2}}=e^{\pi i \frac{k}{r}}$, with $k\in \mathbb{Z}$; we have clearly that $\mathfrak{t}^{-\frac{rn}{2}}=(-1)^{kn} $. Moreover, we have
\[
 \frac{1-\mathfrak{t}^{-rn}}{1-\mathfrak{t}^{-n}}= \begin{cases} r, & \mbox{if } n\in \frac{r}{\gcd(r,k)}\BZ \\ 0, & \mbox{if } n\notin \frac{r}{\gcd(r,k)}\BZ \end{cases}
\]
In particular, if $n\in \frac{r}{\gcd(r,k)}\BZ$, we have
\[
\frac{(1-t_1^{-n}t_2^{-n})(1-t_1^{-n}t_3^{-n})(1-t_2^{-n}t_3^{-n})}{(1-t_1^{-n})(1-t_2^{-n})(1-t_3^{-n})}  =-1
\]
Setting $n=\frac{r}{\gcd(r,k)} m$, with $m\in \BZ$, we have
\begin{align*}
 \DT_r^{\KK}(\BA^3,(-1)^r q,t)&=\exp{\sum_{m\geq 1}\frac{1}{m}\gcd(r,k)\cdot \frac{-1}{(1-\overline{q}^{-m})(1-\overline{q}^m)}}
\end{align*}
where to ease notation we have set $\overline{q}=((-1)^kq)^{\frac{r}{\gcd(r,k)}}$. We conclude by using the description of the MacMahon function as a plethystic exponential
\[
\DT_r^{\KK}(\BA^3,(-1)^r q,t)=\mathsf M\left((-1)^{kr}q^{\frac{r}{\gcd(r,k)}}\right)^{\gcd(r,k)}.\qedhere
\]
\end{proof}

\begin{remark}\label{rem: elliptic indepedence}
A key technical point in the proof of the conjecture proposed in \cite[Equation (3.20)]{BBPT} is the assumption of the independence of $\DT_{r,k}^{\rm ell}(\BA^3)$ on $p$, as in Conjecture \ref{conj: non dependence under some limits of elliptic}. We strongly believe it should be possible to prove this assumption by exploiting  modular properties of the generating series of elliptic DT invariants. One could proceed by considering the integral representation of the DT invariants  given in \cite{BBPT}.
The analysis of the K-thoretic case, which we carried out in the proof of Theorem \ref{thm:elliptic}, shows that no dependence whatsoever is present in the limit $\mathfrak t^{1/2}=e^{\pi ik/r}$. As elliptic DT invariants take the form of meromorphic Jacobi forms, given by quotients of theta functions, poles in the equivariant parameters are only given by shifts along the lattice $\BZ+\tau\BZ$ of the poles found in K-theoretic DT invariants. Then $\DT_{r,k}^{\rm ell}(\BA^3)$, as a function of each of the equivariant parameters $v_i$, $i=1,\dots,r$, and $s_j$, $j=1,2,3$, is holomorphic on the torus $\BC/\BZ+\tau\BZ$, so it also carries no dependence on them. This observation may be not very surprising, if one considers the striking resemblance of the chiral virtual elliptic genus to the usual level$-N$ elliptic genus of almost complex manifolds, which is known to be often rigid. In our case, each $q$-term in the elliptic generating series restricted to $\TT_k$ would be now invariant under modular transformations on $\tau$, hence a constant in $p=e^{2\pi i\tau}$.
\end{remark}
\subsection{Relation to string theory}\label{subsec:Witten}
The definition for the elliptic version of Donaldson--Thomas invariants is motivated by an argument due to Witten \cite{Dirac_index}, which goes as follows: let $M$ be a $2k$-dimensional spin manifold, and take $\mathscr LM=C^0(S^1,M)$ to be the free loop space on $M$. Then $\mathscr LM$ always carries a natural $S^1$-action, given by the rotation of loops, so that fixed points under this action of $S^1$ will only be constant maps $S^1\to M$, and $(\mathscr LM)^{S^1}\cong M\hookrightarrow\mathscr LM$. One can then study the Dirac operator on $\mathscr LM$ by formally computing its index using fixed point formulae. In particular, if $D\colon \Gamma(S_+)\to\Gamma(S_-)$ is the Dirac operator on $M$ then
\[
\Ind(D)=\dim\ker D-\dim\coker D,
\]
and whenever $M$ admits the action of a compact Lie group $G$ one can define the $G$-equivariant index of $D$ as the virtual character
\[
\Ind_G(D)(g)=\Tr_{\ker D}g-\Tr_{\coker D}g,\qquad g\in G,
\]
which only depends on the conjugacy class of $g$ in $G$. In the case of Dirac operators on loop spaces over spin manifolds, a formal computation yields
\begin{align*}
\Ind_{S^1}(D)(q)&=q^{-\frac{d}{24}}\hat A(M)\cdot\ch\left(\bigotimes_{n\ge 1}\Sym^\bullet_{q^n}T_M\right)\cap[M],
\end{align*}
where $q$ denotes a topological generator of $S^1$. The $\hat A$-genus is the characteristic class which computes the index of the Dirac complex on a spin manifold. In general, if $E$ is any rank $r$ complex vector bundle on $M$, one can define $\hat A(E)$ in terms of the Chern roots $x_i$ of $E$ as
\[
\hat A(E)=\prod_{i=1}^r\frac{x_i/2}{\sinh x_i/2},
\]
and $\hat A(M)=\hat A(TM)$, which is completely analogous to the K--theoretic invariants we have been studying so far.
The previous formula can also be classically interpreted as the index of a twisted Dirac operator over the spin structure of $M$. It is worth noticing that the DT partition functions coming from physics are indeed interpreted as being computing indices of twisted Dirac operators, where the twist by a vector bundle $V\to M$ makes sense only if $w_2(T_M)=w_2(V)$ so as to extend $D$ to an operator $D\colon \Gamma(S_+\otimes V)\to\Gamma(S_-\otimes V)$. In this same spirit one might also justify the definition of the half-BPS elliptic genus in terms of computations of Euler characteristics of loop spaces over (compact) almost complex manifolds. Let then $X$ be a $d$-dimensional almost complex manifold, with holomorphic tangent bundle $T_X$, and whose corresponding free loop space will be denoted by $\mathscr LX$. As it was the case also in the previous situation, $\mathscr LX$ is naturally equipped with an $S^1$ action, whose fixed point will be $(\mathscr LX)^{S^1}\cong X\hookrightarrow\mathscr LX$. By formally applying the virtual localisation formula to the computation of the Euler characteristic of $\mathscr LX$ one gets
\[
\chi_{S^1}(\mathscr LX)=q^{-\frac{d}{12}}\td(X)\cdot\ch\left(\bigotimes_{n\ge 1}\Sym_{q^n}^\bullet (T_X\oplus\Omega_X)\right)\cap[X],
\]
which can also be seen as the index of a twisted ${\rm Spin}^c$-Dirac operator $\overline\partial+\overline\partial^\ast$. Moreover, if $c_1(T_X)=0\mod 2$, $X$ is also spin, and it is possible to compute the index of the Dirac operator on $\mathscr LX$ as before.

%%%%%%%%%%%%%%%%%%%%%%%%%%%%%%%%%%%%%%%%%%%%%%%%%%%
%%%%%%%%%%%%%%%%%%%%%%%%%%%%%%%%%%%%%%%%%%%%%%%%%%%
\section{Higher rank DT invariants of compact toric 3-folds}\label{sec:compact_section}
Let $X$ be a smooth \emph{projective} toric $3$-fold, along with an exceptional locally free sheaf $F$ of rank $r$. By \cite[Thm.~A]{Virtual_Quot}, the Quot scheme $\Quot_X(F,n)$ has a $0$-dimensional perfect obstruction theory, so that the rank $r$ Donaldson--Thomas invariant
\[
\DT_{F,n} = \int_{[\Quot_X(F,n)]^{\vir}}1\,\in\,\BZ
\]
is well defined. In this section we confirm the formula
\[
\sum_{n\geq 0}\DT_{F,n} q^n = \mathsf M((-1)^rq)^{r\int_X c_3(T_X \otimes K_X)},
\]
suggested in \cite[Conj.~3.5]{Virtual_Quot}, in the case where $F$ is \emph{equivariant}. This will prove Theorem \ref{mainthm:projective_toric} from the Introduction. The next subsection is an interlude on how to induce a torus action on the Quot scheme and on the associated universal short exact sequence starting from an equivariant structure on $F$. More details are given in \cite{Equivariant_Atiyah_Class}, including a proof that the obstruction theory obtained in \cite[Thm.~A]{Virtual_Quot} is equivariant (but see Proposition \ref{pot_global_equivariant} for a sketch).

%%%%%%%%%%%%%%%%%%%%%%%%%%%%%%%%%%%%%%%%%%%%%%%%%%%
\subsection{Inducing a torus action on the Quot scheme}\label{sec:induced_T-action_quot}
Let $X$ be a quasiprojective toric variety with torus $\BT\subset X$. Let $\sigma_X\colon \BT \times X \to X$ denote the action. If $F$ is a $\BT$-equivariant coherent sheaf on $X$, and $\mathrm{Q} = \Quot_X(F,n)$, then $\sigma_X$ has a canonical lift
\[
\sigma_{\mathrm{Q}}\colon \BT \times \mathrm{Q}\to \mathrm{Q}.
\]
This is proved in \cite[Prop.~4.1]{Kool_Fixed_Point_Loci}, but we sketch here the argument for the sake of completeness\footnote{We thank Martijn Kool for guiding us through the details of this construction.}. Let $p_2\colon \BT \times X \to X$ be the second projection and let $\vartheta\colon p_2^\ast F \simto \sigma_X^\ast F$ denote the chosen $\BT$-equivariant structure on $F$. Let $\pi_X\colon X\times \mathrm{Q} \to \mathrm{Q}$ and $\pi_{\mathrm{Q}}\colon X\times \mathrm{Q}\to \mathrm{Q}$ be the projections, and set $F_Q = \pi_X^\ast F$. Consider the universal exact sequence
\begin{equation}\label{diag_uni}
\begin{tikzcd}
& \mathcal S \into F_{\mathrm{Q}} \overset{u}{\onto} \mathcal T\arrow[dash]{d} & \\ 
& X\times \mathrm{Q}\arrow[swap]{dl}{\pi_X}\arrow{dr}{\pi_{\mathrm{Q}}} & \\
X & & \mathrm{Q} 
\end{tikzcd}
\end{equation}
and note that there is a commutative diagram
\[
\begin{tikzcd}[column sep = large,row sep = large]
\BT\times X\times \mathrm{Q}\arrow[swap]{d}{p_{12}}\arrow{r}{\sigma_X\times\id_{\mathrm{Q}}} & X\times \mathrm{Q}\arrow{d}{\pi_X} \\
\BT\times X\arrow{r}{\sigma_X} & X
\end{tikzcd}
\]
yielding an identity $(\sigma_X\times \id_{\mathrm{Q}})^\ast F_{\mathrm{Q}} = p_{12}^{\ast}\sigma_X^\ast F$.
The induced surjection
\[
(\sigma_X\times \id_{\mathrm{Q}})^\ast u \circ p_{12}^\ast\vartheta\colon F_{\BT\times \mathrm{Q}} \simto p_{12}^\ast\sigma_X^\ast F \onto (\sigma_X\times \id_{\mathrm{Q}})^\ast \mathcal T
\]
defines a $\BT\times \mathrm{Q}$-valued point of $\mathrm{Q}$, i.e.~a morphism 
\[
\sigma_{\mathrm{Q}}\colon \BT \times \mathrm{Q} \to \mathrm{Q}.
\]
It is straightforward to verify that $\sigma_{\mathrm{Q}}$ satisfies the axioms for a $\BT$-action.

\smallbreak
Next, we explain how to make the universal exact sequence \eqref{diag_uni} $\BT$-equivariant. The universal property of the Quot scheme applied to $\sigma_{\mathrm{Q}}$ implies that there is an isomorphism of surjections
\[
(\id_X \times \sigma_{\mathrm{Q}})^\ast u \,\simto\,(\sigma_X\times \id_{\mathrm{Q}})^\ast u \circ p_{12}^\ast\vartheta.
\]
This means that there is a commutative diagram
\begin{equation}\label{eqn:univ_property_quot}
\begin{tikzcd}
(\id_X \times \sigma_{\mathrm{Q}})^\ast F_{\mathrm{Q}}\arrow[equal]{d}\arrow[two heads]{rrr}{(\id_X \times \sigma_{\mathrm{Q}})^\ast u} & & & (\id_X \times \sigma_{\mathrm{Q}})^\ast\mathcal T \isoarrow{d} \\
p_{12}^\ast p_2^\ast F\arrow{r}{p_{12}^\ast\vartheta} & p_{12}^\ast \sigma_X^\ast F\arrow[two heads]{rr}{(\sigma_X\times \id_{\mathrm{Q}})^\ast u} & &  (\sigma_X\times \id_{\mathrm{Q}})^\ast \mathcal T
\end{tikzcd}
\end{equation}
where we used the identity $(\sigma_X\times \id_{\mathrm{Q}})^\ast F_{\mathrm{Q}} = p_{12}^{\ast}\sigma_X^\ast F$ in the bottom row.

Consider the morphism
\[
\varphi \colon X\times \mathrm{Q} \times \BT \to X\times \mathrm{Q}, \quad (x,f,t) \mapsto (\sigma_X(t, x),\sigma_{\mathrm{Q}}(t^{-1},f)).
\]
We view this as a $\BT$-action on $X\times \mathrm{Q}$. Note that $\pi_X\circ \varphi = \sigma_X\circ p_{12}$. 
The moduli map $\mathrm{Q} \times \BT \to \mathrm{Q}$ corresponding to the family of quotients 
\[
\begin{tikzcd}
\varphi^*u\circ p_{12}^\ast\vartheta\colon F_{\mathrm{Q} \times \BT} = p_{12}^\ast p_2^\ast F \simto p_{12}^\ast \sigma_X^\ast F = \varphi^\ast F_{\mathrm{Q}} \arrow[two heads]{r}{\varphi^*u} & \varphi^\ast \mathcal T
\end{tikzcd}
\]
is easily seen to agree with the first projection $p_1\colon \mathrm{Q} \times \BT \to \mathrm{Q}$. Indeed, if $\mathcal T_f$ denotes the quotient of $F$ corresponding to a point $f = [F \onto \mathcal T_f] \in \mathrm{Q}$, it is immediate to see that
\[
\varphi^\ast \mathcal T\big|_{X\times f\times t} = \mathcal T_{\sigma_{\mathrm{Q}}(\sigma_{\mathrm{Q}}(t,f),t^{-1})} = \mathcal T_{\sigma_{\mathrm{Q}}(tt^{-1},f)} = \mathcal T_f
\]
for all $t \in \BT$.

Let then $q = \id_X \times p_1\colon X\times \mathrm{Q} \times \BT \to X\times \mathrm{Q}$ be the projection. Since $\varphi^*u\circ p_{12}^\ast\vartheta$ corresponds to the projection $p_1\colon \mathrm{Q} \times \BT \to \mathrm{Q}$, by the universal property of $(\mathrm{Q},u)$ we obtain an isomorphism of surjections $q^\ast u \simto \varphi^\ast u$ that, after setting $\mathcal S = \ker (u\colon F_{\mathrm{Q}}\onto \mathcal T)$,  we can extend to an isomorphism of short exact sequences
\[
\begin{tikzcd}
q^\ast \mathcal S \arrow[hook]{rr}\isoarrow{d} & & 
 F_{\mathrm{Q} \times \BT} \arrow[two heads]{rr}{q^\ast u}\isoarrow{d} & & q^\ast \mathcal T\isoarrow{d} \\
\varphi^\ast \mathcal S \arrow[hook]{rr} & &  \varphi^\ast F_{\mathrm Q}\arrow[two heads]{rr}{\varphi^*u} & & \varphi^\ast \mathcal T 
\end{tikzcd}
\]
on $X\times \mathrm{Q}\times \mathbb T$, where the middle vertical isomorphism is $p_{12}^\ast\vartheta$ and is a $\BT$-equivariant structure on $F_{\mathrm{Q}}$ because $\vartheta$ is. The diagram allows us to conclude that a $\BT$-equivariant structure on $F$ induces a canonical $\BT$-equivariant structure on the universal short exact sequence
\[
0 \to \mathcal S \to F_{\mathrm{Q}} \to \mathcal T \to 0.
\]

%%%%%%%%%%%%%%%%%%%%%%%%%%%%%%%%%%%%%%%%%%%%%%%%%%%
\subsection{The (equivariant) obstruction theory}
Throughout this subsection, $F$ denotes an exceptional locally free sheaf of rank $r$ on a smooth projective toric $3$-fold $X$. In other words, $F$ is simple, i.e.~$\Hom(F,F) = \BC$, and $\Ext^i(F,F)=0$ for $i>0$. 

By \cite[Thm.~A]{Virtual_Quot}, there is a $0$-dimensional perfect obstruction theory 
\begin{equation}\label{global_pot}
\BE \to \BL_{\Quot_X(F,n)},
\end{equation}
governed by 
\[
    \Def\big|_{[S]} = \Ext^1(S,S),\quad
    \Obs\big|_{[S]} =\Ext^2(S,S)
\]
around a point $[S] \in \Quot_X(F,n)$. We set $\mathrm{Q} = \Quot_X(F,n)$ for brevity, and we denote by $\pi_{\mathrm{Q}}$ and $\pi_X$ the projections from $X\times \mathrm Q$, as in \eqref{diag_uni}. Note that $\omega_{\pi_{\mathrm{Q}}} = \pi_X^\ast \omega_X$.

As we recall during the (sketch of) proof of Proposition \ref{pot_global_equivariant} below, we have
\[
\BE = \RR\pi_{\mathrm Q\,\ast}(\RRlHom(\mathcal S,\mathcal S)_0\otimes \omega_{\pi_{\mathrm Q}})[2]
\]
where $\RRlHom(\mathcal S,\mathcal S)_0$ is the shifted cone of the trace map $\tr\colon\RRlHom(\mathcal S,\mathcal S) \to \OO_{X\times \mathrm Q}$. 
\begin{prop}\label{prop:global pot K theory}
There is an identity
\[
 \BE^\vee=\mathbf{R}\pi_{\mathrm Q\,\ast}\RRlHom(F_{\mathrm Q},F_{\mathrm Q})-\mathbf{R}\pi_{\mathrm Q\,\ast}\RRlHom\mathcal (\mathcal S,\mathcal S)\in K_0(\Quot_X(F,n)).
\]
\end{prop}

\begin{proof}
As in the proof of \cite[Theorem 2.5]{Virtual_Quot} we have
\begin{align*}
    \BE\,&=\,  (\RR\pi_{\mathrm Q\,\ast}\RRlHom(\mathcal S,\mathcal S)_0)^\vee[-1]\\
   \,&=\, (\RR\pi_{\mathrm Q\,\ast}\RRlHom(\mathcal S,\mathcal S))^\vee[-1]- (\RR\pi_{\mathrm Q\,\ast}\RRlHom(\OO,\OO))^\vee[-1]\\
     \,&=\,( \RR\pi_{\mathrm Q\,\ast}\RRlHom(\mathcal S,\mathcal S))^\vee[-1]- (\RR\pi_{\mathrm Q\,\ast}\RRlHom(F_{\mathrm Q},F_{\mathrm Q}))^\vee[-1],
\end{align*}
where the last identity uses that $F$ is an exceptional sheaf.
\end{proof}
The following result is proved in \cite[Thm.~B]{Equivariant_Atiyah_Class} in greater generality, but we sketch a proof here for the reader's convenience. Denote by $\BT = (\BC^\ast)^3$ the torus of $X$.

\begin{prop}\label{pot_global_equivariant}
Let $(X,F)$ be a pair consisting of a smooth projective toric $3$-fold $X$ along with an exceptional locally free $\BT$-equivariant sheaf $F$. Then the perfect obstruction theory \eqref{global_pot} on $\Quot_X(F,n)$ is $\BT$-equivariant.
\end{prop}

The definition of equivariant obstruction theory was recalled in Definition \ref{def:equivariant_pot}.

\begin{proof}
As we explained in \S\,\ref{sec:induced_T-action_quot}, by equivariance of $F$ the Quot scheme $\mathrm Q = \Quot_X(F,n)$ inherits a canonical $\BT$-action with respect to which the universal short exact sequence $\mathcal S \into F_{\mathrm{Q}} \onto \mathcal T$ can be made $\BT$-equivariant. On the other hand, the perfect obstruction theory \eqref{global_pot} is obtained by projecting the truncated Atiyah class
\[
\At_{\mathcal S} \,\in\, \Ext^1_{X \times \mathrm Q}(\mathcal S,\mathcal S\otimes \BL_{X \times \mathrm Q})
\]
onto the Ext group
\begin{equation}\label{Atiyah_Journey}
\begin{split}
    \Ext^1_{X \times \mathrm Q}(\RRlHom(\mathcal S,\mathcal S)_0,\pi_{\mathrm Q}^\ast \BL_{\mathrm Q}) \,\,&=\,\, \Ext^{-2}_{X\times \mathrm Q}(\RRlHom(\mathcal S,\mathcal S)_0 \otimes \omega_{\pi_{\mathrm Q}},\pi_{\mathrm Q}^\ast \BL_{\mathrm Q} \otimes \omega_{\pi_{\mathrm Q}}[3]) \\
    \,\,&=\,\, \Ext^{-2}_{X\times \mathrm Q}(\RRlHom(\mathcal S,\mathcal S)_0 \otimes \omega_{\pi_{\mathrm Q}},\pi_{\mathrm Q}^! \BL_{\mathrm Q}) \\
    \,\,&\cong \,\,\Ext^{-2}_{\mathrm Q}(\RR\pi_{\mathrm Q\,\ast}(\RRlHom(\mathcal S,\mathcal S)_0 \otimes \omega_{\pi_{\mathrm Q}}),\BL_{\mathrm Q}).
\end{split}
\end{equation}
The last isomorphism is Grothendieck duality along the smooth projective morphism $\pi_{\mathrm Q}$. Now we need three ingredients to finish the proof:

\begin{itemize}
    \item [$\circ$] The Atiyah class $\At_{\mathcal S}$ is a $\BT$-invariant extension,
    \item [$\circ$] Grothendieck duality preserves $\BT$-invariant extensions, and
    \item [$\circ$] $\BT$-invariant extensions correspond to morphisms in the equivariant derived category.
\end{itemize}
These assertions are proved in \cite{Equivariant_Atiyah_Class}.
\end{proof}

We let $\Delta(X)$ denote the set of vertices in the Newton polytope of the toric $3$-fold $X$. Then
\[
X^{\BT} = \Set{p_\alpha | \alpha \in \Delta(X)} \subset X
\]
will denote the fixed locus of $X$.
For a given vertex $\alpha$, let $U_\alpha \cong \BA^3$ be the canonical chart containing the fixed point $p_\alpha$. The $\BT$-action on this chart can be taken to be the standard action \eqref{standard_action}.
For every $\alpha$, there is a $\BT$-equivariant open immersion
\[
\iota_{n,\alpha}\colon \Quot_{U_\alpha}(F|_{U_\alpha},n)\into \mathrm Q = \Quot_X(F,n)
\]
parametrising quotients whose support is contained in $U_\alpha$. We think of $F|_{U_\alpha}$ as an equivariant sheaf on $\BA^3$, hence of the form described in \eqref{eqn:sheaf_with_equiv_weights}. We denote by $\BE_{n,\alpha}^{\crit}$ the critical obstruction theory on $\Quot_{U_\alpha}(F|_{U_\alpha},n)$ from Proposition \ref{prop:SPOT_A^3}.

It is natural to ask whether the restriction of the global perfect obstruction theory \eqref{global_pot} along $\iota_{n,\alpha}$ agrees with the critical symmetric perfect obstruction theory described in \S\,\ref{sec:critical_structure_Quot} (see Conjecture \ref{conj:pot_restricted}). However, what we really need is the following weaker result.

\begin{prop}\label{restriction of class in K theory}
Let $\BE \in K_0(\mathrm{Q})$ be the class of the global perfect obstruction theory \eqref{global_pot}. Then
\[
\BE_{n,\alpha}^{\crit}\,=\,\iota_{n,\alpha}^\ast \BE\,\in\,K_0(\Quot_{U_\alpha}(F|_{U_\alpha},n)).
\]
Considering the two obstruction theories as $\BT$-equivariant, the same identity holds in equivariant K-theory:
\[
\BE_{n,\alpha}^{\crit}\,=\,\iota_{n,\alpha}^\ast \BE\,\in\,K_0^{\BT}(\Quot_{U_\alpha}(F|_{U_\alpha},n)).
\]
\end{prop}

\begin{proof}
The chart $U_\alpha$ is Calabi--Yau, so by \cite[Prop.~2.9]{Virtual_Quot} the induced perfect obstruction theory $\iota_{n,\alpha}^\ast \BE$ is symmetric. Since by Remark \ref{rmk:K-class_of_symmetric_pot} all symmetric perfect obstruction theories share the same class in K-theory, the first statement follows.

To prove the $\BT$-equivariant equality, we need a slightly more refined analysis. Just for this proof, let us shorten
\[
\BE_{\crr} = \BE_{n,\alpha}^{\crit}\quad\textrm{and}\quad \BE = \iota_{n,\alpha}^\ast \BE,
\]
to ease notation. We know by Diagram \eqref{equivariant_symmetric_POT_quot} that we can write
\begin{equation}\label{fucking_E_cr}
\BE_{\crr} = \big[\mathfrak t\otimes T_{\NCQuot_{r}^n}\big|_{\mathrm{Q}} \to \Omega_{\NCQuot_{r}^n}\big|_{\mathrm{Q}}\big]=\Omega - \mathfrak t\otimes T\,\in\,K_0^{\BT}(\Quot_{U_\alpha}(F|_{U_\alpha},n)),
\end{equation}
where $\Omega$ (resp.~$T$) denotes the cotangent sheaf (resp.~the tangent sheaf) of $\Quot_{U_\alpha}(F|_{U_\alpha},n)$, equipped with its natural equivariant structure. 

Let $\pi\colon U_\alpha \times \Quot_{U_\alpha}(F|_{U_\alpha},n) \to \Quot_{U_\alpha}(F|_{U_\alpha},n)$ be the projection, let $S$ be the universal kernel living on the product and set $\mathfrak{t}_{\pi} = \pi^\ast \mathfrak{t}^{-1}$.
By definition,
\[
\BE = \RR \pi_\ast (\RRlHom(S,S)_0\otimes \omega_\pi)[2].
\]
The equivariant isomorphism $\omega_\pi \simto \OO \otimes \mathfrak{t}_{\pi}^{-1}$ along with the projection formula yield
\begin{equation}\label{eqn:equivariant_POT1837}
\mathfrak{t}^{-1} \otimes \BE \simto \RR \pi_\ast \RRlHom(S,S)_0[2].
\end{equation}
We next show the right hand side is canonically isomorphic to $\BE^\vee[1]$. We have
\begin{align*}
  \BE^\vee[1] &= \RRlHom(\RR \pi_\ast (\RRlHom(S,S)_0 \otimes \omega_\pi),\OO)[-1] & \small{\textrm{definition of }(-)^\vee}\\
  &= \RR\pi_\ast \RRlHom(\RRlHom(S,S)_0 \otimes \omega_\pi,\omega_{\pi}[3])[-1] & \small{\textrm{Grothendieck duality}} \\
  &= \RR\pi_\ast \RRlHom(\RRlHom(S,S)_0,\OO)[2] & \small{\textrm{shift}}\\
  &= \RR\pi_\ast\RRlHom(S,S)_0^\vee [2] & \small{\textrm{definition of }(-)^\vee} \\
  &=\RR\pi_\ast\RRlHom(S,S)_0 [2] & \small{\RRlHom(S,S)_0\textrm{ is self-dual}}
\end{align*}
in the derived category of $\BT$-equivariant coherent sheaves on $\Quot_{U_\alpha}(F|_{U_\alpha},n)$, 
which by \eqref{eqn:equivariant_POT1837} proves that
\[
\mathfrak{t}^{-1} \otimes \BE \cong  \BE^\vee[1].
\]
We thus have 
\[
\Omega \cong h^0(\BE) \cong \mathfrak t \otimes h^0(\BE^\vee[1]) \cong \mathfrak t \otimes \lExt^2_\pi(S,S),
\]
where we use the standard notation $\lExt^i_\pi(-,-)$ for the $i$th derived functor of $\pi_\ast\circ \lHom(-,-)$. We conclude
\begin{align*}
 \BE 
 &= h^0(\BE) - h^{-1}(\BE) \\
 &= \Omega - h^1(\BE^\vee)^\vee \\
 &= \Omega - h^0(\BE^\vee[1])^\vee \\
 &= \Omega - \lExt^2_\pi(S,S)^\vee \\
 &= \Omega - (\mathfrak{t}^{-1}\otimes \Omega)^\vee \\
 &= \Omega - \mathfrak t \otimes T \\
 &= \BE_{\crr}.\qedhere
\end{align*}
\end{proof}

\subsection{The fixed locus of the Quot scheme and its virtual class}
In this subsection we describe $\Quot_X(F,n)^{\BT}$ and we compute its virtual fundamental class, obtained via Proposition \ref{pot_global_equivariant}. 

If $\mathbf n$ denotes a generic tuple $\set{n_\alpha | \alpha \in \Delta(X)}$ of non-negative integers, we set $|\mathbf n| = \sum_{\alpha \in \Delta(X)}n_\alpha$.

\begin{lemma}\label{T_1 fixed locus projective}
There is a scheme-theoretic identity
\[
\Quot_X(F,n)^{\BT} =  \coprod_{|\mathbf n| = n}\prod_{\alpha \in \Delta(X)} \Quot_{U_\alpha}(F|_{U_\alpha},n_\alpha)^{\BT}. 
\]
\end{lemma}

\begin{proof}
Let $B$ be a (connected) scheme over $\BC$. Let $F_B$ be the pullback of $F$ along the first projection $X\times B \to X$, and fix a $B$-flat family of quotients
\[
\rho\colon F_B \onto \mathcal T
\]
defining a $B$-valued point $B \to \Quot_X(F,n)^{\BT}$. Then, by restriction, we obtain, for each $\alpha \in\Delta(X)$, a $B$-flat family of quotients
\begin{equation}\label{F_restricted_to_Ualpha}
\rho_\alpha \colon F_B\big|_{U_\alpha\times B} \onto \mathcal T_\alpha = \mathcal T\big|_{U_\alpha\times B},
\end{equation}
and we let $n_\alpha$ be the length of the fibres of $\mathcal T_\alpha$. Each $\rho_\alpha$ corresponds to a $B$-valued point $g_\alpha\colon B \to \Quot_{U_\alpha}(F|_{U_\alpha},n_\alpha)^{\BT}$, thus we obtain a $B$-valued point 
\[
(g_\alpha)_\alpha \colon B \to \prod_{\alpha \in \Delta(X)} \Quot_{U_\alpha}(F|_{U_\alpha},n_\alpha)^{\BT}.
\]
Note that the original family $\mathcal T$ is recovered as the direct sum $\bigoplus_\alpha \mathcal T_\alpha$, in particular $n = \sum_\alpha n_\alpha$.
Conversely, suppose given a tuple of $B$-families of $\BT$-fixed quotients 
\[
\left(\left(F|_{U_{\alpha}}\right)_B
\onto \mathcal T_\alpha\right)_\alpha.
\]
We obtain $B$-valued points
\[
B \to \Quot_{U_\alpha}(F|_{U_\alpha},n_\alpha)^{\BT} \subset \Quot_X(F,n_\alpha)^{\BT}.
\]
Since the support of these families is disjoint, we can form the direct sum
\[
\mathcal T = \bigoplus_\alpha \mathcal T_\alpha
\]
to obtain a new $B$-flat family, representing a $B$-valued point of $\Quot_X(F,n)^{\BT}$, as required.
\end{proof}

%%%%%%%%%%%%%%%%%%%%%%%%%%%%%%%%%%%%%%%%%%%%%%%%%%%%%%%%%%%%%%%

Our next goal is to show that, under the identification of Lemma \ref{T_1 fixed locus projective}, the induced virtual fundamental class of the $\mathbf n$-th connected component of $\Quot_X(F,n)^{\BT}$ is the box product of the  virtual fundamental classes of $\Quot_{U_\alpha}(F|_{U_\alpha},n_\alpha)^{\BT}$, whose perfect obstruction theory is the $\BT$-fixed part of the critical one, studied in \S\,\ref{sec:critical_structure_Quot}. For the rest of the section we restrict  our attention to each connected component 
\begin{equation}\label{fixed_n}
\Quot_X(F,n)^{\BT}_{\mathbf n} = \prod_{\alpha \in \Delta(X)} \Quot_{U_\alpha}(F|_{U_\alpha},n_\alpha)^{\BT} \subset \Quot_X(F,n)^{\BT},
\end{equation}
and we denote by
\begin{equation}\label{universal_structures}
\begin{tikzcd}
\mathcal S \into \mathcal F \onto \mathcal{T}\arrow[dash]{d} & \mathcal S_\alpha \into \mathcal F_\alpha \onto \mathcal{T}_\alpha \arrow[dash]{d}\\ 
X\times \Quot_X(F,n)^{\BT}_{\mathbf n}\arrow{d}{\pi} & U_\alpha \times \Quot_{U_\alpha}(F|_{U_\alpha},n_\alpha)^{\BT}\arrow{d}{\pi_\alpha} \\
\Quot_X(F,n)^{\BT}_{\mathbf n} \arrow{r}{p_\alpha} & \Quot_{U_\alpha}(F|_{U_\alpha},n_\alpha)^{\BT}
\end{tikzcd}
\end{equation}
the various universal structures and projection maps between these moduli spaces. For instance, $\mathcal F_\alpha$ is the pullback of $F|_{U_\alpha}$ along the projection $U_\alpha \times \Quot_{U_\alpha}(F|_{U_\alpha},n_\alpha)^{\BT}\to U_\alpha$. Let $\BE_{\mathbf n}$ be the restriction of $\BE \in \derived(\Quot_X(F,n))$ to the closed subscheme $\Quot_X(F,n)^{\BT}_{\mathbf n}  \subset \Quot_X(F,n)$.

\begin{prop}\label{prop: relative cech}
There are identities in  $K_0^{\BT}(\Quot_X(F,n)_{\mathbf n}^{\BT})$ 
\begin{multline*}
    \BE_{\mathbf n}^\vee = \mathbf{R}\pi_*\RRlHom(\mathcal F,\mathcal F)-\mathbf{R}\pi_*\RRlHom(\mathcal{S},\mathcal{S}) \\ =\sum_{\alpha\in \Delta(X)}p_\alpha^*\left(\mathbf{R}{\pi_{\alpha}}_*\RRlHom(\mathcal F_\alpha,\mathcal F_\alpha)-\mathbf{R}{\pi_\alpha}_*\RRlHom(\mathcal{S}_\alpha,\mathcal{S}_\alpha)\right).
\end{multline*}
\end{prop}

\begin{proof}
Exploiting the universal short exact sequence
    \[   
    0\to \mathcal{S}\to \mathcal F\to \mathcal{T}\to 0
    \]
on $X\times \Quot_X(F,n)^{\BT}_{\mathbf n} \subset X\times \Quot_X(F,n)$, and Proposition \ref{prop:global pot K theory}, we may write
    \begin{multline*}
    \BE_{\mathbf n}^\vee = \BE^\vee\big|_{\Quot_X(F,n)^{\BT}_{\mathbf n}} = \mathbf{R}\pi_*\RRlHom(\mathcal F,\mathcal F)-\mathbf{R}\pi_*\RRlHom(\mathcal{S},\mathcal{S}) \\
    = \mathbf{R}\pi_*\RRlHom(\mathcal{S},\mathcal{T})+\mathbf{R}\pi_*\RRlHom(\mathcal{T},\mathcal{S})+\mathbf{R}\pi_*\RRlHom(\mathcal{T},\mathcal{T}).        
    \end{multline*}
    Similarly, we have
    \begin{multline}\label{eqn:rhom_alpha}
        \mathbf{R}\pi_{\alpha*}\RRlHom(\mathcal F_{\alpha},\mathcal F_{\alpha})-\mathbf{R}\pi_{\alpha*}\RRlHom(\mathcal{S}_{\alpha},\mathcal{S}_{\alpha}) \\
        = \mathbf{R}\pi_{\alpha*}\RRlHom(\mathcal{S}_{\alpha},\mathcal{T}_{\alpha})+\mathbf{R}\pi_{\alpha*}\RRlHom(\mathcal{T}_{\alpha},\mathcal{S}_{\alpha})+\mathbf{R}\pi_{\alpha*}\RRlHom(\mathcal{T}_{\alpha},\mathcal{T}_{\alpha}).
    \end{multline}
In the following, we write $(G_1,G_2)$ for any of the three pairs $(\mathcal{S},\mathcal{T})$, $(\mathcal{T},\mathcal{S})$ or $(\mathcal{T},\mathcal{T})$. Applying the Grothendieck spectral sequence yields
\begin{align*}
    \mathbf{R}\pi_*\RRlHom(G_1,G_2)&= \sum_{i,j}(-1)^{i+j}\mathbf{R}^i\pi_* \lExt^j(G_1,G_2)\\
    &=\sum_{j}(-1)^{j}\pi_* \lExt^j(G_1,G_2),
\end{align*}
where we used cohomology and base change along with the fact that $\RR^{i} \pi_\ast$ of a $0$-dimensional sheaf vanishes for $i>0$. The standard \v{C}ech cover $\{U_\alpha\}_{\alpha\in \Delta(X)}$ of $X$ pulls back to a \v{C}ech cover $\{V_\alpha\}_{\alpha\in \Delta(X)}$ of $X\times \Quot_X(F,n)^{\BT}_{\mathbf n}$, where $V_\alpha=U_\alpha\times \Quot_X(F,n)^{\BT}_{\mathbf n} $. For a finite family of indices $I\subset \BN$, set $V_I=\bigcap_{\alpha\in I} V_\alpha$ and let  $j_I\colon V_I\to X\times \Quot_X(F,n)^{\BT}_{\mathbf n}$ be the natural open immersion. We have a \v{C}ech resolution $\lExt^j(G_1,G_2)\to \mathfrak{C}^\bullet $, where $\mathfrak{C}^\bullet$ is defined degree-wise (see e.g.~\cite[Lemma III.4.2]{Hartshorne_AG}) by
\[
\mathfrak{C}^k=\bigoplus_{|I|=k+1}{j_I}_\ast  j_I^\ast \lExt^j(G_1,G_2).
\]
Notice that $\mathcal{T}$    vanishes on the restriction to any double intersection $U_{\alpha\beta}\times \Quot_X(F,n)^{\BT}_{\mathbf n}$, where $U_{\alpha\beta}=U_\alpha\cap U_\beta$. This implies that the only contribution of the \v{C}ech cover is given by $\mathfrak{C}^0$, thus
\begin{align*}
\mathbf{R}\pi_*\RRlHom(G_1,G_2)
&= \sum_{j} (-1)^{j}\pi_* \sum_{\alpha\in \Delta(X)}{j_\alpha}_\ast j_\alpha^\ast\lExt^j(G_1,G_2)\\
&= \sum_{\alpha\in \Delta(X)} 
\sum_{j}(-1)^{j}(\pi\circ {j_\alpha})_\ast j_\alpha^\ast\lExt^j(G_1,G_2).
\end{align*}
Consider the following cartesian diagram
\[  
\begin{tikzcd}
U_\alpha\times \Quot_X(F,n)^{\BT}_{\mathbf n}\arrow[hook]{d}{j_\alpha} \arrow{r}{\widetilde p_\alpha} & U_\alpha \times \Quot_{U_\alpha}(F|_{U_\alpha},n_\alpha)^{\BT}\arrow[hook]{d}{} \arrow[dd,bend left=75, "\pi_\alpha "]\\
X\times \Quot_X(F,n)^{\BT}_{\mathbf n}\arrow{d}{\pi} \arrow{r}{} & X\times \Quot_{U_\alpha}(F|_{U_\alpha},n_\alpha)^{\BT} \arrow{d}\\
 \Quot_X(F,n)^{\BT}_{\mathbf n} \arrow{r}{p_\alpha}&  \Quot_{U_\alpha}(F|_{U_\alpha},n_\alpha)^{\BT}
\end{tikzcd}
\]
As it was already clear from the proof of Lemma \ref{T_1 fixed locus projective}, the universal short exact sequences in Diagram \eqref{universal_structures} satisfy $j_\alpha^\ast(\mathcal S \into \mathcal F \onto \mathcal T) = \widetilde p_\alpha^\ast(\mathcal S_\alpha \into \mathcal F_\alpha \onto \mathcal T_\alpha)$. If $(G_{1 \alpha},G_{2 \alpha})$ denotes any of the pairs belonging to the set $\set{(\mathcal{S}_{\alpha},\mathcal{T}_{\alpha}), (\mathcal{T}_{\alpha},\mathcal{S}_{\alpha}),(\mathcal{T}_{\alpha},\mathcal{T}_{\alpha})}$, we can write
\begin{align*}
   j_\alpha^\ast\lExt^j(G_1,G_2)
   &= \mathbf{L} j_\alpha^\ast\lExt^j(G_1,G_2) \\
   &= \lExt^j(\mathbf{L} j_\alpha^\ast G_1,\mathbf{L} j_\alpha^\ast G_2) \\
   &=\lExt^j(\widetilde{p}_\alpha^\ast G_{1,\alpha},\widetilde{p}_\alpha^\ast G_{2,\alpha}) \\
   &=\widetilde p_\alpha^*\lExt^j({G_1}_\alpha,{G_2}_\alpha).
\end{align*} 
We deduce, by flat base change,
\[
(\pi\circ {j_\alpha})_\ast j_\alpha^\ast\lExt^j(G_1,G_2)=(\pi\circ {j_\alpha})_\ast\widetilde p_\alpha^*\lExt^j({G_1}_\alpha,{G_2}_\alpha) = p_\alpha^{\ast}\pi_{\alpha\,\ast} \lExt^j({G_1}_\alpha,{G_2}_\alpha).
\]
  Combining again the Grothendieck spectral sequence, cohomology and base change and the vanishing of higher derived pushforwards on $0$-dimensional sheaves, we conclude that
\begin{align*}
 \mathbf{R}\pi_*\RRlHom(G_1,G_2)&\,=\,
 \sum_{\alpha\in \Delta(X)} 
p_\alpha^{\ast}\sum_{j}(-1)^{j}\pi_{\alpha\,\ast} \lExt^j({G_1}_\alpha,{G_2}_\alpha)\\
&\,=\,\sum_{\alpha\in \Delta(X)}p_\alpha^* \mathbf{R}{\pi_\alpha}_* \RRlHom({G_1}_\alpha,{G_2}_\alpha).
\end{align*}
Now the result follows from Equation \eqref{eqn:rhom_alpha}.
\end{proof}

\begin{corollary}\label{cor: pot as box product}
The virtual fundamental class of $\Quot_X(F,n)^{\BT}_{\mathbf n}$ is expressed as the product of the virtual fundamental classes
\[
\left[\Quot_X(F,n)^{\BT}_{\mathbf n}\right]^{\vir} = \prod_{\alpha \in \Delta(X)} p_\alpha^\ast \left[ \Quot_{U_\alpha}(F|_{U_\alpha},n_\alpha)^{\BT}\right]^{\vir}.
\]
\end{corollary}

Before we prove the corollary, let us explain what virtual classes are involved. The left hand side is the virtual class induced by the $\BT$-fixed obstruction theory
\[
\BE_{\mathbf n}^{\BT\textrm{-}\fix} \to \BL_{\Quot_X(F,n)^{\BT}_{\mathbf n}},
\]
whereas $\left[ \Quot_{U_\alpha}(F|_{U_\alpha},n_\alpha)^{\BT}\right]^{\vir}$ is the virtual class induced by the obstruction theory 
\[
\iota_{n_\alpha,\alpha}^\ast \BE \to \BL_{\Quot_{U_\alpha}(F|_{U_\alpha},n_\alpha)}
\]
by restricting to the $\BT$-fixed locus and taking the $\BT$-fixed part. Note that by Proposition \ref{restriction of class in K theory}, the perfect obstruction theory 
\[
\BE_{n_\alpha,\alpha}^{\crit}\big|_{\Quot_{U_\alpha}(F|_{U_\alpha},n_\alpha)^{\BT}}^{\BT\textrm{-}\fix} \to \BL_{\Quot_{U_\alpha}(F|_{U_\alpha},n_\alpha)^{\BT}}
\]
induces the same virtual class. This follows from the general fact that the (equivariant) virtual fundamental class depends only on the class in (equivariant) K-theory of the perfect obstruction theory --- cf.~\cite[Theorem 4.6]{Siebert}, where all the ingredients are naturally equivariant.

\begin{proof}
The statement follows by taking $\BT$-fixed parts in Proposition \ref{prop: relative cech} and by Siebert's result \cite[Theorem 4.6]{Siebert} mentioned above.
\end{proof}

\subsection{Higher rank Donaldson--Thomas invariants of compact 3-folds}\label{section on compact invariants}
For a pair $(X,F)$ consisting of a smooth projective toric $3$-fold $X$ and an exceptional locally free sheaf $F$, the perfect obstruction theory \eqref{global_pot} gives rise to a $0$-dimensional virtual fundamental class
\[
\bigl[\Quot_X(F,n)\bigr]^{\vir} \in A_0(\Quot_X(F,n)),
\]
allowing one to define higher rank Donaldson--Thomas invariants
\[
\DT_{F,n} = \int_{[\Quot_X(F,n)]^{\vir}}1 \in \BZ.
\]
Define the generating function
\[
\DT_F(q) = \sum_{n\geq 0} \DT_{F,n} q^n.
\]

We next compute this series in the case of a $\BT$-equivariant exceptional locally free sheaf, thus proving Theorem \ref{mainthm:projective_toric} from the Introduction.

\begin{theorem}\label{thm for toric proj}
Let $(X,F)$ be a pair consisting of a smooth projective toric $3$-fold $X$ along with an exceptional $\BT$-equivariant locally free sheaf $F$. Then
\[
\DT_F(q) = \mathsf M((-1)^rq)^{r\int_{X}c_3(T_X\otimes K_X)}.
\]
\end{theorem}

\begin{proof}
Set $Q=\Quot_X(F,n) $ and $Q_\alpha=\Quot_{U_\alpha}(F|_{U_\alpha},n_\alpha)$. Since by Proposition \ref{pot_global_equivariant} the perfect obstruction theory on $Q$ is $\BT$-equivariant, we can apply the virtual localisation formula 
\[
\DT_{F,n} = 
\int_{[Q^{\BT}]^{\vir}}e^{\BT}(-N_{Q^{\BT}/Q}^{\vir}),
\]
where $N_{Q^{\BT}/Q}^{\vir} $ is the virtual normal bundle on the $\BT$-fixed locus computed in Lemma \ref{T_1 fixed locus projective}. By taking $\BT$-moving parts in Proposition \ref{prop: relative cech}, we obtain the K-theoretic identity
\[
N_{Q^{\BT}/Q}^{\vir} = \sum_{\alpha \in \Delta(X)} p_\alpha^\ast N_{Q_\alpha^{\BT}/Q_\alpha}^{\vir}
\]
of virtual normal bundles. Thus by Corollary \ref{cor: pot as box product} we have
\begin{align*}
    \int_{[Q^{\BT}]^{\vir}}e^{\BT}(-N_{Q^{\BT}/Q}^{\vir})&= \sum_{|{\bf n}|=n}\prod_{\alpha\in \Delta(X)} \int_{[\Quot_{U_\alpha}(F|_{U_\alpha},n_\alpha)^{\BT_1}]^{\vir}} e^{\BT}(-N_{Q_\alpha^{\BT}/Q_\alpha}^{\vir}).
\end{align*}
In particular, the virtual fundamental class $[\Quot_{U_\alpha}(F|_{U_\alpha},n_\alpha)^{\BT}]^{\vir}$ agrees with the one coming from the critical structure. Moreover, by the virtual localisation formula applied with respect to $(\BC^\ast)^r$, we have
\begin{align*}
    \int_{[\Quot_{U_\alpha}(F|_{U_\alpha},n_\alpha)^{\BT}]^{\vir}} e^{\BT}(-N_{Q_\alpha^{\BT}/Q_\alpha}^{\vir})=\int_{[\Quot_{U_\alpha}(F|_{U_\alpha},n_\alpha)]^{\vir}} 1,
\end{align*}
where the right hand side is defined equivariantly in \S\,\ref{sec: variation vertex}. Finally, by Corollary \ref{cor:independence_on_lambda_w}, we have an identity
\[
\int_{[\Quot_{U_\alpha}(F|_{U_\alpha},n_\alpha)]^{\vir}} 1=\int_{[\Quot_{U_\alpha}(\OO_{U_\alpha}^{\oplus r},n_\alpha)]^{\vir}} 1
\]
of equivariant integrals, where in the right hand side we take $\OO^{\oplus r}_{U_\alpha}$ with the trivial $\BT$-equivariant weights. Therefore we conclude
\begin{align*}
      \DT_F(q) 
      &=  \sum_{n\geq 0} q^n \sum_{|{\bf n}|=n}\prod_{\alpha\in \Delta(X)} \int_{[\Quot_{U_\alpha}(\OO_{U_\alpha}^{\oplus r},n_\alpha)]^{\vir}} 1\\
       &= \prod_{\alpha\in \Delta(X)}  \sum_{n_\alpha\geq 0} q^{n_\alpha} \int_{[\Quot_{U_\alpha}(\OO_{U_\alpha}^{\oplus r},n_\alpha)]^{\vir}} 1\\
       &=\prod_{\alpha\in \Delta(X)} \mathsf M((-1)^rq)^{-r\frac{(s^{\alpha}_1+s^{\alpha}_2)(s^{\alpha}_1+s^{\alpha}_3)(s^{\alpha}_2+s^{\alpha}_3)}{s^{\alpha}_1s^{\alpha}_2s^{\alpha}_3}}.
\end{align*}
We have used Theorem \ref{thm:cohomological} to obtain the last identity, in which we have denoted $s_1^\alpha, s_2^\alpha, s_3^\alpha$ the tangent weights at $p_\alpha$. We conclude taking logarithms:
\begin{align*}
\log \DT_F(q)&=\sum_{\alpha\in \Delta(X)}-r\frac{(s^{\alpha}_1+s^{\alpha}_2)(s^{\alpha}_1+s^{\alpha}_3)(s^{\alpha}_2+s^{\alpha}_3)}{s^{\alpha}_1s^{\alpha}_2s^{\alpha}_3}\log \mathsf M((-1)^rq)\\
&= r\int_{X}c_3(T_X\otimes K_X)\cdot \log \mathsf M((-1)^rq)
\end{align*}
where the prefactor is computed through ordinary Atiyah--Bott localisation.
\end{proof}          
We have thus proved Conjecture 3.5 in \cite{Virtual_Quot} in the toric case. The general case is still open and will be investigated in future work.
\subsection{Conjecture: two obstruction theories are the same}

We close this subsection with a couple of conjectures relating the different obstruction theories appeared in the previous section.

\begin{conjecture}\label{conj:pot_restricted}
Let $\BE$ be the perfect obstruction theory \eqref{global_pot}. Then its restriction along the open subscheme $\iota_{n,\alpha}\colon \Quot_{U_\alpha}(F|_{U_\alpha},n)\into \Quot_X(F,n)$ agrees, as a symmetric perfect obstruction theory, with the critical obstruction theory $\BE_{\crit}$ of Proposition \ref{prop:SPOT_A^3}.
\end{conjecture}

One can also ask whether $\iota_{n,\alpha}^\ast \BE$ and $\BE_{\crit}$ are $\BT$-\emph{equivariantly} isomorphic over the cotangent complex of $\Quot_{U_\alpha}(F|_{U_\alpha},n)$. This is of course stronger than the statement of Proposition \ref{restriction of class in K theory}.

\smallbreak
A similar conjecture (essentially the rank $1$ specialisation of Conjecture \ref{conj:pot_restricted}) can be stated for the moduli space $\Hilb^n(\BA^3) = \Quot_{\BA^3}(\OO,n)$, without reference to a compactification $\BA^3 \subset X$. The Hilbert scheme of points has two symmetric perfect obstruction theories: the critical obstruction theory $\BE_{\crit}$ (Proposition \ref{prop:SPOT_A^3}) and the one coming from moduli of ideal sheaves: if $\mathsf p\colon \BA^3 \times \Hilb^n(\BA^3) \to \Hilb^n(\BA^3)$ is the projection and $\mathfrak I$ is the universal ideal sheaf, one has the obstruction theory
\[
\RR \mathsf p_\ast \RRlHom(\mathfrak I,\mathfrak I)_0 [2] \to \BL_{\Hilb^n(\BA^3)}
\]
obtained from the Atiyah class $\At_{\mathfrak I}$ the way we sketched in \eqref{Atiyah_Journey}.

\begin{conjecture}
There is an isomorphism of perfect obstruction theories
\[
\begin{tikzcd}
\BE_{\crit} \arrow{rr}{\sim}\arrow{dr} & & \RR \mathsf p_\ast \RRlHom(\mathfrak I,\mathfrak I)_0 [2]\arrow{dl} \\
& \BL_{\Hilb^n(\BA^3)} &
\end{tikzcd}
\]
on the Hilbert scheme of points $\Hilb^n(\BA^3)$.
\end{conjecture}

\bibliographystyle{amsplain-nodash}
\bibliography{bib}

\end{document}

%% file: planepartition.tex
\newcounter{x}
\newcounter{y}
\newcounter{z}

\newcommand\xaxis{210}
\newcommand\yaxis{-30}
\newcommand\zaxis{90}

\newcommand\topside[3]{
  \fill[fill=yellow, draw=black,shift={(\xaxis:#1)},shift={(\yaxis:#2)},
  shift={(\zaxis:#3)}] (0,0) -- (30:1) -- (0,1) --(150:1)--(0,0);
}

\newcommand\leftside[3]{
  \fill[fill=red, draw=black,shift={(\xaxis:#1)},shift={(\yaxis:#2)},
  shift={(\zaxis:#3)}] (0,0) -- (0,-1) -- (210:1) --(150:1)--(0,0);
}

\newcommand\rightside[3]{
  \fill[fill=blue, draw=black,shift={(\xaxis:#1)},shift={(\yaxis:#2)},
  shift={(\zaxis:#3)}] (0,0) -- (30:1) -- (-30:1) --(0,-1)--(0,0);
}

\newcommand\cube[3]{
  \topside{#1}{#2}{#3} \leftside{#1}{#2}{#3} \rightside{#1}{#2}{#3}
}

\newcommand\planepartition[1]{
 \setcounter{x}{-1}
  \foreach \a in {#1} {
    \addtocounter{x}{1}
    \setcounter{y}{-1}
    \foreach \b in \a {
      \addtocounter{y}{1}
      \setcounter{z}{-1}
      \foreach \c in {1,...,\b} {
        \addtocounter{z}{1}
        \cube{\value{x}}{\value{y}}{\value{z}}
      }
    }
  }
}

%% file: macros.tex
\usepackage[colorinlistoftodos]{todonotes}%textwidth=35mm

\definecolor{antiquewhite}{rgb}{0.98, 0.92, 0.84}
\definecolor{buff}{rgb}{0.94, 0.86, 0.51}
\definecolor{palecopper}{rgb}{0.85, 0.54, 0.4}
\definecolor{fluorescentyellow}{rgb}{0.8, 1.0, 0.0}
\definecolor{bole}{rgb}{0.47, 0.27, 0.23}

\usepackage{amsmath,float}
\usetikzlibrary{decorations.markings}
\DeclareMathOperator{\Res}{Res}
\DeclareMathOperator{\Cyl}{Cyl}
\DeclareMathOperator{\Ind}{Ind}
\DeclareMathOperator{\Ell}{Ell}
\DeclareMathOperator{\ELL}{\mathcal E\ell\ell}

\definecolor{britishracinggreen}{rgb}{0.0, 0.26, 0.15}
\definecolor{cobalt}{rgb}{0.0, 0.28, 0.67}
\DeclareSymbolFont{usualmathcal}{OMS}{cmsy}{m}{n}
\DeclareSymbolFontAlphabet{\mathcal}{usualmathcal}

\newcommand{\TT}{\mathbf{T}}

\newcommand{\BA}{{\mathbb{A}}}

\newcommand{\BC}{{\mathbb{C}}}

\newcommand{\BE}{{\mathbb{E}}}

\newcommand{\BH}{{\mathbb{H}}}

\newcommand{\BL}{{\mathbb{L}}}

\newcommand{\BN}{{\mathbb{N}}}

\newcommand{\BP}{{\mathbb{P}}}
\newcommand{\BQ}{{\mathbb{Q}}}

\newcommand{\BT}{{\mathbb{T}}}

\newcommand{\BZ}{{\mathbb{Z}}}

\newcommand{\CA}{{\mathcal A}}

\newcommand{\CM}{{\mathcal M}}

\newcommand{\simto}{\,\widetilde{\to}\,}

\newcommand{\pt}{{\mathsf{pt}}}
\newcommand{\ch}{{\mathrm{ch}}}
\newcommand{\td}{{\mathrm{td}}}
\newcommand{\fix}{\mathsf{fix}}
\newcommand{\mov}{\mathsf{mov}}

\DeclareMathOperator{\Hilb}{Hilb}

\DeclareMathOperator{\Quot}{Quot}

\DeclareMathOperator{\coker}{coker}

\DeclareMathOperator{\Sym}{Sym}

\DeclareMathOperator{\mot}{mot}
\DeclareMathOperator{\coh}{coh}
\DeclareMathOperator{\crr}{cr}

\DeclareMathOperator{\Coh}{Coh}

\DeclareMathOperator{\Tr}{Tr}

\DeclareMathOperator{\vd}{vd}

\DeclareMathOperator{\Ob}{Ob}
\DeclareMathOperator{\Obs}{Obs}

\DeclareMathOperator{\vir}{\mathrm{vir}}

\DeclareMathOperator{\Exp}{Exp}
\DeclareMathOperator{\Var}{Var}
\DeclareMathOperator{\KK}{K}

\DeclareMathOperator{\Rep}{Rep}
\DeclareMathOperator{\GL}{GL}

\DeclareMathOperator{\rk}{rk}

\DeclareMathOperator{\tr}{tr}

\DeclareMathOperator{\NCQuot}{Quot}

\newcommand{\derived}{\mathbf{D}}
\newcommand{\dd}{\mathrm{d}}

\newcommand{\crit}{\operatorname{crit}}

\newcommand*{\defeq}{\mathrel{\vcenter{\baselineskip0.5ex \lineskiplimit0pt
                     \hbox{\scriptsize.}\hbox{\scriptsize.}}}%
                     =}

\newcommand{\into}{\hookrightarrow}
\newcommand{\onto}{\twoheadrightarrow}

\DeclareFontFamily{OT1}{rsfs}{}
\DeclareFontShape{OT1}{rsfs}{n}{it}{<-> rsfs10}{}
\DeclareMathAlphabet{\curly}{OT1}{rsfs}{n}{it}

\newcommand\Ext{\operatorname{Ext}}
\newcommand\Hom{\operatorname{Hom}}
\newcommand\End{\operatorname{End}}

\DeclareMathOperator{\lHom}{\mathscr{H}\kern-0.3em\mathit{om}}
\DeclareMathOperator{\RRlHom}{\mathbf{R}\kern-0.025em\mathscr{H}\kern-0.3em\mathit{om}}

\DeclareMathOperator{\lExt}{{\mathscr{E}\kern-0.2em\mathit{xt}}}
\DeclareMathOperator{\Def}{Def}

\DeclareMathOperator{\At}{At}

\newcommand{\RR}{\mathbf R}

\newcommand\id{\operatorname{id}}
\newcommand\A{\mathsf{A}}

\newcommand{\Pic}{\mathop{\rm Pic}\nolimits}

\newcommand{\DT}{\mathsf{DT}}

\newcommand{\OO}{\mathscr O}

\makeatletter
\newenvironment{proofof}[1]{\par
  \pushQED{\qed}%
  \normalfont \topsep6\p@\@plus6\p@\relax
  \trivlist
  \item[\hskip3\labelsep
        \itshape
    Proof of #1\@addpunct{.}]\ignorespaces
}{%
  \popQED\endtrivlist\@endpefalse
}
\makeatother

\usepackage[all]{xy}
\usepackage{tikz}
\usepackage{tikz-cd}
\usepackage{adjustbox}
\usepackage{rotating}
\usepackage{comment}
\newcommand*{\isoarrow}[1]{\arrow[#1,"\rotatebox{90}{\(\sim\)}"
]}
\usetikzlibrary{matrix,shapes,intersections,arrows,decorations.pathmorphing}
\tikzset{commutative diagrams/arrow style=math font}
\tikzset{commutative diagrams/.cd,
mysymbol/.style={start anchor=center,end anchor=center,draw=none}}

\tikzset{
shift up/.style={
to path={([yshift=#1]\tikztostart.east) -- ([yshift=#1]\tikztotarget.west) \tikztonodes}
}
}

\theoremstyle{definition}

\newtheorem*{lemma*}{Lemma}
\newtheorem*{theorem*}{Theorem}
\newtheorem*{example*}{Example}
\newtheorem*{fact*}{Fact}
\newtheorem*{notation*}{Notation}
\newtheorem*{definition*}{Definition}
\newtheorem*{prop*}{Proposition}
\newtheorem*{remark*}{Remark}
\newtheorem*{corollary*}{Corollary}

\newtheorem*{conventions*}{Conventions}

\newtheorem{definition}{Definition}[section]

\newtheorem{example}[definition]{Example}

\newtheorem{notation}[definition]{Notation}
\newtheorem{remark}[definition]{Remark}
\newtheorem{conjecture}[definition]{Conjecture}

\newtheoremstyle{thm} % <name> % (ambienti con dimostrazione)
        {3mm}% <Space above>
        {3mm}% <Space below>
        {\slshape}% <Body font> % 
        {0mm}% <Indent amount>
        {\bfseries}% <Theorem head font>
        {.}% <Punctuation after theorem head>
        {1mm}% <Space after theorem head>
        {}% <Theorem head spec (can be left empty, meaning 'normal')> 
\theoremstyle{thm}
\newtheorem{theorem}[definition]{Theorem}
\newtheorem{corollary}[definition]{Corollary}
\newtheorem{lemma}[definition]{Lemma}
\newtheorem{prop}[definition]{Proposition}
\newtheorem{thm}{Theorem}

\newtheoremstyle{ex} % <name> % (ambienti con dimostrazione)
        {3mm}% <Space above>
        {3mm}% <Space below>
        {}% <Body font> % \slshape
        {0mm}% <Indent amount>
        {\scshape}% <Theorem head font>
        {.}% <Punctuation after theorem head>
        {1mm}% <Space after theorem head>
        {}% <Theorem head spec (can be left empty, meaning 'normal')> 
\theoremstyle{ex}

\newtheoremstyle{sol} % <name> % (ambienti con dimostrazione)
        {3mm}% <Space above>
        {3mm}% <Space below>
        {}% <Body font> % 
        {0mm}% <Indent amount>
        {\scshape}% <Theorem head font>
        {.}% <Punctuation after theorem head>
        {1mm}% <Space after theorem head>
        {}% <Theorem head spec (can be left empty, meaning 'normal')> 
\theoremstyle{sol}